\documentclass[a4paper]{article}
\usepackage{geometry}
\usepackage{graphicx,url}
\usepackage[margin=10pt, font=small, labelfont=bf]{caption}

\usepackage[english]{babel}   
\usepackage[T1]{fontenc}
\usepackage[utf8]{inputenc}  

\usepackage{tikz}
\usepackage{verbatim}
\usepackage{listings}
\usepackage{xcolor}
\usepackage{amsmath,amssymb,mathrsfs}
\usepackage{mathtools}
\usepackage{amsthm}
\usepackage{frontespizio}
\usepackage{pictexwd, dcpic}
\usepackage{epigraph}
\usepackage{enumerate}
\usepackage{fancyhdr}
\usepackage{textcomp}
\usepackage{cancel}
\usepackage{bbm}
\usepackage{authblk}
\theoremstyle{plain}
\newtheorem{thm}{Theorem}[subsection]
\newtheorem{prop}[thm]{Proposition}
\newtheorem{cor}[thm]{Corollary}
\newtheorem{lem}[thm]{Lemma}
\theoremstyle{definition}
\newtheorem{defn}[thm]{Definition}
\newtheorem{rem} [thm] {Remark}

\title{The $L^p$-$L^q$ maximal regularity for the Beris-Edward model in the half-space }
\date{}
\author[1]{Daniele Barbera \thanks{daniele.barbera96@gmail.com}}
\author[2,3]{Miho Murata \thanks{murata.miho@shizuoka.ac.jp}}
\affil[1]{\small{Department of Mathematics, Pisa University, Largo B. Pontecorvo 5, 56100 Pisa, Italy}}
\affil[2]{\small{Department of Mathematical and System Engineering, Faculty of Engineering, Shizuoka University, 3-5-1, Johoku, Naka-ku, Hamamatsu-shi, Shizuoka, 432-8561, Japan}}
\affil[3]{\small{Mathematical Institute, Tohoku University, Aoba, Sendai 980-8578, Japan}}

\begin{document}
\maketitle

\begin{abstract}
In this paper, we consider the model describing viscous incompressible liquid crystal flows, 
which is called the Beris-Edwards model in the half-space.
This model is a coupled system by the Navier-Stokes equations with the evolution equation of the director fields $Q$.
The purpose of this paper is to prove the linearized problem has a unique solution satisfying the maximal $L^p$ -$L^q$ regularity estimates, which is essential for the study of quasi-linear parabolic or parabolic-hyperbolic equations.
Our method relies on the $\mathcal R$-boundedness of the solution operator families to the resolvent problem in order to apply operator-valued Fourier multiplier theorems.
Consequently, we also have the local well-posedness for the Beris-Edwards model with small initial data.
\end{abstract}


\section{Introduction}

In this paper we will study the Beris-Edward model for nematic liquid crystals:
\begin{equation}\label{BE.sys.0}
    \left\{\begin{array}{ll}
       (\partial_t-\Delta)u+\nabla p+\beta {\rm Div}(\Delta-a)Q= f(u,Q)  & \text{in}\:\:(0,T)\times \mathbb{R}^N_+ \\
       (\partial_t-\Delta+a)Q-\beta D(u)=G(u,Q)  & \text{in}\:\:(0,T)\times\mathbb{R}^N_+ \\
       {\rm div} u=0 & \text{in}\:\:(0,T)\times\mathbb{R}^N_+ \\
       u=h,\quad D_NQ=H & \text{on}\:\:(0,T)\times\mathbb{R}^N_0 \\
       u(0)=u_0,\quad Q(0)=Q_0 & \text{in}\:\:\mathbb{R}^N_+,
    \end{array}\right.
\end{equation}
with
$$ ({\rm Div} A)_k=\sum_{j=1}^N\partial_j A_{k,j}\quad \forall A\colon\mathbb{R}^N\to\mathbb{R}^{N^2},\quad k=1,\ldots, N. $$
Liquid crystals are a state of matter intermediate between the solid state and the liquid state: such substances flow like liquids but they are strongly anisotropic. As the name suggested, the model was introduced by Beris and Edward in \cite{BE94}. Here $u\colon(0,T)\times \mathbb{R}^N_+\to \mathbb{R}^N$ and $p\colon(0,T)\times \mathbb{R}^N_+\to\mathbb{R}$ for $T>0$ are respectively the velocity field of the particles and the pressure of the material, while $Q\colon(0,T)\times \mathbb{R}^N_+\to S_0(N,\mathbb{R})$ was introduced by \cite{GP93} in order to measure the anisotropy of the substance, where
$$ S_0(N,\mathbb{R})\coloneqq \left\{A\in\mathbb{R}^{N^2}\:\Big|\: A^T=A,\quad {\rm tr}(A)=0\right\}, $$
where $A^T$ and ${\rm tr}A$ are respectively the transpose and the trace of a matrix $A$
$$ \left(A^T\right)_{j,k}=A_{k,j}\quad (j,k=1,\ldots, N),\quad {\rm tr}A=\sum_{j=1}^NA_{j,j} $$
and where $\mathbb{R}^N_+$ and $\mathbb{R}^N_0$ are respectively the half-space and its boundary, i.e.
$$ \mathbb{R}^N_+\coloneqq \left\{(x^\prime,x_N)\in\mathbb{R}^{N-1}\times \mathbb{R}\:\Big|\:x_N>0\right\}, $$
$$ \mathbb{R}^N_0\coloneqq \left\{(x^\prime,x_N)\in\mathbb{R}^{N-1}\times \mathbb{R}\:\Big|\: x_N=0\right\}. $$
Moreover, $\xi,a,b,c\in\mathbb{R}$, $\beta=\frac{2\xi}{N}$, $h\colon(0,T)\times \mathbb{R}^N_+\to\mathbb{R}^N$, $H\colon(0,T)\times \mathbb{R}^N_+\to \mathbb{R}^{N^2}$, 
$$ f(u,Q)=-(u\cdot \nabla) u + {\rm Div}\left[2\xi \mathbb{H}\colon Q\left(Q+\frac{Id}{N}\right) - (\xi+1)\mathbb{H}Q+(1-\xi)Q\mathbb{H}-\nabla Q\odot\nabla Q\right]-\beta {\rm Div}\mathcal{L}[\mathcal{F}(Q)],  $$
$$ G(u,Q)=-(u\cdot\nabla)Q+\xi(D(u)Q+QD(u))+W(u)Q-QW(u)-2\xi\left(Q+\frac{Id}{N}\right)Q\colon \nabla u+\mathcal{L}[\mathcal{F}(Q)],  $$
$$ W(u)=\frac{1}{2}\left(\nabla u-\nabla^Tu\right),\:\:  D(u)=\frac{1}{2}\left(\nabla u+\nabla^Tu\right), \quad \mathcal{F}(Q)=bQ^2-c|Q|^2Q,$$
$$ [\nabla Q\odot\nabla Q]_{jk}=\sum_{\alpha,\beta=1}^N\partial_j Q_{\alpha\beta}\partial_k Q_{\alpha\beta}\quad j,k=1,\ldots, N,  $$
$$ \mathbb{H}=\Delta Q-aQ+b\mathcal{L}[Q^2]-c|Q|^2Q,   $$
where $Id$ is the identity matrix of $\mathbb R^{N^2}$,  $|A|$ and $A\colon B$ of two symmetric matrices $A,B\in\mathbb{R}^{N^2}$ are respectively the Frobenius norm and his associated scalar product
$$ A\colon B\coloneqq {\rm tr}\left(B^TA\right)=\sum_{i,j=1}^NB_{ji}A_{ji}, $$
$$ |A|\coloneqq \sqrt{A\colon A}=\sqrt{\sum_{i,j=1}^N A_{ij}^2},  $$
and where
$$ \mathcal{L}[A]=A-{\rm tr}(A)\frac{Id}{N}. $$

\vspace{1mm}

The Beris-Edwards model was mathematically studied by several authors. Concerning the case $\xi=0$, the first result was obtained by Paicu and Zarnescu \cite{PZ12}. They proved the existence of global weak solutions in $\mathbb R^N$ with $N=2, 3$ as well as weak-strong uniqueness for $N=2$. An improved result of \cite{PZ12} in $\mathbb R^2$ was established in \cite{DA17}. Huang and Ding \cite{HD15} proved the existence of global weak solutions with a more general energy functional in $\mathbb R^3$. Abels, Dolzmann, and Liu \cite{ADL16} proved that the classical Beris-Edwards model, fluid viscosity depends on the $Q$-tensor, has a unique local solution in a bounded domain with Dirichlet boundary conditions. The global well-posedness was proved by Luo, Li, and Zhao \cite{LLX19} in a bounded with Dirichlet boundary conditions under the assumption viscosity is sufficiently large. Xiao \cite{X17} proved the global well-posedness in a bounded domain. The author constructed a strong solution in the $L^p$-$L^q$ maximal regularity class.

On the other hand, concerning the model with general parameter $\xi$, Abels, Dolzmann, and Liu \cite{ADL14} showed the unique existence of a strong local solution and global weak solutions with higher regularity in time in the case of inhomogeneous mixed Dirichlet/Neumann boundary conditions in a bounded domain. Liu and Wang \cite{LW18} improved the spatial regularity of solutions obtained in \cite{ADL14} and generalized their result to the case of anisotropic elastic energy. The global well-posedness and long-time behavior of the model in the two-dimensional periodic case was investigated by Cavaterra et al. \cite{CR16}.  In \cite{SS19} Schonbek and Shibata proved the global well-posedness and the decay properties in the $L^p$-$L^q$ maximal regularity class for the simplified model, which means that the linear terms $\Delta Q - a Q$ are removed from the first equation of \eqref{BE.sys.0}. Shibata and the second author obtained in \cite{MS22}  the unique existence and the decay properties of a strong global solution 
in the same solution spaces as \cite{SS19} for the Beris-Edwards model.

As far as we know, there is no result relating to the well-posedness for boundary value problems in unbounded domains even if $\xi=0$.
In this paper, we prove the $L^p$-$L^q$ maximal regularity for the linearized system in $\mathbb R^N_+$: 
$$ \left\{\begin{array}{ll}
    \partial_tu-\Delta u+\nabla p+\beta {\rm Div}(\Delta-a)Q=f & \text{in}\:\:\mathbb{R}_+\times \mathbb{R}^N_+ \\
    \partial_t Q-(\Delta-a)Q-\beta D(u)=G & \text{in}\:\:\mathbb{R}_+\times \mathbb{R}^N_+ \\
    {\rm div} u=0 & \text{in}\:\:\mathbb{R}_+\times \mathbb{R}^N_+ \\
    u=h, \quad D_NQ=H & \text{on}\:\:\mathbb{R}_+\times \mathbb{R}^N_0 \\
    u(0)=u_0,\quad Q(0)=Q_0 & \text{in}\:\:\mathbb{R}^N_+.
\end{array}\right. $$
For this purpose, $\mathcal R$-boundedness of the solution operator families to the resolvent problem is a key issue. Moreover, we prove 
the $L^p$-$L^q$ maximal regularity yields the local well-posedness for the system \eqref{BE.sys.0} with small initial data in $\mathbb R^N_+$.

\subsection{Notations}

In this section, we summarize the symbols and functional spaces used through the paper. 

\vspace{2mm}

Let $\theta\in\left(0,\frac{\pi}{2}\right)$ and $r>0$, then we can define
$$ \Sigma_\theta\coloneqq \{z\in\mathbb{C}\setminus\{0\}\mid |Arg(z)|<\pi- \theta\} $$
and
$$ \Sigma_{\theta,r}\coloneqq \{z\in\Sigma_\theta\mid |z|> r\}. $$

\vspace{1mm}

We will denote $\mathbb{N}_0=\mathbb{N}\cup\{0\}$, $\mathbb{R}_+=(0,+\infty)$ and $\mathbb{R}_-=(-\infty,0)$. For any $q\in(1,\infty)$ we denote the dual exponent $q^\prime=\frac{q}{q-1}$. For any multi-index $\alpha\in\mathbb{N}^N_0$ we write
$$ |\alpha|=\alpha_1+\cdots+\alpha_N, $$
$$ D^\alpha= \partial^{\alpha_1}_{x_1}\cdots \partial^{\alpha_N}_{x_N}. $$
For any $k\in\mathbb{N}_0$, for any $\Omega\subseteq\mathbb{R}^N$ open set and for any function $f\colon \Omega\to\mathbb{R}$, $g\colon\Omega\to \mathbb{R}^N$ and $A\colon\Omega\to \mathbb R^{N^2}$ we denote
$$ \nabla^kf=(D^\alpha f\mid |\alpha|=k), \quad \nabla^kg=(D^\alpha g_j\mid |\alpha|=k,\quad j=1,\ldots, N), $$
$$ \nabla^k A=(D^\alpha A_{\ell,j}\mid |\alpha|=k,\quad \ell,j=1,\ldots, N). $$
We will also denote $C^\infty(\Omega)$ the space of infinitely differentiable functions in $\Omega$ and $C^\infty_c(\Omega)$ the $C^\infty(\Omega)$-functions with compact support.

Let $\mathcal{F}$ and $\mathcal{F}^{-1}$ denote the Fourier transform and the Fourier inverse transform, respectively, which are defined by setting
$$ \widehat{f}(\tau)=\mathcal{F}[f](\tau)=\int_{\mathbb R}e^{-it\tau}f(t)dt, \quad \mathcal{F}^{-1}[f](t)=\frac{1}{2\pi}\int_{\mathbb R}e^{it\tau}f(\tau)d\tau. $$

Let $X,Y$ be two Banach spaces, then we denote with $\mathcal{L}(X;Y)$ the linear bounded operators between $X$ and $Y$. We will write $\mathcal{L}(X)$ when $Y=X$. 

Let $p,q\in(1,\infty)$, $m\in\mathbb{N}_0$ and $s\in\mathbb R$, then we will denote $L^q(\Omega)$, $W^{m,q}(\Omega)$ and $B^s_{q,p}(\Omega)$ respectively the Lebesgue, the Sobolev and the Besov spaces and we will denote $\|\cdot\|_{L^q(\Omega)}$, $\|\cdot\|_{W^{m,q}(\Omega)}$ and $\|\cdot\|_{B^s_{q,p}(\Omega)}$ their norms. We will denote $H^m(\Omega)\coloneqq W^{m,2}(\Omega)$.

Let $s\in(0,1)$ and $p\in(1,\infty)$, then we recall the definition of
$$ H^s_p(\mathbb{R})\coloneqq \left\{v\in L^p(\mathbb{R})\:\Big|\: \mathcal{F}^{-1}[(1+|\tau|^2)^{s/2}\mathcal{F}[v]]\in L^p(\mathbb{R})\right\} $$
with the norm
$$ \|v\|_{H^s_p(\mathbb{R})}\coloneqq \left\|\mathcal{F}^{-1}\left[(1+|\tau|^2)^{s/2}\widehat{u}\right]\right\|_{L^p(\mathbb{R})}. $$

Let now $A\subseteq\mathbb{R}$ open, then we define
$$ H^s_p(A)\coloneqq\left\{v\in L^p(A)\mid \exists \:\widetilde{v}\in H^s_p(\mathbb{R})\:\:\text{such that}\:\: \widetilde{v}_{|A}=v\right\}, $$
with the norm
$$ \|v\|_{H^s_p(A)}\coloneqq \inf_{\widetilde{v}_{|A}=v}\|\widetilde{v}\|_{H^s_p(\mathbb{R})}. $$
Moreover, let $X$ be a Banach space, then we denote $L^p((a,b);X)$, $W^{m,p}((a,b);X)$ and $H^s((a,b);X)$ the previous spaces function for $X$-valued functions for any $(a,b)\subseteq\mathbb{R}$.

Let $N\in\mathbb N$, $\Omega\subseteq \mathbb R^N$ open set, $0\le k<s<m$ and $p,q\in[1,\infty]$, then we recall the definition of the Besov spaces
$$ B^s_{p,q}\left(\Omega\right)\coloneqq \left(W^{k,p}\left(\Omega\right), W^{m,p}\left(\Omega\right)\right)_{\theta,q}, $$
with $s=(1-\theta)k+\theta m$.

Let $q\in(1,\infty)$ then we denote
$$ J_q\left(\mathbb{R}^N_+\right)\coloneqq \left\{f\in L^q\left(\mathbb{R}^N_+;\mathbb{R}^N\right)\:\Big|\:\left<f,\nabla \varphi\right>=0,\quad \forall \varphi\in\widehat{H}^1_{q^\prime}\left(\mathbb{R}^N_+\right)\right\}, $$
$$ \widehat{H}^1_{q}\left(\mathbb{R}^N_+\right)\coloneqq \left\{\varphi\in L^{q}_{loc}\left(\mathbb{R}^N_+\right)\:\Big|\: \nabla \varphi\in L^{q}\left(\mathbb{R}^N_+;\mathbb{R}^N\right)\right\}. $$

Finally, in the paper we will use $C$ to indicate a constant which depends on the parameters of the problem. In the statements we will use $C(a,b,\ldots)$ to underline the dependence from $a,b,\ldots$, otherwise we will use the symbols 
$$ f(x)\lesssim g(x)\:\Leftrightarrow\: \exists C\:\: \text{s.t.}\:\:f(x)\le Cg(x) $$
$$ f(x)\gtrsim g(x)\:\Leftrightarrow\: \exists C\:\: \text{s.t.}\:\:f(x)\ge Cg(x) $$
$$ f(x)\simeq g(x)\:\Leftrightarrow\: \exists C\:\: \text{s.t.}\:\:f(x)= Cg(x). $$

\subsection{$\mathcal{R}$-boundedness and main results}

The main purpose of the paper is to prove the $L^p$-$L^q$ maximal regularity for the linearized system
\begin{equation}\label{lin.evo.sys.}
    \left\{\begin{array}{ll}
    \partial_tu-\Delta u+\nabla p+\beta {\rm Div}(\Delta-a)Q=f & \text{in}\:\:\mathbb{R}_+\times \mathbb{R}^N_+ \\
    \partial_t Q-(\Delta-a)Q-\beta D(u)=G & \text{in}\:\:\mathbb{R}_+\times \mathbb{R}^N_+ \\
    {\rm div} u=0 & \text{in}\:\:\mathbb{R}_+\times \mathbb{R}^N_+ \\
    u=h, \quad D_NQ=H & \text{on}\:\:\mathbb{R}_+\times \mathbb{R}^N_0 \\
    u(0)=u_0,\quad Q(0)=Q_0 & \text{in}\:\:\mathbb{R}^N_+,
\end{array}\right.
\end{equation}
with 
$$ h,f\colon\mathbb{R}_+\times \mathbb{R}^N_+\to \mathbb{R}^N,\quad G,H\colon\mathbb{R}_+\times \mathbb{R}^N_+\to S_0(N,\mathbb{R}), $$ $$ u_0\colon\mathbb{R}^N_+\to\mathbb{R}^N,\quad Q_0\colon\mathbb{R}^N_+\to S_0(N,\mathbb R) $$
in some suitable function spaces. We will start from the study of the resolvent system:
\begin{equation}\label{res.sys.}
    \left\{\begin{array}{ll}
        (\lambda-\Delta)u+\nabla p+\beta {\rm Div}(\Delta-a) Q=f & \text{in}\:\:\mathbb{R}^N_+ \\
        (\lambda+a-\Delta)Q-\beta D(u)=G & \text{in}\:\:\mathbb{R}^N_+ \\
        {\rm div}u=0 & \text{in}\:\:\mathbb{R}^N_+ \\
        u=h,\quad D_NQ=H & \text{on}\:\:\mathbb{R}^N_0.
    \end{array}\right.
\end{equation}
As is done in \cite{S18}, \cite{SS08}, \cite{SS12}, \cite{SS19} and \cite{MS22}, we need to introduce the notion of $\mathcal{R}$-boundedness:
\begin{defn}\label{d.R-bound}\hfill\\
Let $X$ and $Y$ be two Banach spaces, then we say that a family $\mathscr{I}\subseteq\mathcal{L}(X,Y)$ is $\mathcal{R}$-bounded if there is $C>0$ and $p\in[1,\infty)$ such that, for any $m\in\mathbb{N}$, for any $T_j\in\mathscr{I}$, for any $x_j\in X$ with $j=1,\ldots, m$ and for any sequence $\{r_j(z)\}_{j=1}^m$ of independent, symmetric, random $\{-1,1\}$-valued variables on $[0,1]$ it holds
$$ \int_0^1 \left\|\sum_{j=1}^mr_j(z)T_j(x_j)\right\|_Y^pdz\le C\int_0^1\left\|\sum_{j=1}^m r_j(z)x_j\right\|_X^pdz. $$
The minimal $C$ it is called $\mathcal{R}$-bound of $\mathscr{I}$ and it is denoted by $\mathcal{R}(\mathscr{I})$.
\end{defn}


\vspace{2mm}

We will prove in the case $f=G=0$ that the solutions for \eqref{res.sys.} can be written as $(u,p,Q)=\phi_\lambda(h,H), $
with 
$$ \left\{(\tau\partial_\tau)^\ell\phi_\lambda\:\big|\:\lambda\in\Sigma_{\theta,r}\right\}\quad \ell=0,1 $$
$\mathcal{R}$-bounded for any $\theta\in\left(\theta_0,\frac{\pi}{2}\right)$ and $r>0$ for some $\theta_0>0$. In fact, the $\mathcal{R}$-boundedness will be crucial not only for the study of solution for the resolvent system \eqref{res.sys.}, but also for the existence of the linear evolution system \eqref{lin.evo.sys.}.

\vspace{2mm}

We are now ready to state the main result of the paper, that is the $L^p$-$L^q$ maximal regularity result for the linearized system \eqref{lin.evo.sys.}: 
\begin{thm}\label{t.evol.est.}
Let $N\ge 2$, $a>0$, $\beta\in\mathbb{R}$ and $p,q\in(1,+\infty)$ with $\frac{2}{p}+\frac{1}{q}<2$, then there is $\gamma_0>0$ such that for any $\gamma\ge \gamma_0$, for any $f,G,h,H$ with $h_N=0$ on $\mathbb R^N_0$ and 
$$ e^{-\gamma t}f\in L^p\left(\mathbb{R}_+;L^q\left(\mathbb{R}^N_+;\mathbb{R}^N\right)\right), \quad  e^{-\gamma t}G\in L^p\left(\mathbb{R}_+;W^{1,q}\left(\mathbb{R}^N_+;S_0(N,\mathbb R)\right)\right), $$
$$ e^{-\gamma t}h\in \bigcap_{l=0}^2H^{l/2}_{p}\left(\mathbb{R}_+;W^{2-l,q}\left(\mathbb{R}^N_+;\mathbb{R}^{N}\right)\right),\quad e^{-\gamma t}H\in \bigcap_{l=0}^2H^{l/2}_{p}\left(\mathbb{R}_+;W^{2-l,q}\left(\mathbb{R}^N_+;S_0(N,\mathbb R)\right)\right), $$ 
and $u_0\in B_{q,p}^{2(1-1/p)}(\mathbb{R}^N_+;\mathbb{R}^{N})\cap J_q(\mathbb{R}^N_+)$, $Q_0\in B_{q,p}^{3-2/p}(\mathbb{R}^N_+;S_0(N,\mathbb{R}))$ such that
$$ u_0-h(0)=D_NQ_0-H(0)=0 \quad \text{on}\:\:\mathbb{R}^N_0, $$
there is a solution $(u,p,Q)$ for \eqref{lin.evo.sys.}, unique up to additive functions $c(t)$ on the pressure term $p$, with $p(t)\in \widehat{H}^1_{q}(\mathbb{R}^N_+)$ for a.e. $t>0$ and
$$ e^{-\gamma t}u\in \bigcap_{l=0}^2H^{l/2}_{p}\left(\mathbb{R}_+;W^{2-l,q}\left(\mathbb{R}^N_+;\mathbb{R}^{N}\right)\right),\:\:   e^{-\gamma t}Q\in \bigcap_{l=0}^2H^{l/2}_p\left(\mathbb{R}_+;W^{3-l,q}\left(\mathbb{R}^N_+;S_0(N,\mathbb R)\right)\right), $$
$$ e^{-\gamma t}\nabla p\in L^p\left(\mathbb{R}_+;L^q\left(\mathbb{R}^N_+;\mathbb{R}^N\right)\right), $$
with
$$ \sum_{l=0}^2\|e^{-\gamma t}u\|_{H^{l/2}_{p}(\mathbb{R}_+;W^{2-l,q}(\mathbb{R}^N_+))} + \sum_{l=0}^2 \|e^{-\gamma t}Q\|_{H^{l/2}_{p}(\mathbb{R}_+;W^{3-l,q}(\mathbb{R}^N_+))}+ \|e^{-\gamma t}\nabla_x p\|_{L^p(\mathbb{R}_+;L^q(\mathbb{R}^N_+))} \le $$
$$ \le C\left[\sum_{l=0}^2 \| e^{-\gamma t}(h,H)\|_{H^{l/2}_{p}(\mathbb{R}_+;W^{2-k,q}(\mathbb{R}^N_+))}+\|e^{-\gamma t}f\|_{L^p(\mathbb{R}_+;L^q(\mathbb{R}^N_+))} +\|e^{-\gamma t}G\|_{L^p(\mathbb{R}_+;W^{1,q}(\mathbb{R}^N_+)}+ \right. $$
$$ \left.+\|u_0\|_{B_{q,p}^{2(1-1/p)}(\mathbb{R}^N_+)}+\|Q_0\|_{B_{q,p}^{3-2/p}(\mathbb{R}^N_+)}\right], $$
for some $C=C(a,\beta,p,q,N)>0$. 
\end{thm}
The conditions
$$ u_0-h(0)=D_NQ_0-H(0)=0 \quad \text{on}\:\:\mathbb{R}^N_0 $$
are called compatibility conditions, while $h_N=0$ on $\mathbb R^N_0$ follows by the divergence-free condition of $u$. We notice that it is reasonable to take the trace for $u_0$ and $D_NQ_0$ thanks to the condition $\frac{2}{p}+\frac{1}{q}<2$ (see Theorem 6.6.1 of \cite{BL76}). As we will see later in the paper, the resolvent estimate will follow from the proof of Theorem \ref{t.evol.est.} and, in particular, from the $\mathcal{R}$-boundedness of the map $\phi_\lambda$ we introduced before:
\begin{thm}\label{t.res.est.}
Let $N\ge 2$, $a,r>0$, $\beta\in\mathbb{R}$, $\theta\in\left(\theta_0,\frac{\pi}{2}\right)$ with $\tan\theta_0\ge \frac{|\beta|}{\sqrt{2}}$, let $q\in(1,\infty)$, let $f\in L^q(\mathbb{R}^N_+;\mathbb{R}^N)$, $G\in W^{1,q}(\mathbb{R}^N_+;S_0(N,\mathbb{R}))$, $h\in W^{2,q}(\mathbb{R}^N_+;\mathbb{R}^N)$ with $h_N=0$ on $\mathbb R^N_0$ and $H\in W^{2,q}(\mathbb{R}^N_+;S_0(N,\mathbb{R}))$, then for any $\lambda\in\Sigma_{\theta,r}$ there is a solution $(u,p,Q)$ for (\ref{res.sys.}), unique up to additive constants on the pressure term $p$, with $u\in W^{2,q}(\mathbb{R}^N_+;\mathbb{R}^N)$, $p\in \widehat{H}^1_q(\mathbb{R}^N_+)$ and $Q\in W^{3,q}(\mathbb{R}^N_+,S_0(N,\mathbb{R}))$, moreover there is $C=C(a,\beta,\theta,r,q,N)>0$ such that
$$ \left\|\left(|\lambda|u,|\lambda|^\frac{1}{2}\nabla u,D^2u\right)\right\|_{L^q(\mathbb{R}^N_+)}+\|\nabla p\|_{L^q(\mathbb{R}^N_+)}+\left\|\left(|\lambda|^\frac{3}{2}Q, |\lambda|\nabla Q, |\lambda|^\frac{1}{2}D^2Q, D^3Q\right)\right\|_{L^q(\mathbb{R}^N_+)}\le $$
$$ \le C\left[ \|f\|_{L^q(\mathbb{R}^N_+)}+\left\|\left(|\lambda|^\frac{1}{2}G,\nabla G\right)\right\|_{L^q(\mathbb{R}^N_+)}+ \left\|\left(|\lambda|(h,H),|\lambda|^\frac{1}{2}\nabla (h,H),D^2(h,H)\right)\right\|_{L^q(\mathbb{R}^N_+)}\right]. $$
\end{thm}

Later in the paper, as an application of Theorem \ref{t.evol.est.}, we will prove the local well-posedness for the system \eqref{BE.sys.0} with small initial data:
\begin{thm}\label{t.loc.ex.}
Let $N\ge 2$, $\xi,b,c\in\mathbb{R}$, $a>0$, let $p\in(2,\infty)$ and $q\in(N,\infty)$, then we can find $\varepsilon_0>0$ and $T=T(\varepsilon_0)>0$ such that for any $\varepsilon<\varepsilon_0$, for any $h,H$ with $h_N=0$ on $\mathbb{R}^N_0$ such that 
$$ h\in \bigcap_{l=0}^2H^{l/2}_p\left((0,T);W^{2-l,q}\left(\mathbb{R}^N_+;\mathbb{R}^{N}\right)\right),\quad H\in \bigcap_{l=0}^2H_p^{l/2}\left((0,T);W^{2-l,q}\left(\mathbb{R}^N_+;S_0(N,\mathbb{R})\right)\right), $$ 
for any $u_0\in J_q(\mathbb{R}^N_+)\cap B_{q,p}^{2(1-1/p)}(\mathbb{R}^N_+;\mathbb{R}^{N})$, $Q_0\in B_{q,p}^{3-2/p}(\mathbb{R}^N_+;S_0(N,\mathbb{R}))$ and
$$ u_0-h(0)=D_NQ_0-H(0)=0 \quad \text{on}\:\:\mathbb{R}^N_0, $$
with
$$ \|u_0\|_{B^{2(1-1/p)}_{q,p}(\mathbb{R}^N_+)}+\|Q_0\|_{B^{3-2/p}_{q,p}(\mathbb{R}^N_+)}\le \varepsilon, $$
we can find a solution $(u,p,Q)$ for (\ref{BE.sys.0}), unique up to additive functions $c(t)$ on the pressure term $p$, with $p(t)\in \widehat{H}^1_{q}(\mathbb{R}^N_+)$ for a.e. $t\in(0,T)$ and
$$ u\in \bigcap_{l=0}^2H^{l/2}_p\left((0,T);W^{2-l,q}\left(\mathbb{R}^N_+;\mathbb{R}^{N}\right)\right),\:\: \nabla p\in L^p\left((0,T);L^q\left(\mathbb{R}^N_+;\mathbb{R}^N\right)\right), $$
$$ Q\in \bigcap_{l=0}^2H^{l/2}_p\left(\mathbb{R}_+;W^{3-l,q}\left(\mathbb{R}^N_+;S_0(N,\mathbb{R})\right)\right), $$
moreover we can find $C=C(\xi,a,b,c,p,q,N)>0$ such that
$$ \|u\|_{H^{1}_p((0,T);L^q(\mathbb{R}^N_+))}+\|u\|_{L^p((0,T);W^{2,q}(\mathbb{R}^N_+)}+ $$
$$ + \|Q\|_{H^{1}_p((0,T);W^{1,q}(\mathbb{R}^N_+))} + \|Q\|_{L^p((0,T);W^{3,q}(\mathbb{R}^N_+))}+\|\nabla p\|_{L^p((0,T);L^q(\mathbb{R}^N))}\le  $$
$$ \le C\left[ \sum_{l=0}^2 \|h\|_{H^{l/2}_p((0,T);W^{2-l,q}(\mathbb{R}^N_+))}+\|H\|_{H^{l/2}_p((0,T);W^{2-l,q}(\mathbb{R}^N_+))}+ \|u_0\|_{B^{2(1-1/p)}_{q,p}(\mathbb{R}^N_+)}+\|Q_0\|_{B^{3-2/p}_{q,p}(\mathbb{R}^N_+)}\right]. $$
\end{thm}
The paper is organized as follows: in Section 2 we study the existence of a solution for the system \eqref{lin.evo.sys.} composed with the partial Fourier Transformation in $\mathbb{R}^N_+$; in Section 3 we prove the $\mathcal{R}$-boundedness of the solution for \eqref{lin.evo.sys.} and the consequent proof of Theorem \ref{t.res.est.}; in Section 4 we prove Theorem \ref{t.evol.est.} and finally in Section 5 we prove Theorem \ref{t.loc.ex.}.

\vspace{2mm}

\section{Existence and uniqueness of the solution for the Fourier System}
\subsection{Solution formula for the Fourier System}

As we said in the introduction, we firstly focus on the resolvent system \eqref{res.sys.} with $f=G=0$, that is
\begin{equation}\label{res.sys.2}
    \left\{\begin{array}{ll}
        (\lambda-\Delta)u+\nabla p+\beta {\rm Div}(\Delta-a) Q=0 & \text{in}\:\:\mathbb{R}^N_+ \\
        (\lambda+a-\Delta)Q-\beta D(u)=0 & \text{in}\:\:\mathbb{R}^N_+ \\
        {\rm div}u=0 & \text{in}\:\:\mathbb{R}^N_+ \\
        u=h,\quad D_NQ=H & \text{on}\:\:\mathbb{R}^N_0.
    \end{array}\right.
\end{equation}
Let us take the partial Fourier Transformation
$$ \widehat{v}(\xi^\prime,x_N)\coloneqq \int_{\mathbb{R}^{N-1}}e^{-ix^\prime\cdot \xi^\prime}v(x^\prime,x_N)dx^\prime\quad \xi^\prime\in\mathbb{R}^{N-1}\quad x_N>0 $$
of the system (\ref{res.sys.2}):
\begin{equation}\label{F.u.s.}
\left\{
\begin{aligned}
&(\lambda+|\xi^\prime|^2-D_N^2)\widehat{u}_j + i\xi_j \widehat{p} +\beta \sum^{N-1}_{k=1}i\xi_k (-|\xi^\prime|^2+D_N^2-a)\widehat{Q}_{jk}  \\
&\enskip + \beta D_N (-|\xi^\prime|^2+D_N^2-a)\widehat{Q}_{jN}=0 \quad (j =1, \dots, N-1),\\
&(\lambda+|\xi^\prime|^2-D_N^2)\widehat{u}_N + D_N \widehat{p} +\beta \sum^{N-1}_{k=1}i\xi_k (-|\xi^\prime|^2+D_N^2-a)\widehat{Q}_{Nk} \\
&\enskip + \beta D_N (-|\xi^\prime|^2+D_N^2-a)\widehat{Q}_{NN}=0,
\end{aligned}
\right.
\end{equation}
\begin{equation}\label{F.div.e.}
\sum^{N-1}_{k=1}i\xi_k \widehat{u}_k + D_N \widehat{u}_N=0,
\end{equation}
\begin{equation}\label{F.Q.s.}
   \left\{
\begin{aligned}
&(\lambda +|\xi^\prime|^2-D_N^2+a)\widehat{Q}_{jk} - \frac{\beta}{2} (i\xi_j \widehat{u}_k+i\xi_k \widehat{u}_j) =0  \quad (j,k =1, \dots, N-1), \\
&(\lambda +|\xi^\prime|^2-D_N^2+a)\widehat{Q}_{jN} - \frac{\beta}{2} (i\xi_j \widehat{u}_N+D_N \widehat{u}_j) =0 \quad (j =1, \dots, N-1), \\
&(\lambda +|\xi^\prime|^2-D_N^2+a)\widehat{Q}_{NN} - \beta D_N \widehat{u}_N = 0,
\end{aligned}
\right.
\end{equation}
\begin{equation}\label{F.BCs}
    \widehat{u}_k(0)=\widehat{h}_k,\quad D_N\widehat{Q}_{jk}(0)=\widehat{H}_{jk}\quad (j,k=1,\ldots, N).
\end{equation} 
In order to find the solution formula of the systems (\ref{F.u.s.}), (\ref{F.div.e.}) and (\ref{F.Q.s.}), we multiply the first equation of \eqref{F.u.s.} by $i\xi_j$ and sum with respect to $j$:
\begin{equation}\label{r3}
\begin{aligned}
&\sum^{N-1}_{j=1}(\lambda+|\xi^\prime|^2-D_N^2)i\xi_j\widehat{u}_j -|\xi^\prime|^2 \widehat{p} \\
&\enskip +\beta \sum^{N-1}_{j, k=1}i\xi_j i\xi_k (-|\xi^\prime|^2+D_N^2-a)\widehat{Q}_{jk} 
+ \beta \sum^{N-1}_{j=1}i\xi_j D_N (-|\xi^\prime|^2+D_N^2-a)\widehat{Q}_{jN}=0.
\end{aligned}
\end{equation}
Applying $D_N$ to the second equation of \eqref{F.u.s.},
we have
\begin{equation}\label{r4}
\begin{aligned}
&(\lambda+|\xi^\prime|^2-D_N^2)D_N\widehat{u}_N + D_N^2 \widehat{p} \\
&\enskip +\beta \sum^{N-1}_{k=1}i\xi_k (-|\xi^\prime|^2+D_N^2-a)D_N\widehat{Q}_{Nk} 
+ \beta D_N^2 (-|\xi^\prime|^2+D_N^2-a)\widehat{Q}_{NN}=0.
\end{aligned}
\end{equation}
The summation of \eqref{r3} and \eqref{r4} gives us
\begin{equation}\label{r5}
\begin{aligned}
(-|\xi^\prime|^2 + D_N^2)\widehat{p}
&+\beta \sum^{N-1}_{j, k=1}i\xi_j i\xi_k (-|\xi^\prime|^2+D_N^2-a)\widehat{Q}_{jk} \\
&+ 2\beta \sum^{N-1}_{j=1}i\xi_j D_N (-|\xi^\prime|^2+D_N^2-a)\widehat{Q}_{jN}\\
&+ \beta D_N^2 (-|\xi^\prime|^2+D_N^2-a)\widehat{Q}_{NN}=0,
\end{aligned}
\end{equation}
where we have used \eqref{F.div.e.} and the fact that $Q$ is a symmetric matrix.
Then applying $(\lambda + |\xi^\prime|^2 - D_N^2 + a)$ to \eqref{r5} and using \eqref{F.div.e.} and \eqref{F.Q.s.},
we get
\begin{equation}\label{p}
(-|\xi^\prime|^2 + D_N^2)(\lambda + |\xi^\prime|^2 - D_N^2 + a) \widehat{p}=0. 
\end{equation}
Thanks to  \eqref{F.div.e.}, \eqref{F.Q.s.}, and \eqref{p}, $\widehat{p}$ and $\widehat{Q}$ can be eliminated from \eqref{F.u.s.}:
\begin{equation}\label{u}
(-|\xi^\prime|^2 + D_N^2)L(D_N) \widehat{u} =0,
\end{equation}
where 
$$ L(t)=(\lambda+|\xi^\prime|^2-t^2)(\lambda+a+|\xi^\prime|^2-t^2)+\frac{\beta^2}{2}\left((t^2-|\xi^\prime|^2)^2-a(t^2-|\xi^\prime|^2)\right). $$
Then applying $(-|\xi^\prime|^2 + D_N^2)L(D_N)$ to \eqref{F.Q.s.},
we have
\begin{equation}\label{q}
(\lambda + |\xi^\prime|^2 - D_N^2 +a)(-|\xi^\prime|^2 + D_N^2)L(D_N) \widehat{Q} =0.
\end{equation}
In term of \eqref{p}, \eqref{u}, and \eqref{q}, 
we prove that
\begin{equation}\label{r.f.}
    \begin{aligned}
    & \widehat{u}_j=A^0_je^{-Ax_N}+A_j^1e^{-L_1x_N}+A_j^2e^{-L_2x_N} \quad (j=1,\ldots, N) \\
    & \widehat{Q}_{jk}=A_{jk}e^{-Ax_N}+P_{jk}e^{-B_ax_N}+Q_{jk}^1e^{-L_1x_N}+Q_{jk}^2e^{-L_2x_N}\quad (j, k=1,\ldots, N) \\
    & \widehat{p}=Ce^{-Ax_N}+De^{-B_ax_N},
    \end{aligned}
\end{equation}
are solutions of the systems (\ref{F.u.s.}), (\ref{F.div.e.}) and (\ref{F.Q.s.}), where 
$$ A=|\xi^\prime|,\quad B_a=\sqrt{\lambda+a+|\xi^\prime|^2}, $$
and $L_{1,2}$ are the roots of $L(t)$
with ${\rm Re}(L_{1,2})>0$. It can be seen that 
$$ [L_{1,2}(\lambda,\xi^\prime)]^2=|\xi^\prime|^2+z_{1,2}(\lambda), $$
where $z_{1,2}$ are the roots of
$$ L(z)=(\lambda-z)(\lambda+a-z)+\frac{\beta^2}{2}\left(z^2-az\right). $$
In particular,
\begin{equation}\label{def.z.}
    \begin{array}{ll}
       z_{1,2} & =\displaystyle\frac{2\lambda+a(1+\beta^2/2)\pm\sqrt{(2\lambda+a(1+\beta^2/2))^2-4\lambda(\lambda+a)(1+\beta^2/2)}}{2(1+\beta^2/2)} \\
       & = \displaystyle\frac{2\lambda+a\left(1+\beta^2/2\right)\pm\sqrt{a^2\left(1+\beta^2/2\right)^2-2\lambda^2\beta^2}}{2\left(1+\beta^2/2\right)}.
    \end{array}
\end{equation}
Moreover, the coefficients of (\ref{r.f.}) satisfy the following relationships:
\begin{equation}\label{dep.f.A^0}
    \left\{\begin{array}{ll}
        A^0_j=-\lambda^{-1}i\xi_jC & j=1,\ldots, N-1 \\
        A^0_N=\lambda^{-1}AC & \null
    \end{array}\right.
\end{equation}
\begin{equation}\label{dep.f.A^1_k}
   \begin{aligned}
    &(L_2-L_1)\left\{B_a^3(L_1+L_2)-A^2B_a^2-A^2L_1L_2\right\}A^1_j\\
    &\enskip =-\left\{(B_a^2-L_2^2)E_j-L_2(B_aL_2-A^2)(\widehat{h}_j-A_j^0)\right\}(B_a^2-L_1^2)\quad j=1,\ldots, N-1 
  \end{aligned}
\end{equation}
\begin{equation}\label{dep.f.A^1_N}
    A_N^1=-\frac{A}{\lambda}\frac{L_2-A}{L_2-L_1}C-\frac{1}{L_2-L_1}i\xi^\prime\cdot \widehat{h}^\prime
\end{equation}
\begin{equation}\label{dep.f.A^2_k}
  \begin{aligned}
    &(L_2-L_1)\left\{B_a^3(L_1+L_2)-A^2B_a^2-A^2L_1L_2\right\}A^2_j \\
    &\enskip = \left\{(B_a^2-L_1^2)E_j-L_1(B_aL_1-A^2)(\widehat{h}_j-A_j^0)\right\}(B_a^2-L_2^2)\quad j=1,\ldots, N-1
  \end{aligned}
\end{equation}
\begin{equation}\label{dep.f.A^2_N}
    A_N^2=\frac{A}{\lambda}\frac{L_1-A}{L_2-L_1}C+\frac{1}{L_2-L_1}i\xi^\prime\cdot \widehat{h}^\prime
\end{equation}
\begin{equation}\label{dep.f.A_jk}
    \left\{\begin{array}{ll}
    A_{jk}=-\beta\lambda^{-1}(\lambda+a)^{-1}i\xi_ji\xi_kC & j,k=1,\ldots, N-1 \\
    A_{Nk}=\beta\lambda^{-1}(\lambda+a)^{-1}Ai\xi_k C & k=1,\ldots, N-1 \\
    A_{NN}=-\beta\lambda^{-1}(\lambda+a)^{-1}A^2C & \null
    \end{array}\right.
\end{equation}
\begin{equation}\label{dep.f.Q^1}
    \left\{\begin{array}{ll}
    Q^1_{jk}=\displaystyle\frac{\beta(i\xi_kA_j^1+i\xi_jA_k^1)}{2(B_a^2-L_1^2)} & j,k=1,\ldots, N-1 \\
    Q^1_{Nk}=\displaystyle\frac{\beta(i\xi_kA_N^1-L_1A_k^1)}{2(B_a^2-L_1^2)} & k=1,\ldots, N-1 \\
    Q^1_{NN}=-\displaystyle\frac{\beta L_1A_N^1}{B_a^2-L_1^2} & \null
    \end{array}\right.
\end{equation}
\begin{equation}\label{dep.f.Q^2}
    \left\{\begin{array}{ll}
    Q^2_{jk}=\displaystyle\frac{\beta(i\xi_kA_j^2+i\xi_jA_k^2)}{2(B_a^2-L_2^2)} & j,k=1,\ldots, N-1 \\
    Q^2_{Nk}=\displaystyle\frac{\beta(i\xi_kA_N^2-L_2A_k^2)}{2(B_a^2-L_2^2)} & k=1,\ldots, N-1 \\
    Q^2_{NN}=-\displaystyle\frac{\beta L_2A_N^2}{B_a^2-L_2^2} & \null
    \end{array}\right.
\end{equation}
\begin{equation}\label{dep.f.P}
    P_{jk}=-\frac{1}{B_a}\left(AA_{jk}+L_1Q_{jk}^1+L_2Q_{jk}^2+\widehat{H}_{jk}\right)
\end{equation}
\begin{equation}\label{dep.f.C}
    \left(\mathcal{I}_1+\frac{\mathcal{I}_2}{L_2-L_1}\right)\frac{A}{\lambda}C=\frac{\hbar}{L_2-L_1}+\mathcal{H}_1, 
\end{equation}
\begin{equation}\label{dep.f.D}
    D=\frac{\beta\lambda}{B_a}\left(\sum_{k=1}^{N-1}i\xi_kP_{Nk}-B_aP_{NN}\right),
\end{equation}
where
\begin{align}
    E_j&=\frac{2i\xi_jB_a}{\beta^2\lambda}D-\frac{2A}{\beta}\left(\sum_{k=1}^{N-1}i\xi_kA_{jk}-B_aA_{jN}\right)\nonumber\\
       & \quad + i\xi_j\sum_{\alpha=1,2}\frac{L_\alpha A_N^\alpha}{B_a+L_\alpha}-\frac{2}{\beta}\left(\sum_{k=1}^{N-1}\xi_k\widehat{H}_{jk}-B_a\widehat{H}_{jN}\right), \label{dep.f.E}\\
    \mathcal{I}_1&=\beta\frac{A^2}{B_a^2-A^2}\left\{2A^2-\frac{A(B_a^2+A^2)}{B_a}\right\}, \label{I_1for.}\\
    \mathcal{I}_2&=-\beta\displaystyle\frac{L_1(L_2-A)}{B_a^2-L_1^2}\left\{2A^2L_1-\frac{(B_a^2+A^2)(L_1^2+A^2)}{2B_a}\right\}+ \nonumber\\
    &\quad + \beta\frac{L_2(L_1-A)}{B_a^2-L_2^2}\left\{2A^2L_2-\frac{(B_a^2+A^2)(L_2^2+A^2)}{2B_a}\right\},\label{I_2for.}\\
     \hbar&=\displaystyle\Bigl[\beta\frac{L_1}{B_a^2-L_1^2}\left\{2A^2L_1-\frac{(B_a^2+A^2)(L_1^2+A^2)}{2B_a}\right\}\nonumber\\
      &\quad -\displaystyle\beta\frac{L_2}{B_a^2-L_2^2}\left\{2A^2L_2-\frac{(B_a^2+A^2)(L_2^2+A^2)}{2B_a}\right\}\Bigr]i\xi^\prime\cdot\widehat{h}^\prime, \label{hfor.}\\
        \mathcal{H}_1&=\displaystyle A^2\widehat{H}_{NN}-\frac{B_a^2+A^2}{B_a}\sum_{j=1}^{N-1}i\xi_j\widehat{H}_{jN}+\sum_{j,k=1}^{N-1}i\xi_ji\xi_k\widehat{H}_{jk},\nonumber
\end{align}
It is possible to write explicitly the value of the coefficients of \eqref{r.f.}, anyway these formulas are complicated and, in the paper, we will use just the relations from \eqref{dep.f.A^0} to \eqref{hfor.}. 

\vspace{2mm}

Our next task is to prove, as we anticipated, that 
$$ (u,p,Q)=\phi_\lambda(h,H), $$
with $\{\phi_\lambda\}_\lambda$ $\mathcal{R}$-bounded. In order to do so, we need some estimates over the coefficients of \eqref{r.f.}. We will do it more specifically in Section 3. In this section, we will prove that the relations from \eqref{dep.f.A^0} to \eqref{hfor.} give us a solution for the Fourier system from \eqref{F.u.s.} to \eqref{F.BCs}. This is not so easy for the presence of the functions:
$$ \mathcal{C}_a(\lambda,\xi^\prime)=\frac{1}{\lambda}\left(\mathcal{I}_1+\frac{\mathcal{I}_2}{L_2-L_1}\right), $$
$$ \mathcal{A}_a(\lambda,\xi^\prime)=B_a^3(L_1+L_2)-A^2B_a^2-A^2L_1L_2. $$
As it can be seen from the relations from (\ref{dep.f.A^0}) to (\ref{hfor.}), it is crucial that these two quantities are different from zero in order to have a well-defined solution. This will be the main aim of the section. Before going on, we notice that, if $\beta=0$, then $L_1(\lambda)\neq L_2(\lambda)$ for any $\lambda\in\Sigma_\theta$, otherwise
$$ L_1(\lambda)=L_2(\lambda)\:\Leftrightarrow\:\lambda= \pm\frac{a(1+\beta^2/2)}{\sqrt{2}|\beta|}. $$
The value $\eta\coloneqq \frac{a(1+\beta^2/2)}{\sqrt{2}|\beta|}$, is a positive real number, so it belongs to $\Sigma_\theta$ for any $\theta\in\left(0,\frac{\pi}{2}\right)$. For this reason, in the following, we will separate the case $\lambda\neq \eta$ and the case $\lambda=\eta$.

\subsection{The case $\lambda\neq\eta$}

We want to prove that $\mathcal{C}_a$ and $\mathcal{A}_a$ do not vanish for any $\lambda\in\Sigma_{\theta,r}$ and $\xi^\prime\in\mathbb{R}^{N-1}$ for some $\theta\in\left(0,\frac{\pi}{2}\right)$ and $r\ge 0$. The idea is the following: let us suppose $h=H=0$ and let $\lambda$ and $\xi^\prime$ such that $\mathcal{C}_a(\lambda,\xi^\prime)=0$, then \eqref{dep.f.C} gives no conditions over the coefficient $C$. In other words, we expect the value of $C$ to be not unique for such a choice of $\lambda$ and $\xi^\prime$. For this reason, we prove a uniqueness result for the Fourier system:
\begin{lem}\label{l.u.}
Let 
$$ \widehat{u}(\xi^\prime,\cdot), \:\widehat{Q}(\xi^\prime,\cdot),\:\widehat{p}(\xi^\prime,\cdot)\in C^\infty((0,+\infty))\cap H^2((0,+\infty)) $$
be a solution for the systems (\ref{F.u.s.}), (\ref{F.div.e.}) and (\ref{F.Q.s.}) with $a\ge 0$, $\beta\in\mathbb{R}$ and with initial conditions 
$$ \widehat{u}(0)=0,\quad D_N\widehat{Q}(0)=0, $$
let $\theta\in\left[\theta_0,\frac{\pi}{2}\right)$ with $\tan\theta_0\ge \frac{|\beta|}{\sqrt{2}}$, let $\lambda\in \overline{\Sigma_\theta}$,
then $\widehat{u}=\widehat{Q}\equiv0$. Moreover $\widehat{p}(\xi^\prime)=0$ for any $\xi^\prime\neq 0$.
\end{lem}
\begin{proof}\hfill\\
Thanks to (\ref{F.Q.s.}) and (\ref{F.div.e.}), for any $k=1,\ldots, N$ we have that
\begin{equation}\label{proof.u.1}
\begin{aligned}
    \sum_{j=1}^{N-1}i\xi_j(-|\xi^\prime|^2-a+D_N^2)\widehat{Q}_{kj}+D_N(-|\xi^\prime|^2-a+D_N^2)\widehat{Q}_{kN}= \\
= \lambda\left(\sum_{j=1}^{N-1}i\xi_j\widehat{Q}_{kj}+D_N\widehat{Q}_{kN}\right)+\frac{\beta}{2}(|\xi^\prime|^2-D_N^2)\widehat{u}_k.
\end{aligned}
\end{equation}
Now we use the fact that 
$$ \widehat{Q}_{kj}=\beta(\lambda+a+|\xi^\prime|^2-{D_N}_{|Neu}^2)^{-1}(\widehat{D(u)})_{kj}, $$
where $(\lambda+a+|\xi^\prime|^2-{D_N}_{|Neu}^2)^{-1}f$ is the only solution in $H^2((0,+\infty))$ of the system
$$ \left\{\begin{array}{ll}
    (\lambda+a+|\xi^\prime|^2-D_N^2)v=f & x_N\in (0,+\infty) \\
    D_Nv(0)=0. & \null
\end{array}\right. $$
It is easy to see that 
\begin{equation}\label{proof.u.2} v=(\lambda+a+|\xi^\prime|^2-{D^2_N}_{|Neu})^{-1}f=(\lambda+a+|\xi^\prime|^2-D^2_N)^{-1}E_{even}[f]_{|(0,+\infty)} ,\end{equation}
where $(\lambda+a+|\xi^\prime|^2-D^2)^{-1}$ is the resolvent for $D^2$ in $\mathbb{R}$ and $E_{even}[f]$ is the even extension of $f$. In particular, it can be seen that $(\lambda+a+|\xi^\prime|^2-D^2)^{-1}E_{even}[f]$ is still an even function in $x_N$.
If we use (\ref{proof.u.1}) in (\ref{F.u.s.}) we gain
\begin{equation}\label{proof.u.4}
    \begin{aligned}
        & \left[\lambda+\left(1+\frac{\beta^2}{2}\right)(|\xi^\prime|^2-D_N^2)\right]\widehat{u}_k+i\xi_k\widehat{p}+ \beta^2\lambda \sum_{j=1}^{N-1}i\xi_j(\lambda+a+|\xi^\prime|^2-{D_N}_{|Neu}^2)^{-1}\widehat{D(u)}_{kj} +  \\
        & +\beta^2\lambda D_N(\lambda+a+|\xi^\prime|^2-{D_N}_{|Neu}^2)^{-1}\widehat{D(u)}_{kN}=0 \quad (k=1,\ldots, N-1) \\
        & \left[\lambda+\left(1+\frac{\beta^2}{2}\right)(|\xi^\prime|^2-D_N^2)\right]\widehat{u}_N+D_N\widehat{p}+ \beta^2\lambda \sum_{j=1}^{N-1}i\xi_j(\lambda+a+|\xi^\prime|^2-{D_N}_{|Neu}^2)^{-1}\widehat{D(u)}_{Nj}+ & \null \\
        & +\beta^2\lambda D_N(\lambda+a+|\xi^\prime|^2-{D_N}_{|Neu}^2)^{-1}\widehat{D(u)}_{NN}=0\quad (k=N).
    \end{aligned}
\end{equation}
Thanks to the divergence free condition, we also have that
\begin{equation}\label{proof.u.5} (|\xi^\prime|^2-D_N^2)\widehat{u}= -2\widehat{{\rm Div}(D(u))}. \end{equation}
Let us multiply the $k$-th row of (\ref{proof.u.4}) for $\overline{\widehat{u}_k}$ in the sense of $L^2((0,+\infty))$ for any $k=1,\ldots, N$ and then we sum the rows:
\begin{equation}\label{proof.u.3}
\begin{aligned}
    \lambda\int_0^\infty|\widehat{u}|^2 dx_N+\left(1+\frac{\beta^2}{2}\right)\int_0^\infty |\xi^\prime|^2|\widehat{u}|^2+|D_N\widehat{u}|^2dx_N+ \\
    +\sum_{k=1}^{N-1}\int_0^\infty i\xi_k\widehat{p}\overline{\widehat{u}_k}dx_N+ \int_0^\infty D_N\widehat{p}\overline{\widehat{u}_N}dx_N+ \\
    +\lambda\beta^2\sum_{j=1}^{N-1}\int_0^\infty i\xi_j(\lambda+a+|\xi^\prime|^2-{D_N}_{|Neu}^2)^{-1}\widehat{D(u)}^j\cdot \overline{\widehat{u}}dx_N+ \\
    + \lambda\beta^2\int_0^\infty D_N(\lambda+a+|\xi^\prime|^2-{D_N}_{|Neu})^{-1}\widehat{D(u)}^N\cdot \overline{\widehat{u}}dx_N. 
\end{aligned}
\end{equation}
Firstly, we notice that
\begin{equation}\label{proof.u.6}
\int_0^\infty \sum_{k=1}^{N-1}i\xi_k\widehat{p}\overline{\widehat{u}_k}+D_N\widehat{p}\overline{\widehat{u}_N} dx_N=-\overline{\int_0^\infty\widehat{p}\left(\sum_{k=1}^{N-1}i\xi_k\widehat{u}_i+D_N\widehat{u}_N\right) dx_N}=0,
\end{equation}
where we've used that $\widehat{u}(0)=0$ and (\ref{F.div.e.}). Let us pass to the last terms of (\ref{proof.u.3}):
$$ \int_0^\infty i\xi_j(\lambda+a+|\xi^\prime|^2-{D_N}_{|Neu}^2)^{-1}\widehat{D(u)}^j\cdot \overline{\widehat{u}}dx_N=-\int_0^\infty (\lambda+a+|\xi^\prime|^2-{D_N}_{|Neu}^2)^ {-1}\widehat{D(u)}^j\cdot \overline{\widehat{\partial_ju}}dx_N. $$
Analogously, using the condition $\widehat{u}(0)=0$, we achieve 
$$ \int_0^\infty D_N(\lambda+a+|\xi^\prime|^2-{D_N}_{|Neu}^2)^{-1}\widehat{D(u)}^N\cdot \overline{\widehat{u}}dx_N=-\int_0^\infty(\lambda+a+|\xi^\prime|^2-{D_N}_{|Neu}^2)^{-1}\widehat{D(u)}^N\cdot \overline{\widehat{D_Nu}}dx_N. $$
Therefore 
$$ \int_0^\infty i\xi_j(\lambda+a+|\xi^\prime|^2-{D_N}_{|Neu}^2)^{-1}\widehat{D(u)}^j\cdot \overline{\widehat{u}}+ D_N(\lambda+a+|\xi^\prime|^2-{D_N}_{|Neu}^2)^{-1}\widehat{D(u)}^N\cdot \overline{\widehat{u}}dx_N= $$
$$ =- \int_0^\infty (\lambda+a+|\xi^\prime|^2-{D_N}_{|Neu}^2)^{-1}\widehat{D(u)}\colon \overline{\widehat{\nabla u}}dx_N, $$
where we recall $A\colon B={\rm tr}(B^TA)$ for any $A,B\in\mathbb{R}^{N^2}$. It is easy to check that $A\colon B=0$ when $A$ is symmetric and $B$ is anti-symmetric, so
$$ = - \int_0^\infty (\lambda+a+|\xi^\prime|^2-{D_N}_{|Neu}^2)^{-1}\widehat{D(u)}\colon \overline{\widehat{D(u)}}dx_N. $$
Now, thanks to (\ref{proof.u.2}) and the Plancherel Identity applied in $x_N$, we have that
$$ = -\frac{1}{2}\int_\mathbb{R} (\lambda+a+|\xi^\prime|^2-D_{N}^2)^{-1}E_{even}[\widehat{D(u)}]\colon \overline{E_{even}[\widehat{D(u)}]}dx_N= $$
$$ = -\frac{1}{2}\sum_{j,k=1}^N\int_\mathbb{R} (\lambda+a+|\xi^\prime|^2-D_{N}^2)^{-1}E_{even}[\widehat{D(u)}]_{jk} \overline{E_{even}[\widehat{D(u)}]_{jk}}dx_N= $$
$$ = -\frac{1}{2}\sum_{j,k=1}^N\int_\mathbb{R} \mathcal{F}_N\left[(\lambda+a+|\xi^\prime|^2-D_{N}^2)^{-1}E_{even}[\widehat{D(u)}]_{jk}\right]\overline{\mathcal{F}_N\left[E_{even}[\widehat{D(u)}]_{jk}\right]}dx_N, $$
where $\mathcal{F}_N(f)$ is the Fourier Transformation in $x_N$. Since
$$ \mathcal{F}_N[(\mu-\partial_{x_N}^2)^{-1}v](\xi_N)=\frac{\mathcal{F}_N[v](\xi_N)}{\mu+\xi_N^2}, $$
for any $v\in L^2(\mathbb{R})$ and for any $\mu\in \mathbb{C}\setminus\mathbb{R}_-$, then
$$ -\frac{1}{2}\sum_{j,k=1}^N\int_\mathbb{R} \mathcal{F}_N\left[(\lambda+a+|\xi^\prime|^2-D_{N}^2)^{-1}E_{even}[\widehat{D(u)}]_{jk}\right]\overline{\mathcal{F}_N\left[E_{even}[\widehat{D(u)}]_{jk}\right]}dx_N = $$
$$ = - \frac{1}{2}\int_\mathbb{R}\frac{\left|\mathcal{F}_N\left[E_{even}[\widehat{D(u)}]\right]\right|^2}{\lambda+a+|\xi^\prime|^2+\xi_N^2}d\xi_N. $$
On the other hand, thanks to (\ref{proof.u.5}), we can repeat the same argument for the free gradient part:
$$ \int_0^\infty (|\xi^\prime|^2-D_N^2)\widehat{u}\cdot \overline{\widehat{u}}dx_N=-2\int_0^\infty \widehat{{\rm Div}(D(u))}\cdot \overline{\widehat{u}}dx_N= $$
$$ = 2\int_0^\infty \widehat{D(u)}\colon \overline{\widehat{D(u)}}dx_N=\int_\mathbb{R}|E_{even}[\widehat{D(u)}]|^2dx_N=\int_\mathbb{R}\left|\mathcal{F}_N\left[E_{even}\left[\widehat{D(u)}\right]\right]\right|^2d\xi_N. $$
Therefore, if we come back to (\ref{proof.u.3}), we gain that 
\begin{equation}\label{proof.u.8}
\lambda\int_0^\infty|\widehat{u}|^2dx_N+\int_{\mathbb{R}}\left[\left(1+\frac{\beta^2}{2}\right)-\frac{\lambda\beta^2}{2(\lambda+a+|\xi^\prime|^2+\xi_N^2)}\right]\left|\mathcal{F}_N\left[E_{even}\left[\widehat{D(u)}\right]\right]\right|^2d\xi_N=0.
\end{equation}
Now we notice that
$$ 1+\frac{\beta^2}{2}-\frac{\beta^2\lambda}{2(\lambda+a+|\xi^\prime|^2+\xi_N^2)}=1+\frac{\beta^2(a+|\xi^\prime|^2+\xi_N^2)}{2(\lambda+a+|\xi^\prime|^2+\xi_N^2)}. $$
If $\lambda\in\mathbb{R}$, then $\lambda>0$ and (\ref{proof.u.8}) implies $\widehat{u}=0$ a.e. in $(0,+\infty)$. Otherwise, we can take the imaginary part of (\ref{proof.u.8}) we gain 
\begin{equation}\label{proof.u.9}
        Im\lambda\int_0^\infty |\widehat{u}|^2dx_N=\int_\mathbb{R}Im\left(\frac{\lambda\beta^2}{2(\lambda+a+|\xi^\prime|^2+\xi_N^2)}\right)\left|\mathcal{F}_N\left[E_{even}\left[\widehat{D(u)}\right]\right]\right|^2\xi_N.
\end{equation}
It can be seen that 
$$ Im\left(\frac{\lambda}{\lambda+a+|\xi^\prime|^2+\xi_N^2}\right)=\frac{Im\lambda(a+|\xi^\prime|^2+\xi_N^2)}{|\lambda+a+|\xi^\prime|^2+\xi_N^2|^2}, $$
therefore 
\begin{equation}\label{proof.u.10}
    \int_0^\infty|\widehat{u}|^2dx_N=\int_\mathbb{R}\frac{\beta^2(a+|\xi^\prime|^2+\xi_N^2)}{2|\lambda+a+|\xi^\prime|^2+\xi_N^2|^2}\left|\mathcal{F}_N\left[E_{even}\left[\widehat{D(u)}\right]\right]\right|^2d\xi_N.
\end{equation}
Applying (\ref{proof.u.10}) to (\ref{proof.u.8}) we get
$$ \int_\mathbb{R}\left[\left(1+\frac{\beta^2}{2}\right)-\frac{|\lambda|^2\beta^2}{2|\lambda+a+|\xi^\prime|^2+\xi_N^2|^2}\right]\left|\mathcal{F}_N\left[E_{even}\left[\widehat{D(u)}\right]\right]\right|^2d\xi_N=0. $$
Let us suppose for the moment that 
\begin{equation}\label{proof.u.11}
    \left(1+\frac{\beta^2}{2}\right)-\frac{|\lambda|^2\beta^2}{2|\lambda+a+|\xi^\prime|^2+\xi_N^2|^2}>0\quad \forall \xi=(\xi^\prime,\xi_N)\in\mathbb{R}^N.
\end{equation}
In this case, it has to happen that
$$ \mathcal{F}_N[E_{even}[\widehat{D(u)}]](\xi)=0\quad \text{for a.e.}\:\: \xi\in\mathbb{R}^N, $$
and, thanks to (\ref{proof.u.10}), we get also in this case $\widehat{u}=0$ a.e. on $(0,\infty)$. On the other hand, we have now that for any $k,j=1,\ldots, N$ the function $\widehat{Q}_{kj}$ resolves
$$ \left\{\begin{array}{ll}
    (\lambda+a+|\xi^\prime|^2-D_N^2)\widehat{Q}_{kj}=0 & \text{in}\:\:(0,+\infty) \\
    D_N\widehat{Q}_{ij}(0)=0 & \null
\end{array}\right. $$
and it is well-known that it implies $\widehat{Q}_{kj}\equiv0$. \\
Finally, if $\xi^\prime\neq0$, we can find $\xi_j\neq0$ for some $j=1,\ldots, N-1$. In this case, if we turn back to the $j$-th row of (\ref{F.u.s.}) we have that $p(\xi^\prime)=0$ for any $\xi^\prime\neq0$.
Finally, in order to verify property (\ref{proof.u.11}) we can notice that, if we denote
$$ f(x)\coloneqq \frac{1}{|\lambda+a+x|^2}=\frac{1}{(Re\lambda+a+x)^2+Im\lambda^2}\quad x\in\mathbb{R}, $$
then
$$ \max_{x\ge 0} f(x)=\left\{\begin{array}{ll}
    \frac{1}{|\lambda+a|^2} & Re\lambda\ge -a  \\
    \frac{1}{Im\lambda^2} & Re\lambda\le -a.
\end{array} \right. $$
By standard analytic arguments and the condition $\tan\theta\ge \frac{|\beta|}{\sqrt{2}}$ it can be seen that
$$ \frac{\beta^2|\lambda|^2}{2}\max_{\xi\in \mathbb{R}^N}f(|\xi|^2)\le 1+\frac{\beta^2}{2}. $$
\end{proof}
We are now ready to prove that $\mathcal{C}_a,\mathcal{A}_a\neq 0$ for the case $\lambda\neq \eta$: 
\begin{prop}\label{p.nonzero-prop.}
Let $a\ge0$, $\beta\neq 0$ and $\theta\in\left(0,\frac{\pi}{2}\right)$ with $\tan\theta\ge \frac{|\beta|}{\sqrt{2}}$, let us suppose that the coefficients of (\ref{r.f.}) satisfy (\ref{dep.f.A^0}), (\ref{dep.f.A^1_N}), (\ref{dep.f.A^2_N}), conditions from (\ref{dep.f.Q^1}) to (\ref{dep.f.P}) and (\ref{dep.f.D}), then $(\widehat{u},\widehat{p},\widehat{Q})$ defined as in \eqref{r.f.} is the only solution of the Fourier System and 
$$ \mathcal{A}_a(\lambda,\xi^\prime),\:\mathcal{C}_a(\lambda,\xi^\prime)\neq 0 \quad \forall \xi^\prime\in\mathbb{R}^{N-1},\quad \forall \lambda\in\overline{\Sigma_\theta}\setminus\{\eta\}, $$
where $\eta=\frac{a(1+\beta^2/2)}{\sqrt{2}|\beta|}$.
\end{prop}
\begin{proof}\hfill\\
Let $\widehat{h}=\widehat{H}\equiv 0$ a let us suppose by contradiction that there are $\lambda\in\overline{\Sigma_\theta}\setminus\{\eta\}$ and $\xi^\prime\in\mathbb{R}^{N-1}$ such that $\mathcal{A}_a(\lambda,\xi^\prime)=0$ or $\mathcal{C}_a(\lambda,\xi^\prime)=0$. Our claim is that we can find a solution of (\ref{F.u.s.}), (\ref{F.div.e.}), (\ref{F.Q.s.}) and 
(\ref{F.BCs}) different from 0. If we prove it, we get the contradiction with Lemma \ref{l.u.}.\\
In our first calculation, we will consider $C$, $A_k^1$ and $A_k^2$ for $k=1,\ldots, N-1$ as free variables and we will substitute the other parameters according to the relationships in hypothesis. With this partial substitution, many of the equations of the systems (\ref{F.u.s.}), (\ref{F.div.e.}), (\ref{F.Q.s.}) and (\ref{F.BCs}) are satisfied. In particular, we want to find $C$, $A_k^1$ and $A_k^2$ for $k=1,\ldots, N-1$ which solve the linear system
\begin{equation}\label{proof.n0prop.1} 
    \left\{ \begin{aligned}
    & i\xi_kD+\lambda\beta\sum_{j=1}^{N-1}i\xi_jP_{kj}-\lambda\beta B_aP_{kN}=0 & k=1,\ldots, N-1 \\
    & \sum_{k=1}^{N-1}i\xi_kA_k^1-L_1A_N^1=0 & \null \\
    & \sum_{k=1}^{N-1}i\xi_kA_k^2-L_2A_N^2=0 & \null \\
    & A_k^0+A_k^1+A_k^2=0 & k=1,\ldots, N-1.
\end{aligned}\right. \end{equation}
Let us see that the third equation depends from the fourth equation of (\ref{proof.n0prop.1}) and the definition of $A_N^j$ for $j=0,1,2$ in (\ref{dep.f.A^0}), (\ref{dep.f.A^1_N}) and (\ref{dep.f.A^2_N}): 
$$ \sum_{k=1}^{N-1}i\xi_kA_k^2-L_2A_N^2=-\sum_{k=1}^{N-1}i\xi_kA_k^0-\sum_{k=1}^{N-1}i\xi_kA_k^1+L_2A_N^0+L_2A_N^1= $$
$$ = (L_2-A)A_N^0+(L_2-L_1)A_N^1=\left[\frac{A(L_2-A)}{\lambda}-\frac{A(L_2-A)}{\lambda}\right]C=0. $$
So we reduce to the system
\begin{equation}\label{proof.n0prop.2} 
    \left\{ \begin{aligned}
    & i\xi_kD+\lambda\beta\sum_{j=1}^{N-1}i\xi_jP_{kj}-\lambda\beta B_aP_{kN}=0 & \quad k=1,\ldots, N-1 \\
    & \sum_{k=1}^{N-1}i\xi_kA_k^1-L_1A_N^1=0 & \null \\
    & A_k^0+A_k^1+A_k^2=0 & k=1,\ldots, N-1.
\end{aligned}\right. \end{equation}
Thanks to the relationships (\ref{dep.f.A^0}), (\ref{dep.f.A^1_N}), (\ref{dep.f.A^2_N}), from (\ref{dep.f.Q^1}) to (\ref{dep.f.P}) and (\ref{dep.f.D}), we can rewrite the system (\ref{proof.n0prop.2}) in the following way:
$$ \left(\begin{matrix}
\alpha_1Id_{N-1}  & \alpha_2Id_{N-1} & i\alpha_C\xi^\prime \\
i(\xi^\prime)^T & \underline{0}_{N-1}^T & \Lambda \\
Id_{N-1} & Id_{N-1} & -i\lambda^{-1}\xi^\prime
\end{matrix}\right)\left(\begin{matrix}
A_k^1 \\ A_k^2 \\ C
\end{matrix}\right)=\underline{0}_{2N-1}, $$
where $\underline{0}_d$ is the $0$ vector of $\mathbb{R}^d$ for any $d\in\mathbb{N}$ and 
$$ \alpha_C=\frac{A\beta^2}{B_a^2}\left[\frac{A(B_a-A)^2}{B_a^2-A^2}+\frac{L_2(A^2+L_2^2-3B_aL_2+B_a^2)(L_1-A)}{2(B_a^2-L_2^2)(L_2-L_1)}-\frac{L_1(A^2+L_1^2-3B_aL_1+B_a^2)(L_2-A)}{2(B_a^2-L_1^2)(L_2-L_1)}\right], $$
$$ \alpha_1=\frac{\lambda\beta^2L_1(A^2-B_aL_1)}{2B_a(B_a^2-L_1^2)},\quad \alpha_2=\frac{\lambda\beta^2L_2(A^2-B_aL_2)}{2B_a(B_a^2-L_2^2)},\quad \Lambda=\frac{L_1A(L_2-A)}{\lambda(L_2-L_1)}. $$
Let us call $M$ the previous matrix. By a computation we get 
\begin{equation}\label{proof.n0prop.4}
    \det M=(-1)^{N-1}(\alpha_1-\alpha_2)^{N-2}\left[A^2\left(\frac{\alpha_2}{\lambda}+\alpha_C\right)+\Lambda(\alpha_1-\alpha_2)\right].
\end{equation}
It can be seen that:
$$ \alpha_1-\alpha_2= \frac{\beta^2\lambda(L_2-L_1)\mathcal{A}_a}{2B_a(B_a^2-L_1^2)(B_a^2-L_2^2)}, $$
$$ \frac{\alpha_2}{\lambda}+\alpha_C=\frac{\beta^2L_2(A^2-B_aL_2)}{2B_a(B_a^2-L_2^2)}+ $$
$$ + \frac{\beta^2}{B_a^2}\left[\frac{A^2(B_a-A)^2}{B_a^2-A^2}+\frac{AL_2(A^2+L_2^2-3B_aL_2+B_a^2)(L_1-A)}{2(B_a^2-L_2^2)(L_2-L_1)}-\frac{AL_1(A^2+L_1^2-3B_aL_1+B_a^2)(L_2-A)}{2(B_a^2-L_1^2)(L_2-L_1)}\right]. $$
Finally
$$ A^2\left(\frac{\alpha_2}{\lambda}+\alpha_C\right)+\frac{L_1A(L_2-A)}{\lambda(L_2-L_1)}(\alpha_1-\alpha_2)=-\frac{\beta A}{B_a}\left(\mathcal{I}_1+\frac{\mathcal{I}_2}{L_2-L_1}\right). $$
This tell us that, if $\mathcal{A}_a=0$ or $\mathcal{C}_a=0$, we can find $C$, $A_k^1$ and $A_k^2$ different from 0 which resolves the systems (\ref{F.u.s.}), (\ref{F.div.e.}), (\ref{F.Q.s.}) and (\ref{F.BCs}), which contradicts Lemma \ref{l.u.}.
\end{proof}

\subsection{The case $\lambda=\eta$}

In this chapter, we will suppose $a> 0$. If we admit the case $\lambda=\eta$, we need to rewrite the representation formula (\ref{r.f.}): 
$$ \begin{array}{l}
\widehat{u}_k=A_k^0e^{-Ax_N}+(A_k^1+A_k^2)e^{-L_1x_N}+A_k^2(L_2-L_1)\mathcal{M}(L_2,L_1,x_N), \\
\widehat{Q}_{kj}=A_{kj}e^{-Ax_n}+P_{kj}e^{-B_ax_N}+(Q_{kj}^1+Q_{kj}^2)e^{-L_1x_N}+Q_{kj}^2(L_2-L_1)\mathcal{M}(L_2,L_1,x_N), \\
\widehat{p}=Ce^{-Ax_N}+De^{-B_ax_N},
\end{array} $$
where
$$ \mathcal{M}(\gamma_1,\gamma_2,t)=\frac{e^{-\gamma_1t}-e^{-\gamma_2t}}{\gamma_1-\gamma_2}\quad \forall \gamma_{1,2}\in\mathbb{C},\quad \forall t\in\mathbb{R}. $$
Let us call $L_0\coloneqq L_1(\eta)=L_2(\eta)$, then we claim that the representation formula for the solution is
\begin{equation}\label{rep.f.sol.lambda=eta}
    \begin{aligned}
    & \widetilde{u}_k=\widetilde{A}_k^0e^{-Ax_N}+\widetilde{A}_k^1e^{-L_0x_N}+\widetilde{A}_k^2x_Ne^{-L_0x_N}, \\
    & \widetilde{Q}_{kj}=\widetilde{A}_{kj}e^{-Ax_N}+\widetilde{P}_{kj}e^{-B_ax_N}+\widetilde{Q}_{kj}^1e^{-L_0x_N}+\widetilde{Q}_{kj}^2x_Ne^{-L_0x_N}, \\
    & \widetilde{p}=\widetilde{C}e^{-Ax_N}+\widetilde{D}e^{-B_ax_N}.
    \end{aligned}
\end{equation}
where
\begin{equation}\label{dep.f.A^0.eta}
        \widetilde{A}_k^0\coloneqq \lim_{\lambda\to\eta}A_k^0(\lambda)=A_k^0(\eta),
\end{equation}
\begin{equation}\label{dep.f.A^1.eta}
    \widetilde{A}_k^1\coloneqq \lim_{\lambda\to\eta} A_k^1(\lambda)+A_k^2(\lambda)= \left\{\begin{aligned}
    & \widehat{h}_k+\frac{i\xi_k}{\eta}\widetilde{C} & k<N \\
    & -\frac{A}{\eta}\widetilde{C} & k=N,
    \end{aligned}\right.
\end{equation}
\begin{equation}\label{dep.f.A^2.eta}
\begin{aligned}
    & \widetilde{A}_k^2\coloneqq-\lim_{\lambda\to\eta} (L_2(\lambda)-L_1(\lambda))A_k^2(\lambda) =  \\
    & = \left\{\begin{aligned}
   & -\widetilde{\mathcal{A}}_a^{-1}\left[(B_a^2-L_0^2)\left\{(B_a^2-L_0^2)\widetilde{E_k}-L_0(B_aL_0-A^2)(\widehat{h}_k-\widetilde{A}_k^0)\right\}\right] & k<N \\
   & -\frac{A(L_0-A)}{\eta}\widetilde{C}-i\xi^\prime\cdot \widehat{h}^\prime & k=N,
    \end{aligned}\right.
\end{aligned} 
\end{equation}
\begin{equation}\label{dep.f.Ajk.eta}
    \widetilde{A}_{jk}\coloneqq \lim_{\lambda\to \eta}A_{jk}(\lambda)=A_{jk}(\eta),
\end{equation}
\begin{equation}\label{dep.f.Q^1.eta}
    \widetilde{Q}^1_{jk}\coloneqq \lim_{\lambda\to\eta}Q_{jk}^1(\lambda)+Q_{jk}^2(\lambda)=\left\{\begin{aligned}
    & \frac{\beta(i\xi_k\widetilde{A}_j^1+i\xi_j\widetilde{A}_k^1)}{2(B_a^2-L_0^2)}-\frac{\beta L_0(i\xi_k\widetilde{A}_j^2+i\xi_j\widetilde{A}_k^2)}{(B_a^2-L_0^2)^2} & j,k<N \\
    & \frac{\beta(i\xi_k\widetilde{A}_N^1-L_0\widetilde{A}_k^1)}{2(B_a^2-L_0^2)}+\frac{\beta[(B_a^2+L_0^2)\widetilde{A}_k^2-2i\xi_kL_0\widetilde{A}_N^2]}{2(B_a^2-L_0^2)^2} & k<j=N \\
    & -\frac{\beta L_0\widetilde{A}_N^1}{B_a^2-L_0^2}+\frac{\beta(B_a^2+L_0^2)\widetilde{A}_N^2}{(B_a^2-L_0^2)^2} & k=j=N,
    \end{aligned}\right.
\end{equation}
\begin{equation}\label{dep.f.Q^2.eta}
    \widetilde{Q}^2_{jk}\coloneqq -\lim_{\lambda\to\eta}(L_2(\lambda)-L_1(\lambda))Q_{jk}^2(\lambda)=\left\{\begin{aligned}
    & \frac{\beta(i\xi_k\widetilde{A}_j^2+i\xi_j\widetilde{A}_k^2)}{2(B_a^2-L_0^2)} & j,k<N \\
    & \frac{\beta(i\xi_k\widetilde{A}_N^2-L_0\widetilde{A}_k^2)}{2(B_a^2-L_0^2)} & k<j=N \\
    & -\frac{\beta L_2\widetilde{A}_N^2}{B_a^2-L_0^2} & k=j=N,
    \end{aligned}\right.
\end{equation}
\begin{equation}\label{dep.f.P.tilde}
    \widetilde{P}_{jk}\coloneqq \lim_{\lambda\to\eta}P_{jk}(\lambda)=P_{jk}(\eta),    
\end{equation}
\begin{equation}
    \widetilde{D}\coloneqq \lim_{\lambda\to\eta}D(\lambda)=D(\eta),
\end{equation}
and where 
\begin{equation}\label{def.A-it.til.}
    \widetilde{\mathcal{A}}_a\coloneqq \lim_{\lambda\to\eta}\mathcal{A}_a=
    2B_a^3L_0-A^2B_a^2-A^2L_0^2.
\end{equation}
\begin{equation}\label{def.E.til.}
    \widetilde{E}_k=\frac{2i\xi_kB_a}{\beta^2\eta}\widetilde{D}-\frac{2A}{\beta}\left(\sum_{j=1}^{N-1}i\xi_j\widetilde{A}_{kj}-B_a\widetilde{A}_{kN}\right)+ i\xi_k\frac{L_0 \widetilde{A}_N^1}{B_a+L_0}-\frac{2}{\beta}\left(\sum_{k=1}^{N-1}\xi_j\widehat{H}_{kj}-B_a\widehat{H}_{jN}\right),
\end{equation}
\begin{equation}\label{dep.f.C.eta}
    \widetilde{C}=\lim_{\lambda\to\eta}C(\lambda),
\end{equation}
Everything is well-defined because
$$ C(\lambda)=A^{-1}\mathcal{C}_a(\lambda)^{-1}\left[\frac{\hbar(\lambda)}{L_2(\lambda)-L_1(\lambda)}+\mathcal{H}_1(\lambda)\right]. $$
It can be seen that 
$$ \frac{\mathcal{I}_2}{L_2-L_1}=\frac{4\beta A^2B_a[L_1L_2(B_a^2 + L_1L_2)-AB_a^2(L_1 + L_2)]}{2B_a(B_a^2-L_1^2)(B_a^2-L_2^2)}+ $$
$$ -\frac{\beta(B_a^2 + A^2)[(B_a^2 + A^2)L_1L_2(L_1 + L_2)- AB_a^2(L_1^2 + L_1L_2 + L_2^2 + A^2) + AL_1L_2(L_1L_2- A^2)]}{2B_a(B_a^2-L_1^2)(B_a^2-L_2^2)}; $$
$$ \frac{\hbar}{L_2-L_1}=-\frac{2\beta A^2B_a^2(L_1+L_2)}{(B_a^2-L_1^2)(B_a^2-L_2^2)}+\frac{\beta(B_a^2+A^2)[B_a^2(L_1^2+L_1L_2+L_2^2)+A^2B_a^2-L_1^2L_2^2+A^2L_1L_2]}{2B_a(B_a^2-L_1^2)(B_a^2-L_2^2)}. $$
Therefore there is $\lim_{\lambda\to\eta}C(\lambda)$. If we call 
\begin{equation}\label{f.C-it.til.}
    \widetilde{\mathcal{C}}_a\coloneqq \lim_{\lambda\to\eta}\mathcal{C}_a(\lambda),
\end{equation}
\begin{equation}\label{f.h.til.}
    \widetilde{\hbar}\coloneqq \lim_{\lambda\to\eta}\frac{\hbar(\lambda)}{L_2(\lambda)-L_1(\lambda)},\quad \widetilde{\mathcal{H}}_1=\lim_{\lambda\to\eta}\mathcal{H}_1(\lambda)=\mathcal{H}_1(\eta)
\end{equation}
then
$$ \widetilde{C}=A^{-1}\widetilde{\mathcal{C}}_a^{-1}[\widetilde{\hbar}+\widetilde{\mathcal{H}}_1]. $$
Moreover $L_0^2=A^2+z$, where
$$ z\coloneqq z_1(\eta)=z_2(\eta)=\frac{2\eta+a(1+\beta^2/2)}{2(1+\beta^2/2)}=\frac{a}{2}+\frac{\eta}{1+\beta^2/2}\in(0,a+\eta). $$
Therefore $A<L_0<B_a$ and 
$$ \widetilde{\mathcal{A}}_a=B_a^2(B_aL_0-A^2)+L_0(B_a^3-A^2L_0)>0. $$
As before, we need to prove that $\widetilde{\mathcal{C}}_a\neq0$ using again Lemma \ref{l.u.}:
\begin{prop}\label{p.nonzero-prop.eta}
Let $a>0$, $\beta\neq0$, then $\widetilde{\mathcal{C}}_a(\xi^\prime)\neq0$ for any $\xi^\prime\in\mathbb{R}^{N-1}$ and the functions given by (\ref{rep.f.sol.lambda=eta}) are solution of the system from (\ref{F.u.s.}) to (\ref{F.BCs}) for $\lambda=\eta$.  
\end{prop}
\begin{proof}\hfill\\
As in the previous section, we want to use the relations from \eqref{dep.f.A^0.eta} to \eqref{def.E.til.} in order to rewrite system from \eqref{F.u.s.} to \eqref{F.BCs} as a 
$$ \widetilde{M}\cdot(\widetilde{C},\widetilde{A}_k^1,\widetilde{A}_k^2)^T, $$
with $\widetilde{M}=\widetilde{M}(\xi^\prime)$ a matrix of suitable dimension. Let us suppose for simplicity $\widehat{h}=\widehat{H}=0$. In this case
$$ \widetilde{u}=\lim_{\lambda\to\eta}\widehat{u},\quad  \widetilde{Q}=\lim_{\lambda\to\eta}\widehat{Q}, \quad  \widetilde{p}=\lim_{\lambda\to\eta}\widehat{p}. $$
It is easy to prove that for any $\alpha\in\mathbb{N}$ we have
\begin{equation}\label{switch.der-lim.}
    \begin{array}{c}
D^\alpha_N\lim_{\lambda\to \eta}A_k^0(\lambda)e^{-Ax_N}=\lim_{\lambda\to\eta}D_N^\alpha A_k^0(\lambda)e^{-Ax_N} \\
D^\alpha_N\lim_{\lambda\to\eta}A_{jk}(\lambda)e^{-Ax_N}=\lim_{\lambda\to\eta}D^\alpha_NA_{jk}(\lambda)e^{-Ax_N} \\
D^\alpha_N\lim_{\lambda\to \eta}P_{jk}(\lambda)e^{-B_ax_N}=\lim_{\lambda\to\eta}D_N^\alpha P_{jk}^0(\lambda)e^{-B_ax_N} \\
D^\alpha_N\lim_{\lambda\to\eta}D(\lambda)e^{-B_ax_N}=\lim_{\lambda\to\eta}D^\alpha_ND(\lambda)e^{-B_ax_N} \\
D^\alpha_{N}\left(\lim_{\lambda\to\eta}(A_k^1(\lambda)+A_k^2(\lambda))e^{-L_1x_N}\right)=\lim_{\lambda\to \eta}D^\alpha_N(A_k^1(\lambda)+A_k^2(\lambda))e^{-L_1x_N} \\
D^\alpha_{N}\left(\lim_{\lambda\to\eta}(Q_{jk}^1(\lambda)+Q_{jk}^2(\lambda))e^{-L_1x_N}\right)=\lim_{\lambda\to \eta}D^\alpha_N(Q_{jk}^1(\lambda)+Q_{jk}^2(\lambda))e^{-L_1x_N} \\
D^\alpha_N\left(\lim_{\lambda\to\eta} A_k^2(\lambda) (e^{-L_2x_N}-e^{-L_1x_N})\right)= \lim_{\lambda\to \eta}D^\alpha_N A_k^2(\lambda)(e^{-L_1x_N}-e^{-L_2x_N}) \\
D^\alpha_N\left(\lim_{\lambda\to\eta} Q_{jk}^2(\lambda) (e^{-L_2x_N}-e^{-L_1x_N})\right)= \lim_{\lambda\to \eta}D^\alpha_N Q_{jk}^2(\lambda)(e^{-L_1x_N}-e^{-L_2x_N}).
\end{array}
\end{equation}
Using these identities and the fact that $(\widehat{u},\widehat{p},\widehat{Q})$ is a solution for the system for any $\lambda\neq \eta$, we can prove, repeating the strategy of Proposition \ref{p.nonzero-prop.}, that
\begin{equation}\label{proof.n0prop.eta.2} 
    \left\{ \begin{aligned}
    & i\xi_k\widetilde{\alpha}_C\widetilde{C}+\widetilde{\alpha}_1\widetilde{A}_k^1+\widetilde{\alpha}_2\widetilde{A}_k^2=0 & k=1,\ldots, N-1 \\
    & \sum_{k=1}^{N-1}i\xi_k\widetilde{A}_k^1-L_0\widetilde{A}_N^1+\widetilde{A}_N^2=0 & \null \\
    & \sum_{k=1}^{N-1}i\xi_k\widetilde{A}_k^2-L_0\widetilde{A}_N^2=0. & \null
\end{aligned}\right. \end{equation}
with
$$ \widetilde{\alpha}_C\coloneqq \lim_{\lambda\to \eta}\alpha_C(\lambda),\quad \widetilde{\alpha}_1\coloneqq \lim_{\lambda\to\eta}\alpha_1(\lambda), \quad \widetilde{\alpha}_2\coloneqq -\lim_{\lambda\to\eta}\frac{\alpha_2(\lambda)-\alpha_1(\lambda)}{L_2-L_1}, $$
and $\alpha_1,\alpha_2,\alpha_C$ be the same of the proof of Proposition \ref{p.nonzero-prop.}. Let us see that the second equation of (\ref{proof.n0prop.eta.2}) depends linearly by the other conditions:
$$ \sum_{k=1}^{N-1}i\xi_k\widetilde{A}_k^1-L_0\widetilde{A}_N^1+\widetilde{A}_N^2= $$
$$ =-\sum_{k=1}^{N-1}i\xi_k\widetilde{A}_k^0+L_0\widetilde{A}_N^0-\frac{A(L_0-A)}{\eta}\widetilde{C}= -A\widetilde{A}_N^0+L_0\widetilde{A}_N^0-\frac{A(L_0-A)}{\eta}\widetilde{C}=0. $$
Finally, if we also substitute the formulas (\ref{dep.f.A^1.eta}) and (\ref{dep.f.A^2.eta}) in (\ref{proof.n0prop.eta.2}), our system becomes
\begin{equation}\label{proof.n0prop.eta.3} 
    \left\{ \begin{array}{ll}
    i\xi_k\left(\widetilde{\alpha}_C+\frac{\widetilde{\alpha}_1}{\eta}\right)\widetilde{C}+\widetilde{\alpha}_2\widetilde{A}_k^2=0 & k=1,\ldots, N-1 \\
    \sum_{k=1}^{N-1}i\xi_k\widetilde{A}_k^2+\frac{L_0A(L_0-A)}{\eta}\widetilde{C}=0. & \null
\end{array}\right. \end{equation}
The matrix associated to (\ref{proof.n0prop.eta.3}) is 
$$ \widetilde{M}\coloneqq \left(\begin{matrix} \widetilde{\alpha}_2Id_{N-1} & i\xi^\prime\left(\widetilde{\alpha}_C+\frac{\widetilde{\alpha}_1}{\eta}\right) \\
i(\xi^\prime)^T & \frac{L_0A(L_0-A)}{\eta} 
\end{matrix}\right). $$
If we turn back to the matrix $M$ of Proposition \ref{p.nonzero-prop.}, we can notice that $\det\widetilde{M}=(-1)^{N-1}\det M$, with
$$ \left\{\begin{array}{l}
\alpha_1=\widetilde{\alpha}_2, \\
\alpha_2=0, \\
\alpha_C=\widetilde{\alpha}_C+\frac{\widetilde{\alpha}_1}{\eta}, \\
\Lambda=\frac{L_0A(L_0-A)}{\eta}, \\
\lambda=\eta.
\end{array} \right. $$
So, using the formula (\ref{proof.n0prop.4}), we get that
$$ \det \widetilde{M}=\widetilde{\alpha}_2^{N-2}\left[A^2\left(\widetilde{\alpha}_C+\frac{\widetilde{\alpha}_1}{\eta}\right)+\frac{L_0A(L_0-A)}{\eta}\widetilde{\alpha}_2\right]. $$
Now we notice that
$$ \lim_{\lambda\to\eta}\alpha_2(\lambda)=\lim_{\lambda\to\eta}\frac{\alpha_2(\lambda)-\alpha_1(\lambda)}{L_2(\lambda)-L_1(\lambda)}(L_2(\lambda)-L_1(\lambda))+\alpha_1(\lambda)=\widetilde{\alpha}_1. $$
So
$$ A^2\left(\widetilde{\alpha}_C+\frac{\widetilde{\alpha}_1}{\eta}\right)+\frac{L_0A(L_0-A)}{\eta}\widetilde{\alpha}_2= $$
$$ =\lim_{\lambda\to\eta} A^2\left(\frac{\alpha_2(\lambda)}{\lambda}+\alpha_C(\lambda)\right)+\frac{L_1A(L_2-A)}{\lambda(L_2(\lambda)-L_1(\lambda))}(\alpha_1(\lambda)-\alpha_2(\lambda))=$$
$$ = \lim_{\lambda\to\eta}-\frac{\beta A\lambda}{B_a}\mathcal{C}_a(\lambda)=-\frac{\beta A \eta}{B_a}\widetilde{\mathcal{C}}_a. $$  
We are now ready to conclude: let us suppose by contradiction that exists $\xi^\prime\in\mathbb{R}^{N-1}$ s.t. $\widetilde{\mathcal{C}}_a(\xi^\prime)=0$, then by the calculation we have just done, we have that $\widetilde{M}$ is a singular matrix and, in particular, we can find $(\widetilde{C},\widetilde{A}_k^2)^T\in\ker \widetilde{M}$. This means that we can find $\widetilde{C}$ and $\widetilde{A}_k^2$ not zero which solve the system (\ref{proof.n0prop.eta.2}) with $\widehat{h}=\widehat{H}=0$. Then, again by the calculation we have made so far, we can find a solution $(\widehat{u},\widehat{Q},\widehat{p})$ of systems from (\ref{F.u.s.}) to (\ref{F.BCs}) with $\widehat{h}=\widehat{H}=0$ and it contradicts Lemma \ref{l.u.}. Therefore $\widetilde{\mathcal{C}}_a\neq0$ for any $\xi^\prime\in\mathbb{R}^{N-1}$.

\vspace{2mm}

 On the other hand, since $\widetilde{\mathcal{C}}_a\neq0$ for any $\xi^\prime\in\mathbb{R}^{N-1}$, then $(\widetilde{u},\widetilde{p},\widetilde{Q})$ is the limit as $\lambda\to \eta$ of $(\widehat{u},\widehat{p},\widehat{Q})$ of (\ref{r.f.}) for any $\xi^\prime\in\mathbb{R}^{N-1}$. We already know that the functions of the previous chapter are solutions so, thanks to (\ref{switch.der-lim.}), we conclude that also these new functions solve systems (\ref{F.u.s.}) to (\ref{F.Q.s.}).
\end{proof}

\section{Resolvent Estimate}
\subsection{Lower bound for $\mathcal{C}_a$ and $\mathcal{A}_a$}

In this section, we start the proof of Theorem \ref{t.res.est.}, in particular we want to prove that $(u,p,Q)=\phi_\lambda(h,H)$ with $\{\phi_\lambda\}_\lambda$ a $\mathcal{R}$-bounded family of functions. Firstly, we need an estimate for the functions $\mathcal{C}_a$ and $\mathcal{A}_a$ extended to the case $\lambda=\eta$, that is
$$ \mathcal{C}_a=\left\{\begin{aligned}
    & \frac{1}{\lambda}\left(\mathcal{I}_1+\frac{\mathcal{I}_2}{L_2-L_1}\right) & \lambda\neq \eta \\
    & \lim_{\lambda\to \eta} \frac{1}{\lambda}\left(\mathcal{I}_1+\frac{\mathcal{I}_2}{L_2-L_1}\right) & \lambda=\eta,
\end{aligned}\right. $$
$$ \mathcal{A}_a=\left\{\begin{aligned}
    & B_a^3(L_1+L_2)-A^2B_a^2-A^2L_1L_2 & \lambda\neq \eta \\
& \lim_{\lambda\to\eta} B_a^3(L_1+L_2)-A^2B_a^2-A^2L_1L_2 & \lambda=\eta.
\end{aligned}\right.
$$
 Let us start from the estimates of $\widetilde{z}_j$ for $j=1,2$:
\begin{lem}\label{l.est.z}
Let $a\ge 0$ and $\beta\in\mathbb{R}$, let $\widetilde{z}_j\coloneqq \frac{z_j}{\lambda+a}$ for $j=1,2$, then 
$$ \widetilde{z}_1\xrightarrow{|\lambda|\to+\infty}z_-\coloneqq \frac{2-i|\beta|\sqrt{2}}{2\left(1+\frac{\beta^2}{2}\right)},\quad \widetilde{z}_2\xrightarrow{|\lambda|\to+\infty}z_+\coloneqq \frac{2+i|\beta|\sqrt{2}}{2\left(1+\frac{\beta^2}{2}\right)}. $$
Moreover, let $\theta\in\left(0,\frac{\pi}{2}\right)$ and $r>0$, there are $K_{m}=K_{m}(a,\beta,r)>0$ and $K_M=K_M(a,\beta)$ such that
$$ K_{m}\le |\widetilde{z}_j(\lambda)|\le K_M \quad \forall \lambda\in\overline{\Sigma_{\theta,r}}\:\:j=1,2. $$
\end{lem}
\begin{proof}\hfill\\
It is easy to prove that $\widetilde{z}_1\to z_-$ and $\widetilde{z}_2\to z_+$ as $|\lambda|\to+\infty$. Now we notice that $z_j\neq 0$ for any $\lambda\in \overline{\Sigma_{\theta,r}}$: by definition of $z_j$ we have that
\begin{equation}\label{proof.est.z} (\lambda-z_j)(\lambda+a-z_j)+\frac{\beta^2}{2}(z_j^2-az_j)=0 \quad\forall \lambda\in\overline{\Sigma_{\theta,r}}\quad j=1,2. \end{equation}
Let us suppose by contradiction that exists $\lambda\in\overline{\Sigma_{\theta,r}}$ such that $z_j(\lambda)=0$ for some $j=1,2$. Therefore, turning back to (\ref{proof.est.z}), we should have that
$$ \lambda(\lambda+a)=0, $$
which is impossible if $\lambda\in\overline{\Sigma_{\theta,r}}$. Moreover, since the limits are different from 0 for any $\beta\in\mathbb{R}$, by continuity in $\lambda$ we can conclude.
\end{proof}
In the previous section, we have proved that $\mathcal{C}_a$ and $\mathcal{A}_a$ are always different from 0. Since these two functions are also regular, it means that they assume values far from 0 in every compact of $\overline{\Sigma_{\theta,r}}\times\mathbb{R}^{N-1}$. Now we want to study the behaviour when $\lambda$ or $\xi^\prime$ go to infinity. In order to do so, we will use a Laurent series with a special kind of rest term: 
\begin{defn}\hfill\\
Let $\alpha\in\mathbb{Z}$ and let $A\subseteq\mathbb{C}$, then we denote
$$ o_\infty(A,t^{-\alpha})\coloneqq \left\{g(\lambda,t)\:\Bigg|\: \sup_{\lambda\in A}|g(\lambda,t)||t|^\alpha\xrightarrow{|t|\to+\infty}0\right\}, $$
$$ o_0(A,t^{\alpha})\coloneqq \left\{g(\lambda,t)\:\Bigg|\: \sup_{\lambda\in A}|g(\lambda,t)||t|^{-\alpha}\xrightarrow{|t|\to0}0\right\}. $$
\end{defn}
It is easy to see that such a symbols satisfy the rules of standard Landau Symbols:
\begin{lem}\label{l.opicc.}
Let $\alpha,\beta\in\mathbb{N}_0$ and $A,B\subseteq\mathbb{C}$ then
\begin{itemize}
    \item [1)] If $\beta\ge \alpha$, then $o_\infty(A,t^{-\beta})\subseteq o_\infty(A,t^{-\alpha})$ and $o_0(A,t^\beta)\subseteq o_0(A,t^\alpha)$;
    \item [2)] If $B\subseteq A$, then $o_\infty(A,t^{-\alpha}))\subseteq o_\infty(B,t^{-\alpha})$ and $o_0(A,t^\alpha)\subseteq o_0(B,t^\alpha)$;
    \item [3)] We have that
    $$ t^{\beta} o_\infty(A,t^{-\alpha})=o_\infty(A,t^{-\alpha+\beta}), \quad t^\beta o_0(A,t^\alpha)=o_0(A,t^{\alpha+\beta}); $$
    \item [4)] Let $f(\lambda)$ be s.t. $\sup_{\lambda\in A}|f(\lambda)|\le C$, then
    $$ f(\lambda)o_\infty(A,t^{-\alpha})\subseteq o_\infty(A,t^{-\alpha}), \quad f(\lambda)o_0(A,t^\alpha)\subseteq o_0(A,t^\alpha); $$
    \item [5)] We have that 
    $$ o_\infty(A,t^{-\alpha})+o_\infty(A,t^{-\beta})\subseteq o_\infty(A,t^{-\min\{\alpha,\beta\}}),\quad o_0(A,t^\alpha)+o_0(A,t^\beta)\subseteq o_0(A,t^{\min\{\alpha,\beta\}}); $$
    \item [6)] We have that 
    $$ f(t)\in o_0(A,t^\alpha)\:\Leftrightarrow\: f(1/t)\in o_\infty(A,t^{-\alpha}). $$
\end{itemize}
\end{lem}
The idea is the following: we rewrite $\mathcal{C}_a$ and $\mathcal{A}_a$ in dependence of $\lambda$ and $t=t(\lambda,\xi^\prime)$. We will take the Laurent series with respect with $t$ using the rest expressions we have just introduced:
\begin{prop}\label{p.est.C.lambda.neq.0}
Let $a,r>0$, $\beta\neq0$ and $\theta\in\left(\theta_0,\frac{\pi}{2}\right)$ with $\tan\theta_0\ge\frac{|\beta|}{\sqrt{2}}$, we can find $M=M(a,\beta,\theta,r)>0$ such that
$$ \inf_{\xi^\prime\in\mathbb{R}^{N-1},\lambda\in\Sigma_{\theta,r}}|\mathcal{C}_a(\lambda,\xi^\prime)|\ge M. $$
\end{prop}
\begin{proof}\hfill\\
We have already notice that 
$$ \frac{\mathcal{I}_2}{L_2-L_1}=\frac{4\beta A^2B_a[L_1L_2(B_a^2 + L_1L_2)-AB_a^2(L_1 + L_2)]}{2B_a(B_a^2-L_1^2)(B_a^2-L_2^2)}+ $$
$$ -\frac{\beta(B_a^2 + A^2)[(B_a^2 + A^2)L_1L_2(L_1 + L_2)- AB_a^2(L_1^2 + L_1L_2 + L_2^2 + A^2) + AL_1L_2(L_1L_2- A^2)]}{2B_a(B_a^2-L_1^2)(B_a^2-L_2^2)}. $$
Therefore it can be seen that, for any $\lambda\in\Sigma_\theta$ and for any $\xi^\prime \in\mathbb{R}^{N-1}$, we have
\begin{equation}\label{f.I1.t}
    \mathcal{I}_1(\lambda,\xi^\prime)=\beta(\lambda+a)\mathcal{F}_1\left(\lambda,\frac{|\xi^\prime|}{\sqrt{\lambda+a}}\right),
\end{equation}
\begin{equation}\label{f.I2.t}
    \left(\frac{\mathcal{I}_2}{L_2-L_1}\right)(\lambda,\xi^\prime)=\beta(\lambda+a)\left[\mathcal{F}_2^1+\mathcal{F}_2^2+\mathcal{F}_2^3+\mathcal{F}_2^4\right]\left(\lambda,\frac{|\xi^\prime|}{\sqrt{\lambda+a}}\right),
\end{equation}
where
$$ \mathcal{F}_1(\lambda,t)\coloneqq t^2\left(2t^2-t\sqrt{1+t^2}-\frac{t^3}{\sqrt{1+t^2}}\right), $$
$$ \mathcal{F}_2^1(\lambda,t)\coloneqq \frac{2t^2\left[\sqrt{t^2+\widetilde{z}_1}\sqrt{t^2+\widetilde{z}_2}(t^2+1+\sqrt{t^2+\widetilde{z}_1}\sqrt{t^2+\widetilde{z}_2})-t(t^2+1)(\sqrt{t^2+\widetilde{z}_1}+\sqrt{t^2+\widetilde{z}_2})\right]}{(1-\widetilde{z}_1)(1-\widetilde{z}_2)}, $$
$$ \mathcal{F}_2^2(\lambda,t)\coloneqq -\frac{(2t^2+1)^2\sqrt{t^2+\widetilde{z}_1}\sqrt{t^2+\widetilde{z}_2}(\sqrt{t^2+\widetilde{z}_1}+\sqrt{t^2+\widetilde{z}_2})}{2\sqrt{t^2+1}(1-\widetilde{z}_1)(1-\widetilde{z}_2)}, $$
$$ \mathcal{F}_2^3(\lambda,t)\coloneqq \frac{t(2t^2+1)(t^2+1)(3t^2+\widetilde{z}_1+\widetilde{z}_2+\sqrt{t^2+\widetilde{z}_1}\sqrt{t^2+\widetilde{z}_2})}{2\sqrt{t^2+1}(1-\widetilde{z}_1)(1-\widetilde{z}_2)}, $$
$$ \mathcal{F}_2^4(\lambda,t)\coloneqq -\frac{t(2t^2+1)\sqrt{t^2+\widetilde{z}_1}\sqrt{t^2+\widetilde{z}_2}(\sqrt{t^2+\widetilde{z}_1}\sqrt{t^2+\widetilde{z}_2}-t^2)}{2\sqrt{t^2+1}(1-\widetilde{z}_1)(1-\widetilde{z}_2)}, $$
where we recall that $\widetilde{z}_j=\frac{z_j}{\lambda+a}$. Let us call $t(\lambda,\xi^\prime)\coloneqq \frac{|\xi^\prime|}{\sqrt{\lambda+a}}$. The idea of the proof is to study the behaviour of $\mathcal{C}_a$ when $|t(\lambda,\xi^\prime)|\to+\infty$ and when $t(\lambda,\xi^\prime)$ is bounded. For what concerns the first part, it is sufficient to take the Laurent expression of $\mathcal{F}_1$ and $\mathcal{F}_2^j$ for $j=1,2,3,4$. In fact
$$ \sqrt{t^2+\widetilde{z}_j}=\widetilde{z}_j^{1/2}\sqrt{1+\frac{t^2}{\widetilde{z}_j}}=\widetilde{z}_j^\frac{1}{2}\left(\sum_{l=0}^NC_l\left(\frac{t}{\widetilde{z}_j^{1/2}}\right)^{1-2l}+f\left(\frac{t}{\widetilde{z}_j^{1/2}}\right)\right)\quad \forall N\in\mathbb{N}_0, $$
with $C_l\in\mathbb{R}$ and $f(x)\in o_\infty(|x|^{-2N})$. Now we notice that 
$$ f\left(\frac{t}{\widetilde{z}_j^{1/2}}\right)\in o_\infty\left(\Sigma_{\theta,r},t^{-2N}\right)\quad \forall \lambda\in\Sigma_{\theta,r}: $$
let $\varepsilon>0$, then we can find $M>0$ such that 
$$ \left|\frac{t}{\widetilde{z}_j^{1/2}}\right|^{2N}\left|f\left(\frac{t}{\widetilde{z}_j^{1/2}}\right)\right|\le \varepsilon\quad \forall \left|\frac{t}{\widetilde{z}_j^{1/2}}\right|\ge M. $$
Then, for any $|t|\ge MK_M^{1/2}$, we have that 
$$ \sup_{\lambda\in\Sigma_{\theta,r}}|t|^{2N}\left|f\left(\frac{t}{\widetilde{z}_j^{1/2}}\right)\right|= \sup_{\lambda\in\Sigma_{\theta,r}}|\widetilde{z}_j|^{N}\left|\frac{t}{\widetilde{z}_j^{1/2}}\right|^{2N}\left|f\left(\frac{t}{\widetilde{z}_j^{1/2}}\right)\right|\le K_M^{N}\varepsilon. $$
Moreover, thanks to point 4 of Lemma \ref{l.opicc.}, we have that 
$$ \widetilde{z}_j^\frac{1}{2}o_\infty\left(\Sigma_{\theta,r},t^{-2N}\right)\subseteq o_\infty\left(\Sigma_{\theta,r},t^{-2N}\right)\quad \forall N\in\mathbb{N}_0. $$
Therefore
$$ \sqrt{t^2+\widetilde{z}_j}=\sum_{l=1}^NC_l\left(\frac{t^{1-2l}}{\widetilde{z}_j^{-l}}\right)+o_\infty\left(\Sigma_{\theta,r},t^{-2N}\right) \quad \forall N\in\mathbb{N}_0. $$
Similarly it can be done the same for $\mathcal{F}_1$ and $\mathcal{F}_2^j$ with $j=1,2,3,4$ and it can be seen that 
\begin{equation}\label{proof.C-it.est.bel.}
  \mathcal{F}_1+\sum_{j=1}^4\mathcal{F}_{2}^j=-\frac{\widetilde{z}_1\widetilde{z}_2}{2(1-\widetilde{z_1})(1-\widetilde{z}_2)}+o_\infty(\Sigma_{\theta,r},1). 
\end{equation}
We have already noticed that $\mathcal{C}_a(\lambda,\xi^\prime)=\frac{\beta(\lambda+a)}{\lambda}F_a(\lambda,t(\lambda,\xi^\prime))$ where 
$$ F_a(\lambda,t)\coloneqq \mathcal{F}_1(\lambda,t)+\sum_{j=1}^4\mathcal{F}_2^j(\lambda,t)\quad \forall \lambda,t\in\mathbb{C}. $$
In particular
$$ \inf_{\xi^\prime\in\mathbb{R}^{N-1},\:\:\lambda\in\Sigma_{\theta,r}}|\mathcal{C}_a(\lambda,\xi^\prime)|\ge \inf_{t\in\mathbb{C},\:\:\lambda\in\Sigma_{\theta,r}}\left|\frac{\beta(\lambda+a)}{\lambda}F_a(\lambda,t)\right|\ge K(a,\beta,\theta,r)\inf_{t\in\mathbb{C},\:\lambda\in\Sigma_{\theta,r}}|F_a(\lambda,t)|. $$
Let 
$$ M\coloneqq \inf_{t\in\mathbb{C},\lambda\in\Sigma_{\theta,r}}|F_a(\lambda,t)|, $$
so, it is sufficient to prove that $M>0$. Let us suppose by contradiction that $M=0$. In this case we can find $\{\lambda_n,t_n\}$ such that $F_a(\lambda_n,t_n)\to0$ as $n\to+\infty$. 
Thanks to (\ref{proof.C-it.est.bel.})
$$ \sup_{\lambda\in\Sigma_{\theta,r}}\left|F_a(\lambda,t)+\frac{\widetilde{z}_1\widetilde{z}_2}{2(1-\widetilde{z_1})(1-\widetilde{z}_2)}\right|\xrightarrow{|t|\to+\infty}0. $$
In particular, for any $\varepsilon>0$, we can find $0<R_1=R_1(\varepsilon)$ such that 
$$ \sup_{|t|\ge R_1,\:\:\lambda\in\Sigma_{\theta,r}}\left|F_a(\lambda,t)+\frac{\widetilde{z}_1\widetilde{z}_2}{2(1-\widetilde{z_1})(1-\widetilde{z}_2)}\right|\le \varepsilon. $$
Thanks to Lemma \ref{l.est.z} it is easy to see that $1\lesssim |1-\widetilde{z}_j|\lesssim  1$ and $1\lesssim |\widetilde{z}_j|\lesssim 1$ for any $\lambda\in\Sigma_{\theta,r}$ and for $j=1,2$. Therefore, choosing $\varepsilon$ sufficiently small, we can find $R_1>0$ such that 
$$ \inf_{\lambda\in\Sigma_{\theta,r}}|F_a(\lambda,t)|>0\quad \forall t\in B_{R_1}^c. $$
Since $M=0$, it has to happen that $|t_n|\le R_1$ for any $n\in\mathbb{N}_0$. Less then subsequences, we can suppose $t_n\to \widehat{t}\in\mathbb{C}$. We can distinguish two cases:
\begin{itemize}
    \item If $|\lambda_n|\le R_2$ for some $R_2$ for every $n\in\mathbb{N}_0$, we can suppose that
    $$ \lambda_n\to\widehat{\lambda}\in\overline{\Sigma_{\theta,r}}, $$
    less then subsequences. In particular, by continuity of $F$, we should get
    $$ F_a(\widehat{\lambda},\widehat{t})=\lim_{n\to+\infty}F_a(\lambda_n,t_n)=\inf_{t\in\mathbb{C},\:\:\lambda\in\Sigma_{\theta,r}}F_a(\lambda,t)=0, $$
    which contradicts Propositions \ref{p.nonzero-prop.} (when $\widehat{\lambda}\neq\eta$) or  \ref{p.nonzero-prop.eta} (when $\widehat{\lambda}=\eta$).
    \item Let us suppose, less then subsequences, that $|\lambda_n|\to +\infty$ as $n\to+\infty$. Since the dependence of $F_a$ from $\lambda$ and $a$ derives from $\widetilde{z}_1$ and $\widetilde{z}_2$, we can also write 
    $$ F_a(\lambda,t)=\widetilde{F}(\widetilde{z}_1(a,\lambda_n),\widetilde{z}_2(a,\lambda_n),t) $$
    for a suitable function $\widetilde{F}$. We have already pointed out in Lemma \ref{l.est.z} that $\widetilde{z}_l\to z_\pm\coloneqq \frac{2\pm i\sqrt{2}|\beta|}{2(1+\beta^2/2)}$ as $|\lambda|\to+\infty$ so, by continuity of $\widetilde{F}$, we have that
    $$ \widetilde{F}(\widetilde{z}_1(a,\lambda_n),\widetilde{z}_2(a,\lambda_n),t_n)\xrightarrow{n\to+\infty} \widetilde{F}(z_-,z_+,\widehat{t}). $$
    Now we notice that $\widetilde{F}(z_-,z_+,t)=F_0(1,t)$ for any $t\in\mathbb{C}$. The case $a=0$ and $\lambda=1$ is still described in Proposition \ref{p.nonzero-prop.}, so we get a contradiction as before.
\end{itemize}
\end{proof}
\begin{prop}\label{p.est.A.lambda.neq.0}
Let $a,r>0$, $\beta\neq0$ and $\theta\in\left(\theta_0,\frac{\pi}{2}\right)$ with $\tan\theta_0\ge\frac{|\beta|}{\sqrt{2}}$, we can find $M=M(a,\beta,\theta,r)>0$ such that
$$ \inf_{\xi^\prime\in\mathbb{R}^{N-1},\lambda\in\Sigma_{\theta,r}}\left|\mathcal{A}_a(\lambda,\xi^\prime)\right|\ge M. $$
\end{prop}
\begin{proof}\hfill\\
The idea of the proof is the same as before: it can be seen that 
$$ \mathcal{A}_a(\lambda,\xi^\prime)=(\lambda+a)^2G_a(\lambda,t(\lambda,\xi^\prime)) $$
with $t(\lambda,\xi^\prime)=\frac{|\xi^\prime|}{\sqrt{\lambda+a}}$ and
$$ G_a(\lambda,t)=(t^2+1)^\frac{3}{2}(\sqrt{t^2+\widetilde{z}_1}+\sqrt{t^2+\widetilde{z}_2})-t^2(t^2+1)-t^2\sqrt{t^2+\widetilde{z}_1}\sqrt{t^2+\widetilde{z}_2}. $$
Taking the Laurent expression for $G$, we get as before that
$$ G_a(\lambda,t)=2t^2+o_\infty(\Sigma_{\theta,r},t). $$
Therefore, we have that $|G_a(\lambda,t)|\to+\infty$ as $|t|\to+\infty$. In particular, it remains far from 0 uniformly in $\lambda\in\Sigma_{\theta,r}$. On the other hand, thanks to Propositions \ref{p.nonzero-prop.} and \ref{p.nonzero-prop.eta}, we know that $\mathcal{A}_a(\lambda,\xi^\prime)\neq0$ for any $\lambda\in\overline{\Sigma_{\theta,r}}$ and $\xi^\prime\in\mathbb{R}^{N-1}$. Finally, we can conclude as in the previous proposition.
\end{proof}

\subsection{Fourier Multipliers}

In order to get the $L^p$-$L^q$ maximal regularity, we need to know the behaviour of the derivatives of the coefficients in $\xi^\prime$ and $\lambda$. For this reason, we introduce these two classes of Fourier Multipliers (see \cite{SS12}):
\begin{defn}\hfill\\
Let $\theta\in\left(0,\frac{\pi}{2}\right)$, let $r\ge0$, let $m(\xi^\prime,\lambda)\colon\mathbb{R}^{N-1}\setminus\{0\}\times\Sigma_{\theta,r}\to\mathbb{C}$ be a $C^\infty$ function in $\xi^\prime$ and a $C^1$ function in $\tau\in\mathbb{R}\setminus\{0\}$, where $\lambda=\gamma+i\tau$. If there is $s\in\mathbb{R}$ such that, for any $\xi^\prime\in\mathbb R^{N-1}$ and for any $\lambda\in \Sigma_{\theta,r}$, it holds
$$ \begin{array}{ll}
    \left|D^\alpha_{\xi^\prime}m(\xi^\prime,\lambda)\right|\lesssim (|\lambda|^\frac{1}{2}+1+|\xi^\prime|)^{s-|\alpha|} & \forall \alpha\in\mathbb{N}_0^{N-1}, \\
    \left|D^\alpha_{\xi^\prime}\tau \partial_\tau m(\xi^\prime,\lambda)\right|\lesssim (|\lambda|^\frac{1}{2}+1+|\xi^\prime|)^{s-|\alpha|} & \forall \alpha\in\mathbb{N}_0^{N-1},
\end{array} $$
then we say that $m(\xi^\prime,\lambda)$ is a \textbf{Fourier Multiplier of Order s with type 1}. Conversely, if for any $\xi^\prime\in\mathbb R^{N-1}$ and for any $\lambda\in \Sigma_{\theta,r}$, it holds
$$ \begin{array}{ll}
    \left|D^\alpha_{\xi^\prime}m(\xi^\prime,\lambda)\right|\lesssim (|\lambda|^\frac{1}{2}+1+|\xi^\prime|)^{s}|\xi^\prime|^{-|\alpha|} & \forall \alpha\in\mathbb{N}_0^{N-1}, \\
    \left|D^\alpha_{\xi^\prime}\tau \partial_\tau m(\xi^\prime,\lambda)\right|\lesssim (|\lambda|^\frac{1}{2}+1+|\xi^\prime|)^{s}|\xi^\prime|^{-|\alpha|} & \forall \alpha\in\mathbb{N}_0^{N-1},
\end{array} $$
then we say that $m(\xi^\prime,\lambda)$ is a \textbf{Fourier Multiplier of Order s with type 2}. We will denote $\mathscr{M}_{s,i,\theta,r}$ the set of the Fourier multipliers of order $s$ and type $i$.
\end{defn}
\begin{rem}
In this chapter we will follow the approach used in \cite{SS12}. In this paper, the definition of Fourier Multiplier is stated with $|\lambda|^\frac{1}{2}+|\xi^\prime|$ as right-hand side of the inequality. Anyway, in the following, we will consider the case of $\lambda\in\Sigma_{\theta,r}$ with $r>0$ so, it is easy to see that $|\lambda|^\frac{1}{2}\sim |\lambda|^\frac{1}{2}+1$.
\end{rem}
The Fourier Multipliers satisfy the following algebraic rules:
\begin{lem}\label{l.F.M.0}
Let $s_1,s_2\in\mathbb{R}$, $i=1,2$, $\theta\in\left(0,\frac{\pi}{2}\right)$ and $r\ge 0$, then
\begin{enumerate}
    \item If $m_i\in\mathscr{M}_{s_i,1,\theta,r}$, then $m_1m_2\in\mathscr{M}_{s_1+s_2,1,\theta,r}$;
    \item If $m_i\in\mathscr{M}_{s_i,i,\theta,r}$, then $m_1m_2\in\mathscr{M}_{s_1+s_2,2,\theta,r}$;
    \item If $m_i\in\mathscr{M}_{s_i,2,\theta,r}$, then $m_1m_2\in\mathscr{M}_{s_1+s_2,2,\theta,r}$.
\end{enumerate}
\end{lem}
The proof follows by the definition of Fourier Multipliers. From now on, our aim will be to rewrite the coefficients of the solution formula of $(\widehat{u},\widehat{p},\widehat{Q})$ in \eqref{r.f.} as a combination of Fourier Multipliers. Some results are already known:
\begin{lem}\label{l.F.M.1}
Let $a>0$, $\theta\in\left(0,\frac{\pi}{2}\right)$, then
\begin{enumerate}
    \item $B_a^s\in\mathscr{M}_{s,1,\theta,0}$ for any $s\in\mathbb{R}$;
    \item $A^s\in\mathscr{M}_{s,2,\theta,0}$ for any $s\ge 0$.
\end{enumerate}
\end{lem}
The proof of this lemma can be found in \cite{SS12} (Lemma 5.2) changing $\lambda$ with $\lambda+a$ and using the next well-known lemma:
\begin{lem}\label{l.est.lambda+a}
Let $\theta\in\left(0,\frac{\pi}{2}\right)$, $\alpha\ge 0$ and $\lambda\in\Sigma_\theta$, then 
$$ |\lambda+\alpha|\ge C(\theta)(|\lambda|+\alpha). $$
\end{lem}
Now we need to understand that $L_j$ for $j=1,2$ is a Fourier Multiplier. In order to do so, we will need some lemma. The first is a technical one:
\begin{lem}\label{l.est.Re>0}
Let $z\in\mathbb{C}$ with $Re(z)\ge0$, let $\alpha\ge 0$, then
$$ |z+\alpha|\gtrsim |z|+\alpha. $$
\end{lem}
\begin{proof}\hfill\\
We can divide the proof in two cases:
\begin{itemize}
    \item If $Arg(z)\le \frac{\pi}{4}$, then $z=|z|(\cos\theta+i\sin\theta)$ with $\cos\theta\ge \frac{\sqrt{2}}{2}$. Therefore
    $$ |z+\alpha|=\sqrt{(|z|\cos\theta+\alpha)^2+|z|^2\sin^2\theta}\ge |z|\cos\theta+\alpha\gtrsim|z|+\alpha. $$
    \item If $Arg(z)>\frac{\pi}{4}$, then $\sin\theta>\frac{\sqrt{2}}{2}$ and $Rez\ge 0$, so
    $$ |z+\alpha|=\sqrt{(Rez+\alpha)^2+|z|^2\sin^2\theta}\ge \sqrt{\alpha^2+|z|^2\sin^2\theta}\gtrsim \alpha+|z|\sin\theta\gtrsim |z|+\alpha. $$
\end{itemize}
\end{proof}
In the paper, we will need not only to prove that $L_j$ is a Fourier Multiplier, we will also need an estimate for $L_j^s$ with $s\in\mathbb{R}$. In particular, for the case $s<0$, we will need an estimate from below of $L_j$. The strategy will be the same used for $\mathcal{C}_a$, $\mathcal{A}_a$: we need firstly to prove that $L_j(\lambda,\xi^\prime)\neq0$ for any $\lambda\in\overline{\Sigma_{\theta,r}}$ and $\xi^\prime\in\mathbb{R}^{N-1}$:
\begin{lem}\label{l.realz}
Let $a\ge0$, $r>0$, $\beta\in\mathbb{R}$, $\theta\in\left(0,\frac{\pi}{2}\right)$ and $\lambda\in\overline{\Sigma_{\theta,r}}$, if $\theta\in\left(\theta_0,\frac{\pi}{2}\right)$ with $\tan\theta_0\ge \frac{|\beta|}{\sqrt{2}}$ or if $Re\lambda>-\frac{a}{2}$, then $z_j\not\in\mathbb{R}_-$ for $j=1,2$.
\end{lem}
\begin{proof}\hfill\\
Let us suppose by contradiction that $z_j=-\alpha^2$ for some $\alpha\in\mathbb{R}$ and $j=1,2$ and for some $\lambda\in\overline{\Sigma_{\theta,r}}$. So, by definition of $z_j$, we have that
\begin{equation}\label{proof.est.der.L.2}
    (\lambda+\alpha^2)(\lambda+a+\alpha^2)+\frac{\beta^2}{2}(\alpha^4+a\alpha^2)=0.
\end{equation}
If we take the imaginary part of (\ref{proof.est.der.L.2}) we get
$$ Im\lambda(Re\lambda+a+\alpha^2)+(Re\lambda+\alpha^2)Im\lambda=Im\lambda(2\alpha^2+2Re\lambda+a)=0. $$
If $Im\lambda=0$, then $Re\lambda=\lambda>0$ and we have a contradiction with (\ref{proof.est.der.L.2}). So, if $Re\lambda>-\frac{a}{2}$, we have a contradiction. Otherwise, let us suppose 
$$ Re\lambda=-\alpha^2-\frac{a}{2}. $$
Let us take now the real part of (\ref{proof.est.der.L.2}):
$$ \alpha^4\left(1+\frac{\beta^2}{2}\right)+\alpha^2\left(2Re\lambda+a\left(1+\frac{\beta^2}{2}\right)\right)+Re\lambda(Re\lambda+a)-(Im\lambda)^2=0. $$ 
Therefore
$$ (Im\lambda)^2=\alpha^4\left(1+\frac{\beta^2}{2}\right)+\alpha^2\left[-2\alpha^2+\frac{\beta^2}{2}a\right]+\left(-\alpha^2-\frac{a}{2}\right)\left(-\alpha^2+\frac{a}{2}\right)= \frac{\beta^2}{2}\left[\alpha^4+a\alpha^2\right]-\frac{a^2}{4}. $$
Combining these two information and with standard analytic arguments, we get that 
$$ \tan^2\theta\le \frac{|Im\lambda|^2}{|Re\lambda|^2}=\frac{\frac{\beta^2}{2}\left[\alpha^4+a\alpha^2\right]-\frac{a^2}{4}}{\alpha^4+a\alpha^2+\frac{a^2}{4}}\le \frac{\beta^2}{2}, $$
which contradicts the hypothesis on $\theta$.
\end{proof}
As a consequence, we get that $L_j=\sqrt{|\xi^\prime|^2+z_j}\neq0$.
\begin{lem}\label{l.est.above.BL}
Let $a,r>0$, $\beta\in\mathbb{R}$ and $\theta\in\left(\theta_0,\frac{\pi}{2}\right)$ with $\tan\theta_0\ge \frac{|\beta|}{\sqrt{2}}$ and $\lambda\in\Sigma_{\theta,r}$, then
$$ Re(L_j)\ge C(a,\beta,\theta,r)\left( |\lambda|^\frac{1}{2}+1+|\xi^\prime|\right)\quad j=1,2. $$
\end{lem}
\begin{proof}\hfill\\
Let $\xi^\prime\in\mathbb{R}^{N-1}$ and $\lambda\in\Sigma_{\theta,r}$, then
$$ Re\left(\sqrt{|\xi^\prime|^2+z_j}\right)=\frac{1}{\sqrt{2}}\sqrt{||\xi^\prime|^2+z_j|+|\xi^\prime|^2+Rez_j}. $$
Let $t\coloneqq \frac{|\xi^\prime|}{|z_j|^\frac{1}{2}}$, then 
$$ Re\left(\sqrt{|\xi^\prime|^2+z_j}\right)=\frac{|z_j|^\frac{1}{2}}{\sqrt{2}}\sqrt{\left|t^2+\frac{z_j}{|z_j|}\right|+t^2+\frac{Rez_j}{|z_j|}}= $$
$$ = \frac{|z_j|^\frac{1}{2}}{\sqrt{2}}\sqrt{\sqrt{\left(t^2+\frac{Rez_j}{|z_j|}\right)^2+\left|\frac{Im z_j}{|z_j|}\right|^2}+t^2+\frac{Rez_j}{|z_j|}}. $$
Taking the Laurent expression in $t$, we get
$$ Re\left(\sqrt{|\xi^\prime|^2+z_j}\right) = |z_j|^\frac{1}{2} t+o_\infty(\Sigma_{\theta,r},1). $$
On the other hand
$$ |\lambda|^\frac{1}{2}+a^\frac{1}{2}+|\xi^\prime|=|z_j|^\frac{1}{2}\left(\frac{|\lambda|^\frac{1}{2}+a^\frac{1}{2}}{|z_j|^\frac{1}{2}}+t\right). $$
Moreover, from Lemma \ref{l.est.z} and \ref{l.est.lambda+a}, we have 
$$ \frac{|\lambda|^\frac{1}{2}+a^\frac{1}{2}}{|z_j|^\frac{1}{2}}=\frac{|\lambda|^\frac{1}{2}+a^\frac{1}{2}}{|\lambda+a|^\frac{1}{2}}\cdot|\widetilde{z}_j|^{-\frac{1}{2}}\lesssim 1. $$
Therefore 
$$ \frac{Re(L_j)}{|\lambda|^\frac{1}{2}+a^\frac{1}{2}+|\xi^\prime|}=G(\lambda,t(\lambda,\xi^\prime)), $$
with
\begin{equation}\label{proof.realz}
    G(\lambda,t)\coloneqq \frac{\frac{1}{\sqrt{2}}\sqrt{\left|t^2+\frac{z_j}{|z_j|}\right|+t^2+\frac{Rez_j}{|z_j|}}}{\frac{|\lambda|^\frac{1}{2}+a^\frac{1}{2}}{|z_j|^\frac{1}{2}}+t}=1+o_\infty(\Sigma_{\theta,r},1).
\end{equation}
Let us see that $Re L_j\neq 0$ for any $\xi^\prime\in\mathbb{R}^{N-1}$ and $\lambda\in\overline{\Sigma_{\theta,r}}$: if $Re L_j=0$ for some $j=0,1$ then 
$$ z_j=-|\xi^\prime|^2-\alpha^2\quad \text{for some}\:\:\alpha\in\mathbb{R} $$
and this contradicts Lemma \ref{l.realz}. Finally, we can conclude: let us suppose by contradiction that 
$$ \inf_{\lambda\in\Sigma_{\theta,r},\:t\in\mathbb{C}}|G(\lambda,t)|=0. $$
Thanks to (\ref{proof.realz}), we know that $G$ remains far from 0 uniformly in $\lambda\in\Sigma_{\theta,r}$ as $|t|\to+\infty$. On the other hand we have just seen that $Re(L_j)\neq0$ for any $\xi^\prime\in\mathbb{R}^{N-1}$ and $\lambda\in\overline{\Sigma_{\theta,r}}$. Finally, we can get the contradiction as in the proof of the Proposition (\ref{p.est.C.lambda.neq.0}) (to be noticed that Lemma \ref{l.realz} is true also when $a=0$). In this way we proved that 
$$ |Re L_j|\ge C(a,\beta,\theta,r)\left(|\lambda|^\frac{1}{2}+1+|\xi^\prime|\right). $$
On the other hand, by definition of $L_j$, $Re(L_j)>0$ for any choice of $\lambda\in\Sigma_\theta$ and $j=1,2$. Therefore we conclude.
\end{proof}
We are now ready to prove the estimate for $L_j$:
\begin{lem}\label{l.F.M.2}
Let $a,r>0$, $\beta\in\mathbb{R}$ and $\theta\in\left(\theta_0,\frac{\pi}{2}\right)$ with $\tan\theta_0\ge \frac{|\beta|}{\sqrt{2}}$, then 
\begin{enumerate}
    \item For any $\lambda\in\Sigma_{\theta,r}$ we can find $C_1,C_2>0$ depending on $a,\beta,\theta, r$ such that
    $$ C_1(|\lambda|+1)\le |z_j(\lambda)|\le C_2(|\lambda|+1)\quad j=1,2; $$
    \item For any $\lambda\in\Sigma_{\theta,r}$ we can find $C_1,C_2>0$ depending on $a,\beta,\theta,r$ such that
    $$ C_1(|\lambda|+1)\le |\lambda+a-z_j|\le C_2(|\lambda|+1)\quad j=1,2; $$
    \item For any $\xi^\prime\in\mathbb{R}^{N-1}$, for any $\lambda\in\Sigma_{\theta,r}$, for any $s\in\mathbb{R}$ and for any $\alpha\in\mathbb{N}_0^{N-1}$ we have
    $$ |D^\alpha_{\xi^\prime} L_j^s|\le C(a,\beta,\theta,r) \left(|\lambda|^\frac{1}{2}+1+|\xi^\prime|\right)^{s-|\alpha|}; $$
    \item For any $\xi^\prime\in\mathbb{R}^{N-1}$, for any $\lambda\in\Sigma_{\theta,r}$ and for any $\alpha\in\mathbb{N}_0^{N-1}$ we have
    $$ |D^\alpha_{\xi^\prime} (\tau\partial_\tau L_j)|\le C(a,\beta,\theta,r)\left(|\lambda|^\frac{1}{2}+1+|\xi^\prime|\right)^{1-|\alpha|}. $$
\end{enumerate}
\end{lem}
\begin{proof}\hfill\\
For the first point, it is sufficient to notice that, thanks to Lemma \ref{l.est.lambda+a}
$$ \left|\frac{z_j}{|\lambda|+1}\right|\sim |\widetilde{z}_j| $$
and from Lemma \ref{l.est.z} we know that $1\lesssim |\widetilde{z}_j|\lesssim 1$. For what concerns the second point, it is sufficient to see that
$$ \lambda+a-z_j=(\lambda+a)(1-\widetilde{z}_j) $$
and using the limits for $\widetilde{z}_j$ from Lemma \ref{l.est.z} we get also that $1\lesssim |1-\widetilde{z}_j|\lesssim 1$. So using Lemma \ref{l.est.lambda+a} we conclude also the proof of the second point. In order to prove the third point, we can use Bell's Formula:
\begin{equation}\label{proof.est.der.L.1}
    |D^\alpha_{\xi^\prime}L_j^s|\lesssim \sum_{l=1}^{|\alpha|}|f^{(l)}(|\xi^\prime|^2+z_j)|\sum_{\alpha_1+\cdots+\alpha_l=\alpha,\:\:|\alpha_l|\ge1} |D^{\alpha_1}_{\xi^\prime}(|\xi^\prime|^2+z_j)|\cdots|D^{\alpha_l}_{\xi^\prime}(|\xi^\prime|^2+z_j)|,
\end{equation}
where $f(x)\coloneqq x^{s/2}$. We know from Lemma \ref{l.est.above.BL} that
$$ |L_j|\ge Re(L_j)\gtrsim |\lambda|^\frac{1}{2}+a^\frac{1}{2}+|\xi^\prime|. $$ 
Moreover, it's easy to see that 
$$ |D^\alpha_{\xi^\prime}(|\xi^\prime|^2+z_j)|\lesssim (|\lambda|^\frac{1}{2}+a^\frac{1}{2}+|\xi^\prime|)^{2-|\alpha|}\quad \forall \alpha\in\mathbb{N}_0^{N-1}. $$
In fact, the inequality is obvious for $|\alpha|>2$, while the other cases follow by a simple computation and from the estimate of $z_j$ we have just proved. Finally, turning back to (\ref{proof.est.der.L.1}), we have that
$$ |D^\alpha_{\xi^\prime}L_j^s|\lesssim \sum_{l=1}^{|\alpha|}|L_j|^{s-2l}\sum_{\alpha_1+\cdots+\alpha_l=\alpha,\:|\alpha_l|\ge 1}(|\lambda|^\frac{1}{2}+a^\frac{1}{2}+|\xi^\prime|)^{2-|\alpha_1|}\cdots(|\lambda|^\frac{1}{2}+a^\frac{1}{2}+|\xi^\prime|)^{2-|\alpha_l|}\le $$
$$ \le \sum_{l=1}^{|\alpha|}(|\lambda|^\frac{1}{2}+a^\frac{1}{2}+|\xi^\prime|)^{s-2l}(|\lambda|^\frac{1}{2}+a^\frac{1}{2}+|\xi^\prime|)^{2l-|\alpha|}=(|\lambda|^\frac{1}{2}+a^\frac{1}{2}+|\xi^\prime|)^{s-|\alpha|}. $$
Finally, let us study $\tau\partial_\tau L_j$ for $j=1,2$:
$$  \tau\partial_\tau L_j=\frac{\tau\partial_\tau z_j}{2L_j}=\frac{i\tau z_j^\prime}{2 L_j}=\frac{Im\lambda z_j^\prime}{2L_j}. $$
From the previous point we have then
\begin{equation}\label{proof.est.der.L.3}
    |D^\alpha_{\xi^\prime}(\tau\partial_\tau L_j)|\lesssim (|\lambda|^\frac{1}{2}+a^\frac{1}{2}+|\xi^\prime|)^{-1-|\alpha|}|Im\lambda z_j^\prime|.
\end{equation} 
So, in order to conclude, we need to estimate the last term uniformly in $\lambda$:
$$ z_j^\prime=\frac{1}{2(1+\beta^2/2)}\left[2\pm\frac{2\lambda\beta^2}{\sqrt{a^2(1+\beta^2/2)^2-2\lambda^2\beta^2}}\right]. $$
If $\beta=0$, then $|\tau\partial_\tau z_j|\lesssim |\lambda|$. Otherwise,
$$ z_j^\prime=\frac{1}{2(1+\beta^2/2)}\left[2-\frac{|\beta|\sqrt{2}\lambda}{(\eta^2-\lambda^2)^\frac{1}{2}}\right], $$
where $\eta\coloneqq\frac{a(1+\beta^2/2)}{|\beta|\sqrt{2}}\in\mathbb{R}$. We can now notice that
$$ Im\lambda=-Im(\eta-\lambda). $$
Therefore
$$ |Im\lambda z_j^\prime|\le 2|Im\lambda|+\frac{\sqrt{2}|\beta| |Im\lambda||\lambda|}{|\eta^2-\lambda^2|^\frac{1}{2}}\lesssim |\lambda|+\frac{|Im\lambda|^\frac{1}{2}|\lambda|}{|\eta+\lambda|^\frac{1}{2}}. $$
Since $\eta>0$, we can use Lemma \ref{l.est.lambda+a}. So, for any $\beta\in\mathbb{R}$
\begin{equation}\label{proof.est.tau-der.z}
    |\tau\partial_\tau z_j|\lesssim |\lambda|\le (|\lambda|^\frac{1}{2}+a^\frac{1}{2}+|\xi^\prime|)^2 \quad \forall j=1,2.
\end{equation} 
Finally, applying the previous estimate in (\ref{proof.est.der.L.3}), we conclude.
\end{proof}
Now, we need to rewrite $\mathcal{C}_a$ in a different way from the previous chapters: we already know that
$$ \mathcal{C}_a=\frac{1}{\lambda}\left(\mathcal{I}_1+\frac{\mathcal{I}_2}{L_2-L_1}\right), $$
with 
$$ \mathcal{I}_1=\beta\frac{A^2}{B_a^2-A^2}\left\{2A^2-\frac{A(B_a^2+A^2)}{B_a}\right\}, $$
$$ \mathcal{I}_2=-\beta\frac{L_1(L_2-A)}{B_a^2-L_1^2}\left\{2A^2L_1-\frac{(B_a^2+A^2)(L_1^2+A^2)}{2B_a}\right\}+\beta\frac{L_2(L_1-A)}{B_a^2-L_2^2}\left\{2A^2L_2-\frac{(B_a^2+A^2)(L_2^2+A^2)}{2B_a}\right\}. $$
It will be useful in the following to underline the differences hided inside $\mathcal{I}_1$ and $\mathcal{I}_2$. It can be seen by a calculation that 
\begin{equation}\label{f.alt.C-it}
    \begin{aligned}
        & \mathcal{I}_1+\frac{\mathcal{I}_2}{L_2-L_1}= -\frac{\beta A^3(B_a^2-A^2)}{B_a(B_a+A)^2}+ \\
        & + \frac{\beta AB_a(L_1-A)(L_2-A)[(B_a^2+L_1A)(L_1-A)+L_2(A^2-B_a^2)]}{B_a(B_a^2-L_1^2)(B_a^2-L_2^2)} + \\
        & + \frac{\beta (B_a-A)^2[(A^2B_a^2+A^2L_1L_2-B_a^2L_1L_2+L_2^2B_a^2)(A-L_1)+2A^2L_1(B_a^2-A^2)]}{2B_a(B_a^2-L_1^2)(B_1^2-L_2^2)}.
    \end{aligned}
\end{equation}
As a consequence of Lemma \ref{l.F.M.1} and Lemma \ref{l.F.M.2}, we have that many factors of \eqref{f.alt.C-it} are Fourier Multipliers.
\begin{lem}\label{l.F.M.3}
Let $a,r>0$, $\beta\in\mathbb{R}$, $\theta\in\left(\theta_0,\frac{\pi}{2}\right)$ with $\tan\theta_0\ge \frac{|\beta|}{\sqrt{2}}$, let $j=1,2$ then
\begin{enumerate}
    \item $(A+B_a)^{-1}\in\mathscr{M}_{-1,2,\theta,r}$; 
    \item $(L_j+A)^{-1}\in\mathscr{M}_{-1,2,\theta,r}$;
    \item $(L_j+B_a)^{-1}\in\mathscr{M}_{-1,1,\theta,r}$;
    \item $\frac{L_j-A}{B_a^2-L_k^2}, \frac{L_j-B_a}{B_a^2-L_k^2},\frac{L_j-A}{B_a^2-A^2}, \frac{L_j-B_a}{B_a^2-A^2} \frac{B_a-A}{B_a^2-L_k^2}\in\mathscr{M}_{-1,1,\theta,r}$ for any $j,k=1,2$;
    \item $\left(\frac{S_1-S_2}{\sqrt{\lambda+a}}\right)^s,\left(\frac{\lambda+a}{\lambda}\right)^s\in\mathscr{M}_{0,1,\theta,r}$ for any $s=1,-1$ and for any $S_1,S_2=B_a,L_1,L_2$;
    \item $\frac{S-A}{\sqrt{\lambda+a}}\in\mathscr{M}_{0,2,\theta,r}$ for any $S=B_a,L_1,L_2$.
\end{enumerate}
\end{lem}
\begin{proof}\hfill\\
The first point is proved in Lemma 5.2 of \cite{SS12}. The proof of the second and third points is almost the same, using that
$$ |L_j+A|\ge Re(L_j+A)\ge Re(L_j)\gtrsim |\lambda|^\frac{1}{2}+a^\frac{1}{2}+|\xi^\prime|. $$ 
For what concerns the fourth point, let $j,k=1,2$, then
$$ \frac{L_j-A}{B_a^2-L_k^2}=\frac{L_j^2-A^2}{B_a^2-L_k^2}\frac{1}{L_j+A}. $$
We already know that $(L_j+A)^{-1}\in\mathscr{M}_{-1,1,\theta,r}$. On the other hand
$$ \frac{L_j^2-A^2}{B_a^2-L_k^2}=\frac{z_j}{\lambda+a-z_k} $$
and thanks to Lemma \ref{l.F.M.2}, we have $\left|\frac{L_j^2-A^2}{B_a^2-L_k^2}\right|\lesssim 1$ for any $\lambda\in\Sigma_{\theta,r}$. Moreover
$$ \tau\partial_\tau\left(\frac{L_j^2-A^2}{B_a^2-L_k^2}\right)=\frac{\tau\partial_\tau z_j}{B_a^2-L_k^2}-\frac{(i\tau-\tau\partial_\tau z_k)(L_j^2-A^2)}{(B_a^2-L_k^2)^2}. $$
We have already seen in (\ref{proof.est.tau-der.z}) that 
$$ |\tau\partial_\tau z_j|\lesssim |\lambda|\quad \forall j=1,2. $$
Therefore, as before, we have that $\left|\tau\partial_\tau \left(\frac{L_j^2-A^2}{B_a^2-L_k^2}\right)\right|\lesssim 1$. The other results can be proved in the same way. The last two points are obvious.
\end{proof}
Thanks to the previous lemma and \eqref{f.alt.C-it} we can prove that $\mathcal{C}_a^{-1}$ is a Fourier Multiplier:
\begin{lem}\label{l.F.M.C-it.}
Let $a,r>0$, $\beta\neq0$ let $\theta\in\left(\theta_0,\frac{\pi}{2}\right)$ with $\tan\theta_0\ge \frac{|\beta|}{\sqrt{2}}$, then we can find $L>0$ and $\rho_j(\xi^\prime,\lambda)\in\mathscr{M}_{-1,1,\theta,r}$ and $m_j\in\mathscr{M}_{1,2,\theta,r}$ for $j=1,\ldots, L$ such that
$$ \mathcal{C}_a=\frac{\lambda+a}{\lambda}\sum_{j=1}^L\rho_j(\xi^\prime,\lambda)m_j(\xi^\prime,\lambda). $$
Moreover, $\mathcal{C}_a^{-1}\in\mathscr{M}_{0,2,\theta,r}$.
\end{lem}
\begin{proof}\hfill\\
The first part of the result comes from the decomposition of $\mathcal{C}_a$ we have made in (\ref{f.alt.C-it}) and the Lemmas \ref{l.F.M.0}, \ref{l.F.M.1} and \ref{l.F.M.3}. In particular, $\mathcal{C}_a\in\mathscr{M}_{0,2,\theta,r}$. Let now $s\in\mathbb{N}$, then by Bell's Formula we have
$$ |D^\alpha_{\xi^\prime}\mathcal{C}_a^{-s}|\lesssim \sum_{l=1}^{|\alpha|}|\mathcal{C}_a|^{-s-l}\sum_{\alpha_1+\ldots+\alpha_l=\alpha,\:\:|\alpha_k|\ge 1}|D^{\alpha_1}_{\xi^\prime}\mathcal{C}_a|\cdots |D^{\alpha_l}_{\xi^\prime}\mathcal{C}_a|. $$
Now, using that $\mathcal{C}_a\in\mathscr{M}_{0,2,\theta,r}$ and Proposition \ref{p.est.C.lambda.neq.0}, we get that
$$ |D^\alpha_{\xi^\prime}\mathcal{C}_a^{-s}|\lesssim |\xi^\prime|^{-|\alpha|}\quad \forall s\in\mathbb{N},\quad \forall \alpha\in\mathbb{N}_0^{N-1}. $$
Now we notice that 
$$ \tau\partial_\tau\mathcal{C}_a^{-1}=-\frac{\tau\partial_\tau\mathcal{C}_a}{\mathcal{C}_a^2}. $$
Therefore, using again that $\mathcal{C}_a\in\mathscr{M}_{0,2,\theta,r}$, the previous estimate with $s=2$ and the Leibniz Formula we get
$$ |D^\alpha_{\xi^\prime}(\tau\partial_\tau\mathcal{C}_a^{-1})|\lesssim \sum_{\beta+\gamma=\alpha}|D^\beta_{\xi^\prime}\tau\partial_\tau\mathcal{C}_a||D^\gamma_{\xi^\prime}\mathcal{C}_a^{-2}|\lesssim |\xi^\prime|^{-|\alpha|}. $$
\end{proof}
Let us pass to $\mathcal{A}_a$. As for $\mathcal{C}_a$, it convenient to rewrite $\mathcal{A}_a$ in order to highlight the cancellations:
$$ \mathcal{A}_a=B^2_a(B_aL_1-A^2)+L_2(B_a^3-A^2L_1)= $$
$$ = B_a^3(L_1-B_a)+B_a^2(B_a^2-A^2)+L_2B_a(B_a^2-A^2)+L_2A^2(B_a-L_1)= $$
$$ = (B_a^2-A^2)(B_a^2+L_2B_a)+(B_a-L_1)[A^2(L_2-B_a)+B_a(A^2-B_a^2)]= $$
$$ = B_a(B_a^2-A^2)(L_1+L_2)-A^2(B_a-L_1)(B_a-L_2). $$
\begin{lem}\label{l.F.M.A-it}
Let $a,r>0$, $\beta\neq0$ and $\theta\in\left(\theta_0,\frac{\pi}{2}\right)$ with $\tan\theta_0\ge\frac{|\beta|}{\sqrt{2}}$, then we can find $L>0$, $\rho_j(\xi^\prime,\lambda)\in\mathscr{M}_{0,1,\theta,r}$ and $m_j\in\mathscr{M}_{2,1,\theta,r}$ for $j=1,\ldots, L$ such that
$$ \mathcal{A}_a=(\lambda+a)\sum_{j=1}^L\rho_j(\xi^\prime,\lambda)m_j(\lambda,\xi^\prime). $$
Moreover, $(\lambda+a)\mathcal{A}_a^{-1}\in\mathscr{M}_{-2,1,\theta,r}$.
\end{lem}
\begin{proof}\hfill\\
Thanks to the decomposition we just made and the Lemma \ref{l.F.M.3} we have the formula we are looking for. In particular, $\frac{\mathcal{A}_a}{\lambda+a}\in\mathscr{M}_{2,1,\theta,r}$. Thanks to Lemma \ref{l.F.M.1}, we have that $\frac{\mathcal{A}_a}{(\lambda+a)B_a^2}\in\mathcal{M}_{0,1,\theta,r}$. Moreover, we know from Proposition \ref{p.est.A.lambda.neq.0}, that 
$$ \inf_{\lambda\in\Sigma_{\theta,r},\:\xi^\prime\in\mathbb{R}^{N-1}} \left|\frac{\mathcal{A}_a(\lambda,t)}{(\lambda+a)B_a^2}\right|>0. $$
Let now $s\in\mathbb{N}$, then by Bell's Formula
$$ \left|D^\alpha_{\xi^\prime}\left(\frac{(\lambda+a)B_a^2}{\mathcal{A}_a}\right)^{s}\right|=\left|D^\alpha_{\xi^\prime}\left(\frac{\mathcal{A}_a}{(\lambda+a)B_a^2}\right)^{-s}\right|\lesssim $$
$$ \lesssim \sum_{l=1}^{|\alpha|} \frac{|\mathcal{A}_a|^{-(s+l)}}{B_a^{-2(s+l)}|\lambda+a|^{-(s+l)}}\sum_{\alpha_1+\ldots+\alpha_l=\alpha,\:\:|\alpha_k|\ge 1}\left|D^{\alpha_1}_{\xi^\prime}\left(\frac{\mathcal{A}_a}{B_a^2(\lambda+a)}\right)\right|\cdots\left|D^{\alpha_l}_{\xi^\prime}\left(\frac{\mathcal{A}_a}{B_a^2(\lambda+a)}\right)\right|\lesssim $$
$$ \lesssim (|\lambda|^\frac{1}{2}+a^\frac{1}{2}+|\xi^\prime|)^{-|\alpha|}. $$
Now we notice that
$$ \tau\partial_\tau\left(\frac{\mathcal{A}_a}{(\lambda+a)B_a^2}\right)=B_a^{-2}\tau\partial_\tau\left(\frac{\mathcal{A}_a}{\lambda+a}\right)-\frac{i\tau\mathcal{A}_a}{(\lambda+a)B_a^4}. $$
Therefore, using that $\frac{\mathcal{A}_a}{\lambda+a}\in\mathscr{M}_{2,1,\theta,r}$ and the Leibniz Formula, we get
$$ \left|D^\alpha_{\xi^\prime}\left(\tau\partial_\tau\left(\frac{\mathcal{A}_a}{(\lambda+a)B_a^2}\right)\right)\right|\lesssim $$
$$ \lesssim \sum_{\beta+\gamma=\alpha}|D^\beta_{\xi^\prime}B_a^{-2}|\left|D^\gamma_{\xi^\prime}\tau\partial_\tau\left(\frac{\mathcal{A}_a}{\lambda+a}\right)\right|+|\lambda|\left|D^\beta_{\xi^\prime}\left(\frac{\mathcal{A}_a}{\lambda+a}\right)\right||D^\gamma_{\xi^\prime}B_a^{-4}|\lesssim (|\lambda|^\frac{1}{2}+a^\frac{1}{2}+|\xi^\prime|)^{-|\alpha|}. $$
Now we notice that
$$ \tau\partial_\tau\left(\frac{(\lambda+a)B_a^2}{\mathcal{A}_a}\right)=-\tau\partial_\tau\left(\frac{\mathcal{A}_a}{(\lambda+a)B_a^2}\right)\frac{(\lambda+a)^2B_a^4}{\mathcal{A}_a^2}. $$
Finally, using Leibniz Formula, we get
$$ \left|D^\alpha_{\xi^\prime}\tau\partial_\tau\left(\frac{(\lambda+a)B_a^2}{\mathcal{A}_a}\right)\right|\lesssim $$
$$ \lesssim \sum_{\beta+\gamma=\alpha}\left|D^\beta_{\xi^\prime}\left(\tau\partial_\tau\left(\frac{\mathcal{A}_a}{(\lambda+a)B_a^2}\right)\right)\right|\left|D^\gamma_{\xi^\prime}\left(\frac{(\lambda+a)^2B_a^4}{\mathcal{A}_a^2}\right)\right|\lesssim (|\lambda|^\frac{1}{2}+a^\frac{1}{2}+|\xi^\prime|)^{-|\alpha|}. $$
Therefore $\frac{(\lambda+a)B_a^2}{\mathcal{A}_a}\in\mathscr{M}_{0,1,\theta,r}$. This concludes the proof, because
$$ \frac{\lambda+a}{\mathcal{A}_a}=\frac{(\lambda+a)B_a^2}{\mathcal{A}_a}\cdot B_a^{-2}\in\mathscr{M}_{-2,1,\theta,r}. $$
\end{proof}
Also the term $E_k$ deserves a special treatment: it can be seen by a calculation that
\begin{equation}\label{f.alt.E}
    E_k=E_k^hi\xi^\prime\cdot\widehat{h}^\prime+\sum_{j,l=1}^{N}E_k^{H_{jl}}\widehat{H}_{jl},
\end{equation}
with $\mathcal{E}_hi\xi^\prime\cdot\widehat{h}^\prime\coloneqq\frac{\hbar}{L_2-L_1}$ and 
$$ \mathcal{E}_h= \frac{\beta}{L_2-L_1}\left[\frac{L_1}{B_a^2-L_1^2}\left(2A^2L_1-\frac{(B_a^2+A^2)(L_1^2+A^2)}{2B_a}\right)-\frac{L_2}{B_a^2-L_2^2}\left(2A^2L_2-\frac{(B_a^2+A^2)(L_2^2+A^2)}{2B_a}\right)\right]; $$
$$ E_k^h\coloneqq i\xi_k\left\{\frac{\mathcal{E}_h}{\lambda\mathcal{C}_a}\left[\frac{2AB_a}{(B_a+A)^2}+ \right.\right. $$
$$ \left.-\frac{2[(L_2B_a^2-A^2B_a)(L_1-A)+(AL_1L_2+AB_aL_1)(A-B_a)]}{(B_a^2-A^2)(B_a+L_1)(B_a+L_2)}-\frac{AB_a+L_1L_2}{(B_a+L_1)(B_a+L_2)}\right]+ $$
$$ \left.-\frac{2(B_a^2(L_1+L_2)-A^2B_a+L_1L_2B_a)}{(B_a^2-A^2)(B_a+L_1)(B_a+L_2)}+\frac{B_a}{(B_a+L_1)(B_a+L_2)}\right\}. $$
$$ E_k^{H_{NN}}\coloneqq \frac{ i\xi_k\mathcal{B}}{\lambda\mathcal{C}_a}A^2+\frac{2i\xi_kB_a^2}{\beta(B_a^2-A^2)}\quad \forall k=1,\ldots, N-1; $$
$$ E^{H_{jN}}_k\coloneqq -\frac{(B_a^2+A^2)i\xi_ki\xi_j\mathcal{B}}{B_a\lambda\mathcal{C}_a}-\frac{4i\xi_ji\xi_k B_a}{\beta(B_a^2-A^2)}+\frac{2B_a}{\beta}\delta_{jk}\quad \forall j,k=1,\ldots,N-1; $$
$$ E^{H_{jl}}_k=\frac{i\xi_ji\xi_ki\xi_l\mathcal{B}}{\lambda\mathcal{C}_a}+\frac{2i\xi_ji\xi_ki\xi_l}{\beta(B_a^2-A^2)}-\frac{2}{\beta}i\xi_j\delta_{kl}-\frac{2}{\beta}i\xi_l\delta_{jk}\quad \forall j,k,l=1,\ldots, N-1; $$
$$ \mathcal{B}\coloneqq \frac{2A^2}{(B_a+A)^2}+\frac{2[(L_2B_a^2-A^2B_a)(L_1-A)+(AL_1L_2+AB_aL_1)(A-B_a)]}{(B_a+L_1)(B_a+L_2)(B_a^2-A^2)}+ $$
$$ +\frac{AB_a+L_1L_2}{(B_a+L_1)(B_a+L_2)}. $$
Moreover, for what concerns $\mathcal{E}_h$, the internal parenthesis are the same of $\frac{\mathcal{I}_2}{L_2-L_1}$, so 
$$ \mathcal{E}_h=\frac{\beta A\{(B_a-L_2)[2B_aL_1(A-L_1)+L_1^2(B_a-L_2)+L_1^2(B_a-L_1)]-L_1L_2(B_a-A)^2-B_a^2(L_2-A)^2\}}{(B_a^2-L_1^2)(B_a^2-L_2^2)}+ $$
$$ +\frac{\beta(B_a-A)^2[-A^2B_a^2-L_1L_2A^2-B_a^2(L_1^2+L_1L_2+L_2^2)+L_1^2L_2^2]}{2B_a(B_a^2-L_1^2)(B_a^2-L_2^2)}. $$
\begin{lem}\label{l.F.M.E,B-it}
Let $a,r>0$, $\beta\neq0$ and $\theta\in\left(\theta_0,\frac{\pi}{2}\right)$ with $\tan\theta_0\ge \frac{|\beta|}{\sqrt{2}}$, then
\begin{enumerate}
    \item $\mathcal{E}_h=Am_0+m_1$ with $m_0\in\mathscr{M}_{0,1,\theta,r}$ and $m_1\in\mathscr{M}_{1,1,\theta,r}$;
    \item $\mathcal{B}\in\mathscr{M}_{0,2,\theta,r}$
\end{enumerate}
\end{lem}
The proof follows from the previous decomposition and Lemmas \ref{l.F.M.0}, \ref{l.F.M.1} and \ref{l.F.M.3}.

\subsection{Setting and technical lemmas}
\medskip

We are near to the proof of the main theorem of the resolvent part. Let us start writing the solutions in a different way: for any $k=1,\ldots, N$, we have
$$ \widehat{u}_k=A_k^0e^{-Ax_N}+A_k^1e^{-L_1x_N}+A_k^2e^{-L_2x_N}= $$
$$ = A_k^0e^{-Ax_N}+(A_k^1+A_k^2)e^{-L_1x_N}+A_k^2(L_2-L_1)\mathcal{M}(L_2,L_1,x_N)= $$
$$ = (A_k^0+A_k^1+A_k^2)e^{-Ax_N}+(A_k^1+A_k^2)(L_1-A)\mathcal{M}(L_1,A,x_N)+A_k^2(L_2-L_1)\mathcal{M}(L_2,L_1,x_N), $$
where
$$ \mathcal{M}(\gamma_1,\gamma_2,x_N)\coloneqq \frac{e^{-\gamma_1 x_N}-e^{-\gamma_2x_N}}{\gamma_1-\gamma_2}\quad \forall \gamma_1,\gamma_2\in\mathbb{C}. $$
Thanks to the boundary conditions, we have that
$$ A_k^0+A_k^1+A_k^2=\widehat{h}_k\quad \forall k=1,\ldots, N. $$
Therefore
$$ \widehat{u}_k=\widehat{h}_ke^{-Ax_N}+(\widehat{h}_k-A_k^0)(L_1-A)\mathcal{M}(L_1,A,x_N)+A_k^2(L_2-L_1)\mathcal{M}(L_2,L_1,x_N)= $$
$$ = \widehat{h}_k(e^{-Ax_N}+\mathcal{M}(L_1,A,x_N)(L_1-A))-A_k^0(L_1-A)\mathcal{M}(L_1,A,x_N)+A_k^2(L_2-L_1)\mathcal{M}(L_2,L_1,x_N)= $$
$$ =\widehat{h}_ke^{-L_1x_N}-A_k^0(L_1-A)\mathcal{M}(L_1,A,x_N)+A_k^2(L_2-L_1)\mathcal{M}(L_2,L_1,x_N). $$
Now we use the Volhevic trick (see \cite{V65}): as an example, let us take $v\colon\mathbb{R}^N_+\to\mathbb{R}$ such that
$$ v(x)=\mathcal{F}^{-1}\left[ M(\xi^\prime,\lambda)g(\xi^\prime,0)e^{-\gamma(\xi^\prime,\lambda)x_N}\right], $$
with $Re(\gamma)>0$, where $\mathcal{F}^{-1}$ is the Fourier inverse transform in the tangential components. Then we can rewrite $v$ as
$$ v(x)=\mathcal{F}^{-1}\left[\int_0^\infty  \partial_{y_N}\left(M(\xi^\prime,\lambda)g(\xi^\prime,y_N)e^{-\gamma(\xi^\prime,\lambda)(x_N+y_N)}\right)dy_N\right]. $$
What we will see in this section is that these kind of functions are strongly related with the $\mathcal{R}$-boundedness of the solution:
\begin{lem}\label{l.est.exp.}
Let $a,r>0$, $\beta\in\mathbb{R}$ and $\theta\in\left(\theta_0,\frac{\pi}{2}\right)$ with $\tan\theta_0\ge \frac{|\beta|}{\sqrt{2}}$, then for any $\ell=0,1$, $j=1,2$, $\alpha\in\mathbb{N}_0^{N-1}$ and for any $\lambda\in\Sigma_{\theta,r}$ it can be found $d>0$ such that
$$ \begin{aligned}
    & |D^\alpha_{\xi^\prime}(\tau\partial_\tau)^\ell e^{-L_jx_N}|\lesssim \left(|\lambda|^\frac{1}{2}+1+|\xi^\prime|\right)^{-|\alpha|}e^{-d\left(|\lambda|^\frac{1}{2}+1+|\xi^\prime|\right)x_N}; \\
    & |D^\alpha_{\xi^\prime}(\tau\partial_\tau)^\ell\mathcal{M}(L_1,A,x_N)|\lesssim \left(\left(|\lambda|^{-\frac{1}{2}}+1\right) \:\:\text{or}\:\:x_N\right)|\xi^\prime|^{-|\alpha|}e^{-d\left(|\lambda|^\frac{1}{2}+1+|\xi^\prime|\right)}; \\
    & |D^\alpha_{\xi^\prime}(\tau\partial_\tau)^\ell\mathcal{M}(L_2,L_1,x_N)|\lesssim \left(\left(|\lambda|^{-\frac{1}{2}}+1\right) \:\:\text{or}\:\:x_N\right)\left(|\lambda|^\frac{1}{2}+1+|\xi^\prime|\right)^{-|\alpha|}e^{-d\left(|\lambda|^\frac{1}{2}+1+|\xi^\prime|\right)}. 
\end{aligned} $$
\end{lem}
The proof of these estimates, can be gained using Lemma \ref{l.F.M.2} and following the proof of Lemma 5.3 of \cite{SS12}. 
\begin{lem}\label{l.R-bound.int.0d}
Let $a,r>0$, $\beta\in\mathbb{R}$, $\theta\in\left(\theta_0,\frac{\pi}{2}\right)$ with $\tan\theta_0\ge \frac{|\beta|}{\sqrt{2}}$, let $q\in(1,\infty)$, let $m_1\in\mathscr{M}_{0,1,\theta,r}$, $m_2\in\mathscr{M}_{0,2,\theta,r}$ and $m_3\in\mathscr{M}_{1,2,\theta,r}$, let us define the following operators in $\mathcal{L}(L^q(\mathbb{R}^N_+))$:
$$ \begin{array}{l}
    (K_1(\lambda)g)(x)\coloneqq \int_0^\infty \mathcal{F}^{-1}_{\xi^\prime}[m_1(\lambda,\xi^\prime)(|\lambda|^\frac{1}{2}+1)e^{-L_1(x_N+y_N)}\widehat{g}(\xi^\prime,y_N)](x^\prime)dy_N \\
    (K_2(\lambda)g)(x)\coloneqq \int_0^\infty \mathcal{F}^{-1}_{\xi^\prime}[m_2(\lambda,\xi^\prime)Ae^{-L_1(x_N+y_N)}\widehat{g}(\xi^\prime,y_N)](x^\prime)dy_N \\
    (K_3(\lambda)g)(x)\coloneqq \int_0^\infty \mathcal{F}^{-1}_{\xi^\prime}[m_2(\lambda,\xi^\prime)A^2\mathcal{M}(L_1,A,x_N+y_N)\widehat{g}(\xi^\prime,y_N)](x^\prime)dy_N \\
    (K_4(\lambda)g)(x)\coloneqq \int_0^\infty \mathcal{F}^{-1}_{\xi^\prime}[m_2(\lambda,\xi^\prime)A(|\lambda|^\frac{1}{2}+1)\mathcal{M}(L_1,A,x_N+y_N)\widehat{g}(\xi^\prime,y_N)](x^\prime)dy_N \\
    (K_5(\lambda)g)(x)\coloneqq \int_0^\infty \mathcal{F}^{-1}_{\xi^\prime}[m_2(\lambda,\xi^\prime)A^2\mathcal{M}(L_2,L_1,x_N+y_N)\widehat{g}(\xi^\prime,y_N)](x^\prime)dy_N \\ 
    (K_6(\lambda)g)(x)\coloneqq \int_0^\infty \mathcal{F}^{-1}_{\xi^\prime}[m_1(\lambda,\xi^\prime)(|\lambda|+1)\mathcal{M}(L_2,L_1,x_N+y_N)\widehat{g}(\xi^\prime,y_N)](x^\prime)dy_N \\
    (K_7(\lambda)g)(x)\coloneqq \int_0^\infty \mathcal{F}^{-1}_{\xi^\prime}[m_2(\lambda,\xi^\prime)(|\lambda|^\frac{1}{2}+1)A\mathcal{M}(L_2,L_1,x_N+y_N)\widehat{g}(\xi^\prime,y_N)](x^\prime)dy_N \\
    (K_8(\lambda)g)(x)\coloneqq \int_0^\infty \mathcal{F}^{-1}_{\xi^\prime}[m_3(\lambda,\xi^\prime)A\mathcal{M}(L_2,L_1,x_N+y_N)\widehat{g}(\xi^\prime,y_N)](x^\prime)dy_N.
\end{array} $$
Then, for $\ell=0,1$, for $h=1,2$ and for $j,s=1,\ldots, N$ the sets
$$ \left\{(\tau\partial_\tau)^\ell(K_h(\lambda))\mid \lambda\in\Sigma_{\theta,r}\right\} $$
are $\mathcal{R}$-bounded, where the $\mathcal{R}$-bound doesn't exceed a constant $C=C(\theta,r,N,q)$.
\end{lem}
The proof can be done following the one of Lemma 5.4 of \cite{SS12}.
\begin{lem}\label{l.R-bound.int.2d}
Let $a,r>0$, $\beta\in\mathbb{R}$, $\theta\in\left(\theta_0,\frac{\pi}{2}\right)$ with $\tan\theta_0\ge \frac{|\beta|}{\sqrt{2}}$, let $q\in(1,\infty)$, let $k_1\in\mathscr{M}_{-1,1,\theta,r}$, $k_2\in\mathscr{M}_{-2,2,\theta,r}$,
$k_3\in\mathscr{M}_{-1,2,\theta,r}$ and
$k_4\in\mathscr{M}_{0,1,\theta,r}$,
let us define the following operators in $\mathcal{L}(L^q(\mathbb{R}^N_+))$:
$$ \begin{array}{l}
    (L_1(\lambda)g)(x)\coloneqq \int_0^\infty \mathcal{F}^{-1}_{\xi^\prime}[k_1(\lambda,\xi^\prime)e^{-L_1(x_N+y_N)}\widehat{g}(\xi^\prime,y_N)](x^\prime)dy_N \\
    (L_2(\lambda)g)(x)\coloneqq \int_0^\infty \mathcal{F}^{-1}_{\xi^\prime}[k_2(\lambda,\xi^\prime)Ae^{-L_1(x_N+y_N)}\widehat{g}(\xi^\prime,y_N)](x^\prime)dy_N \\
    (L_3(\lambda)g)(x)\coloneqq \int_0^\infty \mathcal{F}^{-1}_{\xi^\prime}[k_2(\lambda,\xi^\prime)A^2\mathcal{M}(L_1,A,x_N+y_N)\widehat{g}(\xi^\prime,y_N)](x^\prime)dy_N \\
    (L_4(\lambda)g)(x)\coloneqq \int_0^\infty \mathcal{F}^{-1}_{\xi^\prime}[k_2(\lambda,\xi^\prime)A|\lambda|^\frac{1}{2}\mathcal{M}(L_1,A,x_N+y_N)\widehat{g}(\xi^\prime,y_N)](x^\prime)dy_N \\   
    (L_5(\lambda)g)(x)\coloneqq \int_0^\infty \mathcal{F}^{-1}_{\xi^\prime}[k_4(\lambda,\xi^\prime)\mathcal{M}(L_2,L_1,x_N+y_N)\widehat{g}(\xi^\prime,y_N)](x^\prime)dy_N \\
    (L_6(\lambda)g)(x)\coloneqq \int_0^\infty \mathcal{F}^{-1}_{\xi^\prime}[k_3(\lambda,\xi^\prime)A\mathcal{M}(L_2,L_1,x_N+y_N)\widehat{g}(\xi^\prime,y_N)](x^\prime)dy_N.
\end{array} $$
Then, for $\ell=0,1$, for $k=1,\ldots, 6$ and for $j,s=1,\ldots, N$ the sets
$$ \left\{(\tau\partial_\tau)^\ell(\lambda L_k(\lambda))\mid \lambda\in\Sigma_{\theta,r}\right\}, \quad  \left\{(\tau\partial_\tau)^\ell(\gamma L_k(\lambda))\mid \lambda\in\Sigma_{\theta,r}\right\}, $$
$$  \left\{(\tau\partial_\tau)^\ell(|\lambda|^\frac{1}{2} D_jL_k(\lambda))\mid \lambda\in\Sigma_{\theta,r}\right\}, \quad  \left\{(\tau\partial_\tau)^\ell(D_jD_s L_k(\lambda))\mid \lambda\in\Sigma_{\theta,r}\right\}, $$
are $\mathcal{R}$-bounded, where the $\mathcal{R}$-bound doesn't exceed a constant $C=C(\theta,r,N,q)$.
\end{lem}
The proof can be done following the proof of Lemma 5.6 of \cite{SS12} and using Lemma \ref{l.R-bound.int.0d}. It is now clear the reason beyond the Volhevic trick and the study of the Fourier Multipliers. What we will do in the proof of Theorem \ref{t.res.est.} is to verify that our solution is a combination of functions like the one of Lemma \ref{l.R-bound.int.2d}.

\subsection{Proof of Theorem \ref{t.res.est.}}
\medskip

We are ready to prove the existence of the solution by $\mathcal{R}$-boundedness. Before, we state Proposition 3.4 of \cite{DH03} and Theorem 3.3 of \cite{ES13}:
\begin{lem}\label{l.mult.R-bound.}
Let $X,Y, Z$ be three Banach spaces,
\begin{enumerate}
    \item Let $\mathcal{T}$ and $\mathcal{S}$ be two $\mathcal{R}$-bounded families of functions in $\mathcal{L}(X;Y)$, then $\mathcal{T}+\mathcal{S}$ is still an $\mathcal{R}$-bounded family of functions in $\mathcal{L}(X;Y)$ with
    $$ \mathcal{R}_{\mathcal{L}(X;Y)}(\mathcal{T}+\mathcal{S})\le \mathcal{R}_{\mathcal{L}(X;Y)}(\mathcal{T})+\mathcal{R}_{\mathcal{L}(X;Y)}(\mathcal{S}); $$
    \item Let $\mathcal{T}$ and $\mathcal{S}$ be two $\mathcal{R}$-bounded families of functions in $\mathcal{L}(X;Y)$ and $\mathcal{L}(Y;Z)$ respectively, then $\mathcal{S}\mathcal{T}$ is still an $\mathcal{R}$-bounded family of functions in $\mathcal{L}(X;Z)$ with
    $$ \mathcal{R}_{\mathcal{L}(X;Z)}(\mathcal{S}\mathcal{T})\le \mathcal{R}_{\mathcal{L}(Y;Z)}(\mathcal{S})\mathcal{R}_{\mathcal{L}(X;Y)}(\mathcal{T}). $$
\end{enumerate}
\end{lem}
\begin{lem}\label{l.F.mult.RN}
Let $m\in\mathscr{M}_{0,2,\theta,0}$ in $\mathbb{R}^N$ for any $\theta\in\left(0,\frac{\pi}{2}\right)$, then 
$$ \mathcal{R}_{\mathcal{L}(W^{q,n}(\mathbb{R}^N))}(L(\lambda))\le C(q,N,\theta) \quad n=0,1, $$
where 
$$ L(\lambda)[f]\coloneqq \mathcal{F}^{-1}\left(m(\lambda,\xi)\widehat{f}(\xi)\right). $$
\end{lem}
 We have already noticed that the case $\beta=0$ is special for our problem: when $\beta=0$, the system (\ref{res.sys.2}) can be perfectly splitted in a Stokes system for $u$ and a Laplace problem for $Q$ and it is easy to prove with the same approach the $\mathcal{R}$-boundedness results for these problems:
\begin{prop}\label{p.R-bound.Stokes}
Let $N\ge 2$, $q\in(1,\infty)$, $r>0$, $\theta\in \left(0,\frac{\pi}{2}\right)$, $\lambda\in\Sigma_{\theta,r}$ and let $h\in W^{2,q}(\mathbb{R}^N_+;\mathbb{R}^N)$ with $h_N=0$ on $\mathbb{R}^N_0$, then we can find a unique solution $(u,p)$ of
$$ \left\{\begin{array}{ll}
(\lambda-\Delta)u+\nabla p=0 & \text{in}\:\:\mathbb{R}^N_+ \\
{\rm div} u=0 & \text{in}\:\: \mathbb{R}^N_+ \\ 
u=h & \text{on}\:\:\mathbb{R}^N_0
\end{array}\right. $$
with $u\in W^{2,q}(\mathbb{R}^N_+;\mathbb{R}^N)$ and $p\in\widehat{H}^1_q(\mathbb{R}^N_+)$ and 
$$ u=\mathcal{A}(\lambda)\left[D^2h,|\lambda|^\frac{1}{2}\nabla h,|\lambda|h\right], $$
with
$$ \mathcal{A}(\lambda)\in Hol\left(\Sigma_{\theta,r};\mathcal{L}(W^{2,q}(\mathbb{R}^N_+;\mathbb{R}^N))\right) $$
and we can find $K>0$ such that
$$ \mathcal{R}_{\mathcal{L}(L^q(\mathbb{R}^N_+)^3)}\left(\left\{(\tau\partial_\tau)^\ell\mathcal{S}(\lambda)\mathcal{A}(\lambda)\:\Big|\:\tau\in\mathbb{R},\:\:\lambda=\gamma+i\tau\right\}\right)\le K\quad \ell=0,1$$
where $\mathcal{S}(\lambda)=(D^2,\lambda^\frac{1}{2}\nabla,\lambda)$.
\end{prop}
\begin{prop}\label{p.R-bound.Lap.}
Let $N\ge 2$, $q\in(1,\infty)$, $r>0$, $\theta\in \left(0,\frac{\pi}{2}\right)$, $\lambda\in\Sigma_{\theta,r}$ and let $H\in W^{2,q}(\mathbb{R}^N_+)$, then we can find a unique solution $v$ of
$$ \left\{\begin{array}{ll}
(\lambda-\Delta)v=0 & \text{in}\:\:\mathbb{R}^N_+ \\
D_Nv=H & \text{on}\:\:\mathbb{R}^N_0
\end{array}\right. $$
with $v\in W^{3,q}(\mathbb{R}^N_+)$ and 
$$ v=\mathcal{A}(\lambda)\left[D^2H,|\lambda|^\frac{1}{2}\nabla H,|\lambda|H\right], $$
with
$$ \mathcal{A}(\lambda)\in Hol\left(\Sigma_{\theta,r};\mathcal{L}(W^{2,q}(\mathbb{R}^N_+);W^{3,q}(\mathbb{R}^N_+))\right) $$
and we can find $K>0$ such that
$$ \mathcal{R}_{\mathcal{L}(L^q(\mathbb{R}^N_+)^3)}\left(\left\{(\tau\partial_\tau)^\ell\mathcal{S}(\lambda)\mathcal{A}(\lambda)\:\Big|\:\tau\in\mathbb{R},\:\:\lambda=\gamma+i\tau\right\}\right)\le K\quad \ell=0,1$$
where $\mathcal{S}(\lambda)=(D^3,\lambda^\frac{1}{2}D^2,\lambda\nabla, \lambda^\frac{3}{2})$.
\end{prop}
We are finally ready to prove the main theorem:
\begin{thm}\label{t.R-bound.res.}
Let $a,r>0$, $\beta\in\mathbb{R}$, $\theta\in\left(\theta_0,\frac{\pi}{2}\right)$ with $\tan\theta_0\ge \frac{|\beta|}{\sqrt{2}}$, let $q\in(1,\infty)$, let $h\in W^{2,q}(\mathbb{R}^N_+;\mathbb{R}^N)$ and $H\in W^{2,q}(\mathbb{R}^N_+;S_0(N,\mathbb{R}))$, then we can find a unique $(u,p,Q)$ which solves
$$ \left\{\begin{array}{ll}
        (\lambda-\Delta)u+\nabla p+\beta {\rm Div}(\Delta-a) Q=0 & \text{in}\:\:\mathbb{R}^N_+ \\
        (\lambda+a-\Delta)Q-\beta D(u)=0 & \text{in}\:\:\mathbb{R}^N_+ \\
        {\rm div}u=0 & \text{in}\:\:\mathbb{R}^N_+ \\
        u=h,\quad D_NQ=H & \text{on}\:\:\mathbb{R}^N_0,
    \end{array}\right. $$
with $u\in W^{2,q}(\mathbb{R}^N_+;\mathbb{R}^N)$, $p\in \widehat{H}^1_q(\mathbb{R}^N_+)$ and $Q\in W^{3,q}(\mathbb{R}^N_+;S_0(N,\mathbb{R}))$; moreover, let us call
$$ X\coloneqq L^q\left(\mathbb{R}^N_+;\mathbb{R}^{N^3+N^4}\right)\times L^q\left(\mathbb{R}^N_+;\mathbb{R}^{N^2+N^3}\right)\times L^q\left(\mathbb{R}^N_+;\mathbb{R}^{N+N^2}\right), $$
then
$$ (u,Q)=(\mathcal{A}(\lambda),\mathcal{B}(\lambda))\left[D^2(h,H),|\lambda|^\frac{1}{2}\nabla (h,H),|\lambda|(h,H)\right], $$
with
$$ \mathcal{A}(\lambda)\in Hol\left(\Sigma_{\theta,r};\mathcal{L}\left(X;W^{2,q}\left(\mathbb{R}^N_+;\mathbb{R}^N\right)\right)\right) $$
$$ \mathcal{B}(\lambda)\in Hol\left(\Sigma_{\theta,r};\mathcal{L}\left(X;S_0(N,\mathbb R))\right)\right), $$
and we can find $K_1,K_2>0$ such that
$$ \mathcal{R}_{\mathcal{L}(X;Y_1)}\left(\left\{(\tau\partial_\tau)^\ell\mathcal{S}_1(\lambda)\mathcal{A}(\lambda)\:\Big|\:\tau\in\mathbb{R},\:\:\lambda=\gamma+i\tau\right\}\right)\le K_1 \quad \ell=0,1$$
$$ \mathcal{R}_{\mathcal{L}(X;Y_2)}\left(\left\{(\tau\partial_\tau)^\ell\mathcal{S}_2(\lambda)\mathcal{B}(\lambda)\:\Big|\:\tau\in\mathbb{R},\:\:\lambda=\gamma+i\tau\right\}\right)\le K_2\quad \ell=0,1, $$
where $\mathcal{S}_1(\lambda)=(D^2,\lambda^\frac{1}{2}\nabla,\lambda)$, $\mathcal{S}_2(\lambda)=(D^3,\lambda^\frac{1}{2}D^2,\lambda\nabla, \lambda^\frac{3}{2})$ and
$$ Y_1=L^q\left(\mathbb{R}^N_+;\mathbb{R}^{N^3}\right)\times L^q\left(\mathbb{R}^N_+;\mathbb{R}^{N^2}\right)\times L^q\left(\mathbb{R}^N_+;\mathbb{R}^N\right) $$
$$ Y_2=L^q\left(\mathbb{R}^N_+;\mathbb{R}^{N^5}\right)\times L^q\left(\mathbb{R}^N_+;\mathbb{R}^{N^4}\right)\times L^q\left(\mathbb{R}^N_+;\mathbb{R}^{N^3}\right)\times L^q\left(\mathbb{R}^N_+;\mathbb{R}^{N^2}\right). $$
\end{thm}
\begin{proof}\hfill\\
From Propositions \ref{p.R-bound.Stokes} and \ref{p.R-bound.Lap.} we have already the result for $\beta=0$. So, from now on, we will suppose $\beta\neq0$. Let $k<N$, the we have already seen that 
$$ \widehat{u}_k=\widehat{h}_ke^{-L_1x_N}-A_k^0(L_1-A)\mathcal{M}(L_1,A,x_N)+A_k^2(L_2-L_1)\mathcal{M}(L_2,L_1,x_N). $$
It can be seen by (\ref{dep.f.A^2_k}) that
$$ A_k^2(L_2-L_1)=\frac{\left\{(B_a^2-L_1^2)E_k-L_1(B_aL_1-A^2)(\widehat{h}_k-A_k^0)\right\}(B_a^2-L_2^2)}{B_a^3(L_1+L_2)-A^2B_a^2-L_1L_2}. $$
Using the Volhevic trick (see \cite{V65})
$$ u_k=\sum_{m=1}^{14}\mathcal{U}_k^m\coloneqq -\int_0^\infty \mathcal{F}_{\xi^\prime}^{-1}\left[\partial_{y_N}\left(\widehat{h}(\xi^\prime,y_N)^\prime e^{-L_1(x_N+y_N)}\right)\right]dy_N+ $$
$$ -\int_0^\infty \mathcal{F}_{\xi^\prime}^{-1}\left[\partial_{y_N}\left(\frac{i\xi_k\mathcal{E}_h(L_1-A)}{\lambda\mathcal{C}_aA}i\xi^\prime\cdot \widehat{h}(\xi^\prime,y_N)^\prime\mathcal{M}(L_1,A,x_N+y_N)\right)\right]dy_N+ $$
$$ - \int_0^\infty \mathcal{F}_{\xi^\prime}^{-1}\left[\partial_{y_N}\left(\frac{i\xi_k}{\lambda\mathcal{C}_aA}A^2(L_1-A)\widehat{H}_{NN}(\xi^\prime,y_N)\mathcal{M}(L_1,A,x_N+y_N)\right)\right]dy_N+ $$
$$ -\sum_{j=1}^{N-1}\int_0^\infty\mathcal{F}_{\xi^\prime}^{-1}\left[\partial_{y_N}\left(\frac{i\xi_k}{\lambda\mathcal{C}_aA}\frac{(L_1-A)(B_a^2+A^2)}{B_a}i\xi_j\widehat{H}_{jN}(\xi^\prime,y_N)\mathcal{M}(L_1,A,x_N+y_N)\right)\right]dy_N+ $$
$$ -\sum_{j,l=1}^{N-1}\int_0^\infty \mathcal{F}_{\xi^\prime}^{-1}\left[\partial_{y_N}\left(\frac{i\xi_k}{\lambda\mathcal{C}_aA}(L_1-A)i\xi_ji\xi_l\widehat{H}_{jl}(\xi^\prime,y_N)\mathcal{M}(L_1,A,x_N+y_N)\right)\right]dy_N+ $$
$$ -\int_0^\infty \mathcal{F}_{\xi^\prime}^{-1}\left[\partial_{y_N}\left(\frac{(B_a^2-L_1^2)(B_a^2-L_2^2)}{\mathcal{A}_a}E_k^hi\xi^\prime\cdot \widehat{h}(\xi^\prime,y_N)\mathcal{M}(L_2,L_1,x_N+y_N)\right)\right]dy_N+ $$
$$ - \int_0^\infty \mathcal{F}_{\xi^\prime}^{-1}\left[\partial_{y_N}\left(\frac{(B_a^2-L_1^2)(B_a^2-L_2^2)}{\mathcal{A}_a}E_k^{H_{NN}}\widehat{H}_{NN}(\xi^\prime,y_N)\mathcal{M}(L_2,L_1,x_N+y_N)\right)\right]dy_N+ $$
$$ - \sum_{j=1}^{N-1}\int_0^\infty \mathcal{F}_{\xi^\prime}^{-1}\left[\partial_{y_N}\left(\frac{(B_a^2-L_1^2)(B_a^2-L_2^2)}{\mathcal{A}_a}E_k^{H_{jN}}\widehat{H}_{jN}(\xi^\prime,y_N)\mathcal{M}(L_2,L_1,x_N+y_N)\right)\right]dy_N+ $$
$$ - \sum_{j,l=1}^{N-1}\int_0^\infty \mathcal{F}_{\xi^\prime}^{-1}\left[\partial_{y_N}\left(\frac{(B_a^2-L_1^2)(B_a^2-L_2^2)}{\mathcal{A}_a}E_k^{H_{jl}}\widehat{H}_{jl}(\xi^\prime,y_N)\mathcal{M}(L_2,L_1,x_N+y_N)\right)\right]dy_N+ $$
$$ +\int_0^\infty \mathcal{F}^{-1}_{\xi^\prime}\left[\partial_{y_N}\left(\frac{L_1(B_aL_1-A^2)(B_a^2-L_2^2)}{\mathcal{A}_a}\widehat{h}_k(\xi^\prime,y_N)\mathcal{M}(L_2,L_1,x_N+y_N)\right)\right]dy_N+ $$
$$ - \int_0^\infty \mathcal{F}_{\xi^\prime}^{-1}\left[\partial_{y_N}\left(\frac{i\xi_k\mathcal{E}_hL_1(B_aL_1-A^2)(B_a^2-L_2^2)}{\lambda \mathcal{C}_aA\mathcal{A}_a}i\xi^\prime\cdot \widehat{h}(\xi^\prime,y_N)^\prime\mathcal{M}(L_2,L_1,x_N+y_N)\right)\right]dy_N+ $$
$$ - \int_0^\infty \mathcal{F}_{\xi^\prime}^{-1}\left[\partial_{y_N}\left(\frac{i\xi_k}{\lambda\mathcal{C}_aA}\frac{L_1(B_aL_1-A^2)(B_a^2-L_2^2)A^2}{\mathcal{A}_a}\widehat{H}_{NN}(\xi^\prime,y_N)\mathcal{M}(L_2,L_1,x_N+y_N)\right)\right]dy_N+ $$
$$ -\sum_{j=1}^{N-1}\int_0^\infty\mathcal{F}_{\xi^\prime}^{-1}\left[\partial_{y_N}\left(\frac{i\xi_k}{\lambda\mathcal{C}_aA}\frac{L_1(B_aL_1-A^2)(B_a^2-L_2^2)}{\mathcal{A}_a}\frac{(B_a^2+A^2)}{B_a}i\xi_j\widehat{H}_{jN}(\xi^\prime,y_N)\mathcal{M}(L_2,L_1,x_N+y_N)\right)\right]dy_N+ $$
$$ -\sum_{j,l=1}^{N-1}\int_0^\infty \mathcal{F}_{\xi^\prime}^{-1}\left[\partial_{y_N}\left(\frac{i\xi_k}{\lambda\mathcal{C}_aA}\frac{L_1(B_aL_1-A^2)(B_a^2-L_2^2)}{\mathcal{A}_a}i\xi_ji\xi_l\widehat{H}_{jl}(\xi^\prime,y_N)\mathcal{M}(L_2,L_1,x_N+y_N)\right)\right]dy_N. $$
The idea will be to see every $\mathcal{U}_k^m$ with $m=1,\ldots, 14$ as one of the operators of Lemma \ref{l.R-bound.int.2d}. The proof is almost the same for every term, so we see just few of them:
$$ \mathcal{U}^1_k=-\int_0^\infty \mathcal{F}^{-1}_{\xi^\prime}\left[\partial_{y_N}\widehat{h}_k(\xi^\prime,y_N)e^{-L_1(x_N+y_N)}-L_1\widehat{h}_ke^{-L_1(x_N+y_N)}\right]dy_N. $$
Using the identity $1=\frac{\lambda+a+A^2}{B_a^2}$, we have that 
$$ \mathcal{U}_k^1=-\int_0^\infty \mathcal{F}^{-1}_{\xi^\prime}\left[\frac{\sqrt{\lambda+a}}{B^2_a}\sqrt{\lambda+a}\partial_{y_N}\widehat{h}_k(\xi^\prime,y_N)e^{-L_1(x_N+y_N)}-\sum_{j=1}^{N-1}\frac{i\xi_j}{AB_a^2}i\xi_j\partial_{y_N}\widehat{h}_k(\xi^\prime,y_N)Ae^{-L_1(x_N+y_N)}\right]dy_N+ $$
$$ +\int_0^\infty \mathcal{F}_{\xi^\prime}^{-1}\left[\frac{L_1}{B_a^2}(\lambda+a)\widehat{h}_ke^{-L_1(x_N+y_N)}+\frac{L_1}{B_a^2}A^2\widehat{h}_ke^{-L_1(x_N+y_N)}\right]dy_N. $$
We already that 
$$ \frac{\sqrt{\lambda+a}}{B_a^2},\: \frac{L_1}{B^2_a}\in\mathscr{M}_{-1,1,\theta,r},\:\:\frac{i\xi_j}{A}\in\mathscr{M}_{0,2,\theta,r},\:\:B_a^{-2}\in\mathscr{M}_{-2,1,\theta,r}. $$
Therefore, using Lemma \ref{l.R-bound.int.2d} we get the $\mathcal{R}$-boundedness for $\mathcal{U}_k^1$. 


We will repeat a similar argument for all the terms:
$$ \mathcal{U}^2_k=-\int_0^\infty\mathcal{F}^{-1}_{\xi^\prime}\left[\frac{i\xi_k\mathcal{E}_h(L_1-A)}{\lambda\mathcal{C}_aA}i\xi^\prime\cdot \partial_{y_N}\widehat{h}(\xi^\prime,y_N)\mathcal{M}(L_1,A,x_N+y_N)\right]dy_N+ $$
$$ - \int_0^\infty\mathcal{F}^{-1}_{\xi^\prime}\left[\frac{i\xi_k\mathcal{E}_h(L_1-A)}{\lambda\mathcal{C}_aA}i\xi^\prime\cdot \widehat{h}(\xi^\prime,y_N)\left(-e^{-L_1(x_N+y_N)}-A\mathcal{M}(L_1,A,x_N+y_N)\right)\right]dy_N, $$
where we used the identity
$$ \partial_{x_N}\mathcal{M}(L_1,A,x_N)=-e^{-L_1x_N}-A\mathcal{M}(L_1,A,x_N). $$
Now we notice that 
\begin{equation}\label{proof.F.M.1} \frac{L_1-A}{\lambda}=\frac{L_1-A}{\lambda+a}\frac{\lambda+a}{\lambda}\in\mathscr{M}_{-1,1,\theta,r}. \end{equation}
Therefore:
$$ \mathcal{U}^2_k=-\sum_{j=1}^{N-1}\int_0^\infty\mathcal{F}^{-1}_{\xi^\prime}\left[\frac{i\xi_ki\xi_j\mathcal{E}_h(L_1-A)\sqrt{\lambda+a}}{\lambda\mathcal{C}_aA^2|\lambda|^\frac{1}{2}B_a^2} \sqrt{\lambda+a}\partial_{y_N}\widehat{h}_j(\xi^\prime,y_N)A|\lambda|^\frac{1}{2}\mathcal{M}(L_1,A,x_N+y_N)\right]dy_N+ $$
$$ - \int_0^\infty\mathcal{F}^{-1}_{\xi^\prime}\left[\frac{i\xi_k\mathcal{E}_h(L_1-A)}{\lambda\mathcal{C}_aAB_a^2} i\xi^\prime\cdot\partial_{y_N}\widehat{h}^\prime(\xi^\prime,y_N)A^2\mathcal{M}(L_1,A,x_N+y_N)\right]dy_N+ $$
$$ + \sum_{j=1}^{N-1}\int_0^\infty\mathcal{F}^{-1}_{\xi^\prime}\left[\frac{i\xi_ki\xi_j\mathcal{E}_h(L_1-A)}{\lambda\mathcal{C}_aA^2B_a^2}[\lambda+a+A^2]\widehat{h}_j(\xi^\prime,y_N)\left(Ae^{-L_1(x_N+y_N)}+A^2\mathcal{M}(L_1,A,x_N+y_N)\right)\right]dy_N. $$
Using (\ref{proof.F.M.1}) with Lemmas \ref{l.F.M.1}, \ref{l.F.M.C-it.} and \ref{l.F.M.E,B-it}, we have that
$$ \frac{i\xi_k}{A},\mathcal{C}_a^{-1},\frac{\mathcal{E}_h(L_1-A)}{\lambda},\frac{\sqrt{\lambda+a}}{|\lambda|^\frac{1}{2}}\in\mathscr{M}_{0,2,\theta,r}, \quad B_a^{-2}\in\mathscr{M}_{-2,1,\theta,r}. $$
Therefore we can apply Lemma \ref{l.R-bound.int.2d}. 
From the definition of $E_k^h$ and Lemma \ref{l.F.M.2}
$$ E_k^h=\frac{i\xi_k\mathcal{E}_h}{\lambda\mathcal{C}_a}m_0(\lambda,\xi^\prime)+i\xi_k\frac{m_1(\lambda,\xi^\prime)}{B_a^2-A^2}, $$
with $m_0\in\mathscr{M}_{0,2,\theta,r}$ and $m_1\in\mathscr{M}_{1,2,\theta,r}$. Therefore
$$ \mathcal{U}_k^6= -\int_0^\infty\mathcal{F}_{\xi^\prime}^{-1}\left[\frac{i\xi_k(\lambda+a)(B_a^2-L_2^2)(B_a^2-L_1^2)\mathcal{E}_hm_0(\lambda,\xi^\prime)}{A\mathcal{A}_a\lambda\mathcal{C}_a(\lambda+a)}i\xi^\prime\cdot \partial_{y_N}\widehat{h}^\prime(\xi^\prime,y_N)A\mathcal{M}(L_2,L_1,x_N+y_N)\right]dy_N+ $$
$$ - \int_0^\infty\mathcal{F}_{\xi^\prime}^{-1}\left[\frac{i\xi_k(\lambda+a)(B_a^2-L_2^2)(B_a^2-L_1^2)m_1(\lambda,\xi^\prime)}{A\mathcal{A}_a(B_a^2-A^2)(\lambda+a)}i\xi^\prime\cdot \partial_{y_N}\widehat{h}^\prime(\xi^\prime,y_N)A\mathcal{M}(L_2,L_1,x_N+y_N)\right]dy_N+ $$
$$ +\int_0^\infty\mathcal{F}_{\xi^\prime}^{-1}\left[\frac{i\xi_k(\lambda+a)^\frac{3}{2}(B_a^2-L_2^2)(B_a^2-L_1^2)\mathcal{E}_hm_0(\lambda,\xi^\prime)}{B_a^2\mathcal{A}_a\lambda\mathcal{C}_a(\lambda+a)A}\sqrt{\lambda+a}i\xi^\prime\cdot \widehat{h}^\prime(\xi^\prime,y_N)Ae^{-L_2(x_N+y_N)}\right]dy_N+ $$
$$ -\sum_{j=1}^{N-1}\int_0^\infty\mathcal{F}_{\xi^\prime}^{-1}\left[\frac{i\xi_k(\lambda+a)(B_a^2-L_2^2)(B_a^2-L_1^2)\mathcal{E}_hm_0(\lambda,\xi^\prime)i\xi_j}{B_a^2\mathcal{A}_a\lambda\mathcal{C}_a(\lambda+a)A}i\xi_ji\xi^\prime\cdot \widehat{h}^\prime(\xi^\prime,y_N)Ae^{-L_2(x_N+y_N)}\right]dy_N+ $$
$$ +\int_0^\infty\mathcal{F}_{\xi^\prime}^{-1}\left[\frac{i\xi_k(\lambda+a)^\frac{3}{2}(B_a^2-L_2^2)(B_a^2-L_1^2)m_1(\lambda,\xi^\prime)}{B_a^2\mathcal{A}_a(B_a^2-A^2)(\lambda+a)A}\sqrt{\lambda+a}i\xi^\prime\cdot \widehat{h}^\prime(\xi^\prime,y_N)Ae^{-L_2(x_N+y_N)}\right]dy_N+ $$
$$ -\sum_{j=1}^{N-1}\int_0^\infty\mathcal{F}_{\xi^\prime}^{-1}\left[\frac{i\xi_k(\lambda+a)(B_a^2-L_2^2)(B_a^2-L_1^2)m_1(\lambda,\xi^\prime)i\xi_j}{B_a^2\mathcal{A}_a(B_a^2-A^2)(\lambda+a)A}i\xi_ji\xi^\prime\cdot \widehat{h}^\prime(\xi^\prime,y_N)Ae^{-L_2(x_N+y_N)}\right]dy_N+ $$
$$ +\int_0^\infty\mathcal{F}_{\xi^\prime}^{-1}\left[\frac{i\xi_k(\lambda+a)^\frac{3}{2}(B_a^2-L_2^2)(B_a^2-L_1^2)\mathcal{E}_hL_1m_0}{B_a^2\mathcal{A}_a\lambda\mathcal{C}_a(\lambda+a)A|\lambda|^\frac{1}{2}}\sqrt{\lambda+a}i\xi^\prime\cdot \widehat{h}^\prime(\xi^\prime,y_N)|\lambda|^\frac{1}{2}A\mathcal{M}(L_2,L_1,x_N+y_N)\right]dy_N+ $$
$$ +\int_0^\infty\mathcal{F}_{\xi^\prime}^{-1}\left[\frac{(\lambda+a)(B_a^2-L_2^2)(B_a^2-L_1^2)\mathcal{E}_hL_1m_0}{B_a^2\mathcal{A}_a\lambda\mathcal{C}_a(\lambda+a)}i\xi_ki\xi^\prime\cdot \widehat{h}^\prime(\xi^\prime,y_N)A^2\mathcal{M}(L_2,L_1,x_N+y_N)\right]dy_N+ $$
$$ +\int_0^\infty\mathcal{F}_{\xi^\prime}^{-1}\left[\frac{i\xi_k(\lambda+a)^\frac{3}{2}(B_a^2-L_2^2)(B_a^2-L_1^2)L_1m_1}{B_a^2\mathcal{A}_a(B_a^2-A^2)(\lambda+a)A|\lambda|^\frac{1}{2}}\sqrt{\lambda+a}i\xi^\prime\cdot \widehat{h}^\prime(\xi^\prime,y_N)|\lambda|^\frac{1}{2}A\mathcal{M}(L_2,L_1,x_N+y_N)\right]dy_N+ $$
$$ +\int_0^\infty\mathcal{F}_{\xi^\prime}^{-1}\left[\frac{(\lambda+a)(B_a^2-L_2^2)(B_a^2-L_1^2)L_1m_1}{B_a^2\mathcal{A}_a(B_a^2-A^2)(\lambda+a)}i\xi_ki\xi^\prime\cdot \widehat{h}^\prime(\xi^\prime,y_N)A^2\mathcal{M}(L_2,L_1,x_N+y_N)\right]dy_N+ $$
Thanks to Lemmas \ref{l.F.M.A-it} and \ref{l.F.M.E,B-it} we have that
$$ \frac{i\xi_k}{A},\frac{B_a^2-L_1^2}{\lambda},\frac{B_a^2-L_2^2}{\lambda+a},\mathcal{C}_a^{-1},m_0,\frac{(\lambda+a)L_1\mathcal{E}_h}{\mathcal{A}_a},\frac{(\lambda+a)i\xi_j\mathcal{E}_h}{\mathcal{A}_a},\frac{(\lambda+a)L_1m_1}{\mathcal{A}_a},\frac{(\lambda+a)m_1i\xi_j}{\mathcal{A}_a}\in\mathscr{M}_{0,2,\theta,r}, $$
$$ \frac{\sqrt{\lambda+a}}{|\lambda|^\frac{1}{2}},\frac{(\lambda+a)^\frac{3}{2}\mathcal{E}_h}{\mathcal{A}_a},\frac{(\lambda+a)^\frac{3}{2}m_1}{\mathcal{A}_a},\frac{(\lambda+a)^\frac{3}{2}i\xi_j}{\mathcal{A}_a},\frac{(\lambda+a)^\frac{3}{2}L_1}{\mathcal{A}_a}, \in\mathscr{M}_{0,2,\theta,r}, $$
$$ \frac{(\lambda+a)\mathcal{E}_h}{\mathcal{A}_a},\frac{(\lambda+a)m_1}{\mathcal{A}_a}\in\mathscr{M}_{-1,2,\theta,r}\:\: B_a^{-2}\in\mathscr{M}_{-2,1,\theta,r}  .  $$
Then we can apply Lemma \ref{l.R-bound.int.2d}. We can divide $\mathcal{U}_k^8=\mathcal{U}_k^{8,1}+\mathcal{U}_k^{8,2}+\mathcal{U}_k^{8,3}$. The terms $\mathcal{U}_k^{8,1}$ and $\mathcal{U}_k^{8,2}$ can be treated as before. For what concerns $\mathcal{U}_k^{8,3}$
$$ \mathcal{U}_k^{8,3}=-\int_0^\infty \mathcal{F}^{-1}_{\xi^\prime}\left[\frac{2(\lambda+a)^\frac{3}{2}(B_a^2-L_1^2)(B_a^2-L_2^2)}{\beta B_a(\lambda+a)\mathcal{A}_a}\sqrt{\lambda+a}\partial_{y_N}\widehat{H}_{kN}(\xi^\prime,y_N)\mathcal{M}(L_2,L_1,x_N+y_N)\right]dy_N+ $$
$$ +\sum_{l=1}^{N-1}\int_0^\infty \mathcal{F}^{-1}_{\xi^\prime}\left[\frac{2i\xi_l(\lambda+a)(B_a^2-L_1^2)(B_a^2-L_2^2)}{\beta AB_a(\lambda+a)\mathcal{A}_a}i\xi_l\partial_{y_N}\widehat{H}_{kN}(\xi^\prime,y_N)A\mathcal{M}(L_2,L_1,x_N+y_N)\right]dy_N+ $$
$$ +\int_0^\infty \mathcal{F}^{-1}_{\xi^\prime}\left[\frac{2(\lambda+a)(B_a^2-L_1^2)(B_a^2-L_2^2)}{\beta(\lambda+a) B_a\mathcal{A}_a}B_a^2\widehat{H}_{jN}(\xi^\prime,y_N)\left(e^{-L_2(x_N+y_N)}+L_1\mathcal{M}(L_2,L_1,x_N+y_N)\right)\right]dy_N. $$
Moreover
$$ \frac{B_a^2-L_1^2}{\lambda+a},\frac{(\lambda+a)(B_a^2-L_2^2)}{\mathcal{A}_a},\frac{L_1}{B_a}\in\mathscr{M}_{0,1,\theta,r},\:\:\frac{i\xi_l}{A}\in\mathscr{M}_{0,2,\theta,r},\:\: B_a^{-1}\in\mathscr{M}_{-1,1,\theta,r}. $$
Then we can apply Lemma \ref{l.R-bound.int.2d}.

\vspace{2mm}

 When $k=N$ then $\widehat{h}_N=0$ and
$$ \widehat{u}_N=-A_N^0(L_1-A)\mathcal{M}(L_1,A,x_N)+A_N^2(L_2-L_1)\mathcal{M}(L_2,L_1,x_N)= $$
$$ =-\frac{AC(L_1-A)}{\lambda}\mathcal{M}(L_1,A,x_N)+\frac{A(L_1-A)}{\lambda}C\mathcal{M}(L_2,L_1,x_N)+i\xi^\prime\cdot \widehat{h}^\prime\mathcal{M}(L_2,L_1,x_N) $$
and we can conclude as before. 
Let us pass now to $Q$: $Q$ solves the system 
$$ \left\{\begin{array}{ll}
    (\lambda+a-\Delta)Q=\beta D(u) & \text{in}\:\:\mathbb{R}^N_+ \\
    D_NQ=H & \text{on}\:\:\mathbb{R}^N_0. 
\end{array}\right. $$
Then using Proposition \ref{p.R-bound.Lap.} and the $\mathcal{R}$-boundedness of $u$ that we have just proved, we conclude.
\end{proof}
\begin{rem}\label{r.res.est.}
Considering the proof just done, it can be seen that $\mathcal{S}_j(\lambda)$ can be switched with $\mathcal{S}_j(\gamma)$ for $j=1,2$, where $\lambda=\gamma+i\tau$ (look at the proof of Lemma 5.6 of \cite{SS12}).
\end{rem}
Let us focus on the resolvent estimate for the system
$$ \left\{\begin{array}{ll}
        (\lambda-\Delta)u+\nabla p+\beta {\rm Div}(\Delta-a) Q=f & \text{in}\:\:\mathbb{R}^N_+ \\
        (\lambda+a-\Delta)Q-\beta D(u)=G & \text{in}\:\:\mathbb{R}^N_+ \\
        {\rm div}u=0 & \text{in}\:\:\mathbb{R}^N_+ \\
        u_{|x_N=0}=h,\quad D_NQ_{|x_N=0}=H & \text{on}\:\:\mathbb{R}^N_0.
    \end{array}\right.
    $$
The resolvent estimate for the case $f=G=0$ follows just from the $\mathcal{R}$-boundedness of the solution $(u,p,Q)$ (see $m=1$ of the definition \ref{d.R-bound}). The case with $f,G\neq 0$ is a consequence of Theorem \ref{t.R-bound.res.}:
\begin{cor}\label{c.ex.}
Let $a,r>0$, $\beta\in\mathbb{R}$, $\theta\in\left(\theta_0,\frac{\pi}{2}\right)$ with $\tan\theta_0\ge \frac{|\beta|}{\sqrt{2}}$, let $q\in(1,\infty)$, let $f\in L^q(\mathbb{R}^N_+;\mathbb{R}^N)$ and $G\in W^{1,q}(\mathbb{R}^N_+;S_0(N,\mathbb{R}))$, then for any $\lambda\in\Sigma_{\theta,r}$ we can find $u\in W^{2,q}(\mathbb{R}^N_+;\mathbb{R}^N)$, $p\in \widehat{H}^1_q(\mathbb{R}^N_+)$ and $Q\in W^{3,q}(\mathbb{R}^N_+,S_0(N,\mathbb{R}))$ solution of the system
\begin{equation}\label{sys.no-BCs}
     \left\{\begin{array}{ll}
        (\lambda-\Delta)u+\nabla p+\beta {\rm Div}(\Delta-a) Q=f & \text{in}\:\:\mathbb{R}^N_+ \\
        (\lambda+a-\Delta)Q-\beta D(u)=G & \text{in}\:\:\mathbb{R}^N_+ \\
        {\rm div}u=0 & \text{in}\:\:\mathbb{R}^N_+ \\
        u_{|x_N=0}=0,\quad D_NQ_{|x_N=0}=0 & \text{on}\:\:\mathbb{R}^N_0,
    \end{array}\right.
\end{equation}
moreover
$$ \left\|\left(|\lambda|u,|\lambda|^\frac{1}{2}\nabla u,D^2u\right)\right\|_{L^q(\mathbb{R}^N_+)}+\|\nabla p\|_{L^q(\mathbb{R}^N_+)}+\left\|\left(|\lambda|^\frac{3}{2}Q, |\lambda|\nabla Q, |\lambda|^\frac{1}{2}D^2Q, D^3Q\right)\right\|_{L^q(\mathbb{R}^N_+)}\le $$
$$ \le C(a,\beta,\theta,r,q,N)\left[ \|f\|_{L^q(\mathbb{R}^N_+)}+\left\|\left(|\lambda|^\frac{1}{2}G,\nabla G\right)\right\|_{L^q(\mathbb{R}^N_+)}\right]. $$
\end{cor}
\begin{proof}\hfill\\
From \cite{MS22} we know that (\ref{sys.no-BCs}) admits a solution in $\mathbb{R}^N$. Let 
$$ \lambda Id-\mathcal{A}\colon L^q(\mathbb{R}^N;\mathbb{R}^N)\times W^{1,q}(\mathbb{R}^N;S_0(N,\mathbb{R}))\to W^{2,q}(\mathbb{R}^N;\mathbb{R}^N)\times \widehat{H}^1_q(\mathbb{R}^N;\mathbb{R})\times W^{3,q}(\mathbb{R}^N;S_0(N,\mathbb{R})), $$
be the operator correspondent to the system (\ref{sys.no-BCs}) in $\mathbb{R}^N$. Then our solution $(u,p,Q)$ will be the summation of $(v_1,\rho_1,V_1)$ and $(v_2,\rho_2, V_2)$, where 
$$ (v_1,\rho_1,V_1)=(\lambda-\mathcal{A})^{-1}(E_v[f],E_M[G]), $$
where $E_v,E_M$ are suitable extensions from $\mathbb{R}^N_+$ to $\mathbb{R}^N$ 
and $(v_2,\rho_2,V_2)$ is the solution of
$$  \left\{\begin{array}{ll}
        (\lambda-\Delta)v_2+\nabla \rho_2+\beta {\rm Div}(\Delta-a) V_2=0 & \text{in}\:\: \mathbb{R}^N_+ \\
        (\lambda+a-\Delta)V_2-\beta D(v_2)=0 & \text{in}\:\:\mathbb{R}^N_+ \\
        {\rm div}v_2=0 & \text{in}\:\:\mathbb{R}^N_+ \\
        v_2=-{v_1}_{|x_N=0},\quad D_NV_2=-{D_NV_1}_{|x_N=0} & \text{on}\:\:\mathbb{R}^{N}_0.
    \end{array}\right. $$
It can be seen that, if choose $E_v$ and $E_M$ as it follows
$$ E_v[f_k]=\left\{\begin{array}{ll}
        E_{even}[f_k] & k=1,\ldots, N-1 \\
        E_{odd}[f_N] & k=N
    \end{array}\right. $$
    $$ E_M[G_{jk}]=\left\{\begin{array}{ll}
        E_{even}[G_{jk}] & j,k=1n\ldots, N-1 \\
        E_{odd}[G_{Nk}] & j=N,\:\:k=1,\ldots, N-1 \\
        E_{odd}[G_{jN}] & j=1,\ldots, N-1,\:k=N \\
        E_{even}[G_{NN}] & j=k=N,
    \end{array}\right. $$
    then $(v_1\cdot e_N)_{x_N=0}=0$, so we can really find $(v_2,\rho_2,V_2)$ which resolves the previous system. By construction $(u,p,Q)$ solves (\ref{sys.no-BCs}) in $\mathbb{R}^N_+$ and thanks to the resolvent estimate for $(\lambda-\mathcal{A})^{-1}$ in \cite{MS22} we conclude.
\end{proof}
 The uniqueness is a consequence of the existence result:
\begin{cor}\label{c.un.}
Let $a,r>0$, $\beta\in\mathbb{R}$ $\theta\in\left(\theta_0,\frac{\pi}{2}\right)$ with $\tan\theta_0\ge \frac{|\beta|}{\sqrt{2}}$, let $q\in(1,\infty)$, let $f\in L^q(\mathbb{R}^N_+;\mathbb{R}^N)$ and $G\in W^{1,q}(\mathbb{R}^N_+;S_0(N,\mathbb{R}))$, then for any $\lambda\in\Sigma_{\theta,r}$ there are $u\in W^{2,q}(\mathbb{R}^N_+;\mathbb{R}^N)$, $p\in \widehat{H}^1_q(\mathbb{R}^N_+)$ and $Q\in W^{3,q}(\mathbb{R}^N_+,S_0(N,\mathbb{R}))$ unique solution of (\ref{sys.no-BCs}), up to additive constants on the pressure term $p$.
\end{cor}
\begin{proof}\hfill\\
For the uniqueness we use a duality argument: we can rewrite our system in the following way:
$$  \left\{\begin{array}{ll}
        \left(\lambda-\left(1+\frac{\beta^2}{2}\right)\Delta\right)u+\nabla p+\beta \lambda {\rm Div} Q=f+\beta {\rm Div} G & \text{in}\:\:\mathbb{R}^N_+ \\
        (\lambda+a-\Delta)Q-\beta D(u)=G & \text{in}\:\:\mathbb{R}^N_+ \\
        {\rm div}u=0 & \text{in}\:\:\mathbb{R}^N_+ \\
        u_{|x_N=0}=0,\quad D_NQ_{|x_N=0}=0 & \text{on}\:\:\mathbb{R}^N_0.
    \end{array}\right. $$
It can be seen now that the adjoint system is 
$$  \left\{\begin{array}{ll}
        \left(\overline{\lambda}-\left(1+\frac{\beta^2}{2}\right)\Delta\right)u+\nabla p+\beta {\rm Div} Q=f+\beta{\rm Div}G & \text{in}\:\:\mathbb{R}^N_+ \\
        (\overline{\lambda}+a-\Delta)Q-\beta\overline{\lambda} D(u)=G & \text{in}\:\:\mathbb{R}^N_+ \\
        {\rm div}u=0 & \text{in}\:\:\mathbb{R}^N_+ \\
        u_{|x_N=0}=0,\quad D_NQ_{|x_N=0}=0 & \text{on}\:\:\mathbb{R}^N_0.
    \end{array}\right. $$
which can be rewritten as
\begin{equation}\label{sys.adjoint}
    \left\{\begin{array}{ll}
        \left(\overline{\lambda}-\Delta\right)u+\nabla p+\frac{\beta}{\overline{\lambda}} {\rm Div} (\Delta-a) Q=f & \text{in}\:\:\mathbb{R}^N_+ \\
        (\overline{\lambda}+a-\Delta)Q-\beta\overline{\lambda} D(u)=\overline{\lambda}G & \text{in}\:\:\mathbb{R}^N_+ \\
        {\rm div}u=0 & \text{in}\:\:\mathbb{R}^N_+ \\
        u_{|x_N=0}=0,\quad D_NQ_{|x_N=0}=0 & \text{on}\:\:\mathbb{R}^N_0.
    \end{array}\right.
\end{equation}
Now we notice that (\ref{sys.adjoint}) admits a solution: the previous system is equivalent to 
$$     \left\{\begin{array}{ll}
        \left(\overline{\lambda}-\Delta\right)u+\nabla p+\beta {\rm Div} (\Delta-a)\left(\frac{Q}{\overline{\lambda}}\right)=f & \text{in}\:\:\mathbb{R}^N_+ \\
        (\overline{\lambda}+a-\Delta)\left(\frac{Q}{\overline{\lambda}}\right)-\beta D(u)=G & \text{in}\:\:\mathbb{R}^N_+ \\
        {\rm div}u=0 & \text{in}\:\:\mathbb{R}^N_+ \\
        u_{|x_N=0}=0,\quad D_NQ_{|x_N=0}=0 & \text{on}\:\:\mathbb{R}^N_0.
    \end{array}\right. $$
So, $\left(u,p,Q/\overline{\lambda}\right)$ is the solution of (\ref{sys.no-BCs}) with $(f,G)$. Since (\ref{sys.adjoint}) admits a solution, we can conclude with the uniqueness: let $(u,p,Q)$ be a solution of (\ref{sys.no-BCs}) with $f=G\equiv0$, let $\varphi\in C^\infty_c(\mathbb{R}^N_+;\mathbb{R}^N)$ and $\phi\in C^\infty_c(\mathbb{R}^N;S_0(N,\mathbb{R}^N))$ and let $(v,\rho,V)$ be the solution of (\ref{sys.adjoint}) with $f=\varphi$ and $G=\phi$. Therefore,
$$ \int_{\mathbb{R}^N_+}u\cdot \overline{\varphi} dx+\int_{\mathbb{R}^N_+}Q\colon \overline{\phi} dx= \int_{\mathbb{R}^N_+}u\cdot \overline{(\overline{\lambda}-(1+\beta^2/2)\Delta)v+\nabla \rho+\beta {\rm Div} V}dx+ $$
$$ + \int_{\mathbb{R}^N_+}Q\colon \overline{(\overline{\lambda}+a-\Delta)V-\beta\overline{\lambda}D(v)}dx. $$
Thanks to the boundary conditions on $u$ and $Q$ and thanks to ${\rm div} u=0$ we get
$$ = \int_{\mathbb{R}^N_+}(\lambda-(1+\beta^2/2)\Delta)u\cdot \overline{v}dx-\beta\int_{\mathbb{R}^N_+}D(u)\colon \overline{V}dx+ $$
$$ + \int_{\mathbb{R}^N_+}(\lambda+a-\Delta)Q\colon \overline{V}dx+\beta\lambda \int_{\mathbb{R}^N_+} {\rm Div}Q\cdot \overline{v}dx. $$
Using the system solved by $(u,p,Q)$ and the property ${\rm div} v=0$ we get
$$ = -\int_{\mathbb{R}^N_+}\beta \lambda {\rm Div} Q\cdot \overline{v}dx-\beta\int_{\mathbb{R}^N_+} D(u)\colon \overline{V}dx+\beta\int_{\mathbb{R}^N_+}D(u)\colon\overline{V}dx+\beta\lambda\int_{\mathbb{R}^N_+}{\rm Div} Q\cdot \overline{v}dx=0. $$
So, $u,Q=0$ a.e. and from the first equation of (\ref{sys.no-BCs}) it follows that $\nabla p=0$.
\end{proof}
Finally, thanks to Theorem \ref{t.R-bound.res.} and Corollaries \ref{c.ex.} and \ref{c.un.}, we get Theorem \ref{t.res.est.}: we have already proved the estimate for $u$ and $Q$. For what concerns $p$, it is sufficient to see that 
$$ \nabla p=(\Delta-\lambda)u+\beta {\rm Div}(a-\Delta)Q=(\Delta-\lambda)u-\beta\lambda {\rm Div} Q+\frac{\beta^2}{2}\Delta u\quad \text{in}\:\:\mathbb{R}^N_+, $$
therefore 
$$ \|\nabla p\|_{L^q(\mathbb{R}^N_+)}\lesssim |\lambda|\|u\|_{L^q(\mathbb{R}^N_+)}+\|D^2u\|_{L^q(\mathbb{R}^N_+)}+|\lambda|\|\nabla Q\|_{L^q(\mathbb{R}^N_+)}. $$

\section{$L^p$-$L^q$ maximal regularity}
\subsection{Semigroup Setting}

Let us pass to the evolution problem. We focus on the linear one
\begin{equation}\label{lin.evo.sys.+}
    \left\{\begin{array}{ll}
    \partial_tu-\Delta u+\nabla p+\beta {\rm Div}(\Delta-a)Q=f & \text{in}\:\:\mathbb{R}_+\times \mathbb{R}^N_+ \\
    \partial_t Q-(\Delta-a)Q-\beta D(u)=G & \text{in}\:\:\mathbb{R}_+\times \mathbb{R}^N_+ \\
    {\rm div} u=0 & \text{in}\:\:\mathbb{R}_+\times \mathbb{R}^N_+ \\
    u=0, \quad D_NQ=0 & \text{on}\:\:\mathbb{R}_+\times \mathbb{R}^N_0 \\
    u(0)=u_0,\quad Q(0)=Q_0 & \text{in}\:\:\mathbb{R}^N_+,
\end{array}\right.
\end{equation}
The first standard step for the study of a linear evolution system, is to prove the existence of the semigroup corresponding to the operator of the system. The linear resolvent system 
\begin{equation}\label{res.sys.3}
    \left\{\begin{array}{ll}
        (\lambda-\Delta)u+\nabla p+\beta {\rm Div}(\Delta-a) Q=f & \text{in}\:\:\mathbb{R}^N_+ \\
        (\lambda+a-\Delta)Q-\beta D(u)=G & \text{in}\:\:\mathbb{R}^N_+ \\
        {\rm div}u=0 & \text{in}\:\:\mathbb{R}^N_+ \\
        u=0,\quad D_NQ=0 & \text{on}\:\:\mathbb{R}^N_0.
    \end{array}\right.
\end{equation}
is not written in the semigroup setting: we need to express $p$ as a bounded linear operator with respect to $u$ and $Q$ and we also need to erase the divergence-free condition. From (\ref{res.sys.3}) we have that
$$     \left<\nabla p,\nabla \varphi\right>=\left<-\lambda u+\Delta u-\beta {\rm Div}(\Delta-aId)Q+f,\nabla \varphi\right> \quad \forall \varphi\in \widehat{H}^1_{q^\prime}(\mathbb{R}^N_+). $$
Thanks to divergence-free condition on $u$, we have then
$$ \left<\nabla p,\nabla \varphi\right>=\left<\Delta u-\beta {\rm Div}(\Delta-aId)Q+f,\nabla \varphi\right> \quad \forall \varphi\in\widehat{H}^1_{q^\prime}(\mathbb{R}^N_+). $$
Then we can use Theorem 2 and 3 of \cite{GN18}:
\begin{thm}\label{t.H.P.}\hfill\\
Let $q\in (1,+\infty)$, let $N\ge 2$, then for any $f\in L^q(\mathbb{R}^N_+;\mathbb{R}^N)$ there is a unique $\pi\in \widehat{H}^1_q(\mathbb{R}^N_+)$ which satisfies the Neumann weak problem
$$ \left<\nabla \pi,\nabla \varphi\right>=\left<f,\nabla \varphi\right>\quad \forall \varphi\in\widehat{H}^1_{q^\prime}\left(\mathbb{R}^N_+\right), $$
with
$$ \|\nabla \pi\|_{L^q(\mathbb{R}^N_+)}\lesssim \|f\|_{L^q(\mathbb{R}^N_+)}. $$
\end{thm}
As a consequence, we get a representation formula for $\nabla p$:
\begin{cor}\label{c.red.sys.}
Let $q\in(1,+\infty)$, $N\ge 2$, $a,r>0$, $\beta\in\mathbb{R}$, $\theta\in\left(\theta_0,\frac{\pi}{2}\right)$ with $\tan\theta_0\ge \frac{|\beta|}{\sqrt{2}}$, $\lambda\in\Sigma_{\theta,r}$ and let 
$$ (u,p,Q)\in W^{2,q}\left(\mathbb{R}^N_+;\mathbb{R}^N\right)\times \widehat{H}^{1}_{q}\left(\mathbb{R}^N_+\right)\times W^{3,q}\left(\mathbb{R}^N_+;S_0(N,\mathbb R)\right), $$
be the solution of (\ref{res.sys.3}) with $f\in L^q(\mathbb{R}^N_+;\mathbb{R}^N)$ and $G\in W^{1,q}(\mathbb{R}^N_+;S_0(N,\mathbb{R}))$, then $\nabla p=\nabla K_A(u,Q) + \nabla K_e(f)$, where $\nabla K_A(u,Q)$ is the only solution in $\widehat{H}^1_q(\mathbb{R}^N_+)$ of
$$  \left<\nabla K_A(u,Q),\nabla \varphi\right>=\left<\Delta u-\beta {\rm Div}(\Delta-aId)Q,\nabla \varphi\right>\quad \forall \varphi\in \widehat{H}^1_{q^\prime}\left(\mathbb{R}^N_+\right), $$
and $\nabla K_e(f)$ is the only solution in $\widehat{H}^1_q(\mathbb{R}^N_+)$ of
$$  \left<\nabla K_e(f),\nabla \varphi\right>=\left<f,\nabla \varphi\right>\quad \forall \varphi\in \widehat{H}^1_{q^\prime}\left(\mathbb{R}^N_+\right). $$
\end{cor}
\begin{proof}\hfill\\
By hypothesis, we know that there is $\widetilde{p}\in\widehat{H}^1_q(\mathbb{R}^N_+)$ which resolves (\ref{res.sys.3}) with $f-\nabla K_e(f)$ in the place of $f$. If we multiply the first equation of (\ref{res.sys.3}) by $\nabla \varphi$ with $\varphi\in \widehat{H}^1_{q^\prime}(\mathbb{R}^N_+)$, we get
$$ \left<(\lambda-\Delta)u+\nabla \widetilde{p}+\beta {\rm Div}(\Delta-aId)Q,\nabla \varphi\right>=\left<f-\nabla K_e(f),\nabla \varphi\right>=0. $$
Therefore
$$ \left<\nabla \widetilde{p},\nabla \varphi\right>=\left<\Delta u-\beta {\rm Div}(\Delta-aId)Q,\nabla \varphi\right>. $$
So, thanks to Theorem \ref{t.H.P.}, $\nabla \widetilde{p}=\nabla K_A(u,Q)$.
On the other hand, we also know that the solution of (\ref{res.sys.3}) is unique, so $\nabla p=\nabla \widetilde{p}+\nabla K_e(f)$.
\end{proof}
 We can now consider the reduced system
\begin{equation}\label{red.res.sys.}
    \left\{\begin{array}{ll}
        (\lambda-\Delta)u+\nabla K_A(u,Q)+\beta {\rm Div}(\Delta-a) Q=f-\nabla K_e(f) & \text{in}\:\:\mathbb{R}^N_+ \\
        (\lambda+a-\Delta)Q-\beta D(u)=G & \text{in}\:\:\mathbb{R}^N_+ \\
        u=0,\quad D_NQ=0 & \text{on}\:\:\mathbb{R}^N_0,
    \end{array}\right.
\end{equation}
This system is written in the semigroup setting, but before we need to prove that it is equivalent to the system \eqref{lin.evo.sys.+}:
\begin{cor}
Let $q\in(1,+\infty)$, $N\ge 2$, $a,r>0$, $\beta\in\mathbb{R}$, $\theta\in\left(\theta_0,\frac{\pi}{2}\right)$ with $\tan\theta_0\ge \frac{|\beta|}{\sqrt{2}}$, $\lambda\in\Sigma_{\theta,r}$ and $f\in L^q(\mathbb{R}^N_+;\mathbb{R}^N)$ and $G\in W^{1,q}(\mathbb{R}^N_+;S_0(N,\mathbb{R}))$, then $(u,Q)$ solves (\ref{red.res.sys.}) if and only if $(u,p,Q)$ solves (\ref{res.sys.3}) with $\nabla p=\nabla K_A(u,Q)+\nabla K_e(f)$. In particular, the solution for (\ref{red.res.sys.}) is unique.
\end{cor}
\begin{proof}\hfill\\
We have already seen that the solution of (\ref{res.sys.3}) is a solution for (\ref{red.res.sys.}). On the other hand, we only need to check that, if $(u,Q)$ is a solution for (\ref{red.res.sys.}), then $u$ is divergence-free: let us multiply the first equation (\ref{red.res.sys.}) with $\nabla \varphi$ with $\varphi\in \widehat{H}^1_{q^\prime}(\mathbb{R}^N_+)$
$$ \left<(\lambda-\Delta)u+\nabla K_A(u,Q)+\beta {\rm Div}(\Delta-aId)Q,\nabla \varphi\right>=\left<f-K_e(f),\nabla \varphi\right>. $$
By definition of $K_A(u,Q)$ and $K_e(f)$ we have then
$$ \lambda \left<u,\nabla \varphi\right>=0\quad \forall \varphi\in \widehat{H}^1_{q^\prime}(\mathbb{R}^N_+). $$
Which implies (thanks to the Dirichlet boundary conditions on $u$) that ${\rm div} u=0$.
\end{proof}
 We are now ready to introduce the semigroup: 
\begin{prop}\label{p.est.semigroup}
Let us call $\mathcal{A}$ the operator which defines the resolvent system (\ref{red.res.sys.}) for $a>0$ and $\beta\in\mathbb{R}$, let $q\in(1,\infty)$ and let us define the space
$$ X_q\coloneqq J_q\left(\mathbb{R}^N_+\right)\times W^{1,q}\left(\mathbb{R}^N_+;S_0(N,\mathbb{R})\right), $$
then $(\mathcal{A},D(\mathcal{A}))$ is the generator of an analytic semigroup $\{T(t)\}_{t\ge0}$ on $X_q$, where
$$ D(\mathcal{A})=\mathcal{D}_1(\mathcal{A})\times\mathcal{D}_2(\mathcal{A}), $$
with 
$$ \mathcal{D}_1(\mathcal{A})=J_q\left(\mathbb{R}^N_+\right)\cap \left\{v\in W^{2,q}\left(\mathbb{R}^N_+;\mathbb{R}^N\right)\mid v=0\:\:\mathbb{R}^N_0\right\} $$
$$ \mathcal{D}_2(\mathcal{A})=\left\{V\in W^{3,q}\left(\mathbb{R}^N_+;S_0(N,\mathbb{R})\right)\mid D_NV=0\:\:\mathbb{R}^N_0\right\}. $$
Moreover, there is $\gamma_0>0$ such that for any $t>0$ it holds
\begin{equation}\label{est.sem.}
    \begin{array}{ll}
        \|T(t)(u_0,Q_0)\|_{X_q}\le C(q,N,\gamma_0)e^{\gamma_0t}\|(u_0,Q_0)\|_{X_q} & (u_0,Q_0)\in X_q \\
        \|\partial_tT(t)(u_0,Q_0)\|_{X_q}\le C(q,N,\gamma_0) t^{-1}e^{\gamma_0t}\|(u_0,Q_0)\|_{X_q} & (u_0,Q_0)\in X_q \\
        \|\partial_tT(t)(u_0,Q_0)\|_{X_q}\le C(q,N,\gamma_0) e^{\gamma_0t}\|(u_0,Q_0)\|_{D(\mathcal{A})} & (u_0,Q_0)\in D(\mathcal{A})
    \end{array}
\end{equation}
\end{prop}
The proof follows from Corollary \ref{c.red.sys.} and standard semigroup arguments.

\subsection{Linear Evolution problem}

We recall the evolution problem
\begin{equation}\label{evo.sys.full}
    \left\{\begin{array}{ll}
    \partial_tu-\Delta u+\nabla p+\beta {\rm Div}(\Delta-a)Q=f & \text{in}\:\:\mathbb{R}_+\times \mathbb{R}^N_+ \\
    \partial_t Q-(\Delta-a)Q-\beta D(u)=G & \text{in}\:\:\mathbb{R}_+\times \mathbb{R}^N_+ \\
    {\rm div} u=0 & \text{in}\:\:\mathbb{R}_+\times \mathbb{R}^N_+ \\
    u=h, \quad D_NQ=H & \text{on}\:\:\mathbb{R}_+\times \mathbb{R}^N_0 \\
    u(0)=u_0,\quad Q(0)=Q_0 & \text{in}\:\:\mathbb{R}^N_+.
\end{array}\right.
\end{equation}
The aim of this section is to find a solution for \eqref{evo.sys.full}. In order to find it, we will consider the following systems:
\begin{equation}\label{evo.sys.ext.force}
    \left\{\begin{array}{ll}
    \partial_tu_1-\Delta u_1+\nabla p_1+\beta {\rm Div}(\Delta-a)Q_1=f_1 & \text{in}\:\:\mathbb{R}_+\times \mathbb{R}^N \\
    \partial_t Q_1-(\Delta-a)Q_1-\beta D(u_1)=G_1 & \text{in}\:\:\mathbb{R}_+\times \mathbb{R}^N \\
    {\rm div} u_1=0 & \text{in}\:\:\mathbb{R}_+\times \mathbb{R}^N \\
    u_1(0)=0,\quad Q_1(0)=0 & \text{in}\:\:\mathbb{R}^N_+.
\end{array}\right.
\end{equation}
\begin{equation}\label{evo.sys.BCs}
       \left\{\begin{array}{ll}
    \partial_tu_2-\Delta u_2+\nabla p_2+\beta {\rm Div}(\Delta-a)Q_2=0 & \text{in}\:\:\mathbb{R}\times \mathbb{R}^N_+ \\
    \partial_t Q_2-(\Delta-a)Q_2-\beta D(u_2)=0 & \text{in}\:\:\mathbb{R}\times \mathbb{R}^N_+ \\
    {\rm div} u_2=0 & \text{in}\:\:\mathbb{R}\times \mathbb{R}^N_+ \\
    u_2=h_2, \quad D_nQ_2=H_2 & \text{on}\:\:\mathbb{R}\times \mathbb{R}^N_0;
\end{array}\right. 
\end{equation}
\begin{equation}\label{evo.sys.ICs}
        \left\{\begin{array}{ll}
    \partial_tu_3-\Delta u_3+\nabla p_3+\beta {\rm Div}(\Delta-a)Q_3=0 & \text{in}\:\:\mathbb{R}_+\times \mathbb{R}^N_+ \\
    \partial_t Q_3-(\Delta-a)Q_3-\beta D(u_3)=0 & \text{in}\:\:\mathbb{R}_+\times \mathbb{R}^N_+ \\
    {\rm div} u_3=0 & \text{in}\:\:\mathbb{R}_+\times \mathbb{R}^N_+ \\
    u_3=0, \quad D_NQ_3=0 & \text{on}\:\:\mathbb{R}_+\times \mathbb{R}^N_0 \\
    u_3(0)=u_3^0,\quad Q_3(0)=Q_3^0 & \text{in}\:\:\mathbb{R}^N_+.
\end{array}\right.
\end{equation}
Firstly we will find a solution for all these systems and then we will use the results to get the existence for the general linear problem (\ref{evo.sys.full}). Let us start from the second system. We need to introduce the Laplace Transformation:
$$ \mathcal{L}[f](\lambda)=\int_\mathbb{R} e^{-\lambda t}f(t)dt, \quad \mathcal{L}^{-1}[f](t)=\frac{1}{2\pi}\int_\mathbb{R} e^{\lambda t}f(\tau)d\tau,\quad \lambda=\gamma+i\tau. $$
\begin{rem}\label{r.Lapl.transf.}
It follows from the definition that 
$$ \mathcal{L}[f](\lambda)=\mathcal{F}[e^{-\gamma t}f(t)](\tau); \quad \mathcal{L}^{-1}[f](t)=e^{\gamma t}\mathcal{F}^{-1}[f](t), $$
where $\mathcal{F}$ is the 1-dimensional Fourier Transformation. Moreover
$$ \mathcal{L}[\partial_t f](\lambda)=\int_\mathbb{R} e^{-\lambda t}\partial_tf(t)dt=\lambda\int_\mathbb{R}e^{-\lambda t}f(t)dt, $$
$$ \partial_t\mathcal{L}^{-1}[f](t)=\frac{1}{2\pi}\int_\mathbb{R}\lambda e^{\lambda t}f(\lambda)d\lambda=\mathcal{L}^{-1}[\lambda f(\lambda)](t). $$
\end{rem}
Thanks to this remark, applying the Laplace Transformation to the system \eqref{evo.sys.BCs}, we get that $\mathcal{L}_\lambda[(u,p,Q)]$ solve the resolvent system \eqref{res.sys.}. We already know that such a system admits a solution, but we want to transfer the resolvent estimate we got in the previous section to the linear evolution estimate we need. In order to do so, we will use again the $\mathcal{R}$-boundedness and the following Theorem:
\begin{defn}\hfill\\
We say that a Banach space $X$ is a UMD Space, if the Hilbert Transformation $H$ is bounded on $L^p(\mathbb{R};X)$ for some $p\in(1,\infty)$, where 
$$ H(f)(t)=\frac{1}{\pi}\lim_{\varepsilon\to 0^+}\int_{|t-s|>\varepsilon}\frac{f(s)}{t-s}ds \quad t\in\mathbb{R}. $$
\end{defn}
\begin{thm}\label{t.Weis}{Weis' Theorem \cite{W01}}\\
Let $X,Y$ be two UMD spaces and $p\in(1,\infty)$, let $m\in C^1(\mathbb{R}\setminus\{0\};\mathcal{L}(X;Y))$ be such that
$$ \mathcal{R}_{\mathcal{L}(X;Y)}\left(\left\{(\tau\partial_\tau)^\ell m\mid \tau\in\mathbb{R}\setminus\{0\}\right\}\right)<k_\ell\quad \ell=0,1, $$
let $T_m\colon\mathcal{F}^{-1}\mathcal{D}(\mathbb{R};X)\to\mathcal{S}^\prime(\mathbb{R};Y)$ defined as
$$ T_m\phi\coloneqq \mathcal{F}^{-1}\left[m\mathcal{F}[\phi]\right], $$
where $\mathcal{D}$ and $\mathcal{S}$ are respectively the distributional and the tempered functions spaces, then $T_m$ can be extended as an operator from $L^p(\mathbb{R};X)$ to $L^p(\mathbb{R};Y)$ with 
$$ \|T_mf\|_{L^p(\mathbb{R};Y)}\le C(p)(k_0+k_1)\|f\|_{L^p(\mathbb{R};X)}. $$
\end{thm}

In the following, we will use these operators:
$$ \Lambda_{\gamma,k}f=\mathcal{L}^{-1}\left[|\lambda|^\frac{k}{2}\mathcal{L}[f](\lambda)\right]=e^{\gamma t}\mathcal{F}_\tau\left[|\lambda|^\frac{k}{2}\mathcal{F}^{-1}[e^{-\gamma t}]\right] \quad k\in\mathbb{N}_0. $$
By the Mikhlin theorem it can be easily seen that 
$$ \|e^{-\gamma t}\Lambda_{\gamma,k}f\|_{L^p(\mathbb{R})}\simeq \|e^{-\gamma t}f\|_{H^{k/2}_p(\mathbb{R})} \quad k\in\mathbb{N}. $$
Finally, we define
$$ H^s_{p,\gamma}(A)\coloneqq \left\{v\in L^p_{loc}(A)\mid e^{-\gamma t}v\in H^s_p(A)\right\}\quad \forall s\ge 0,\:\:p\in(1,\infty), $$
with the norm
$$ \|v\|_{H^s_{p,\gamma}(A)}\coloneqq \|e^{-\gamma t}v\|_{H^s_p(A)} $$
for any $A\subseteq\mathbb{R}$ open set. We are now ready to prove the existence and the estimate for the system \eqref{evo.sys.BCs}:
\begin{prop}\label{p.est.BCs}
Let $a,\beta,\gamma>0$ and $p,q\in(1,+\infty)$ then for any $h_2,H_2$ such that $h_2\cdot e_N=0$ in $\mathbb{R}^N_0$,  $h_2(t)=H_2(t)=0$ when $t<0$ and
$$ h_2\in \bigcap_{l=0}^2 H^{l/2}_{p,\gamma}\left(\mathbb{R};W^{2-l,q}\left(\mathbb{R}^N_+;\mathbb{R}^N\right)\right), \quad H_2\in \bigcap_{l=0}^2 H^{l/2}_{p,\gamma}\left(\mathbb{R};W^{2-l,q}\left(\mathbb{R}^N_+;S_0(N,\mathbb{R})\right)\right), $$
then we can find $(u_2,p_2,Q_2)$ solution for (\ref{evo.sys.BCs}) with $p_2(t)\in \widehat{H}^1_q(\mathbb{R}^N_+)$ for a.e. $t>0$ and 
$$ u_2\in \bigcap_{l=0}^2 H^{l/2}_{p,\gamma}\left(\mathbb{R};W^{2-l,q}\left(\mathbb{R}^N_+;\mathbb{R}^{N}\right)\right), \quad  Q_2\in \bigcap_{l=0}^3 H^{l/2}_{p,\gamma}\left(\mathbb{R};W^{3-l,q}\left(\mathbb{R}^N_+;S_0(N,\mathbb{R})\right)\right), $$
$$ \nabla_x p_2\in L^p_\gamma\left(\mathbb{R};L^q\left(\mathbb{R}^N_+;\mathbb{R}^N\right)\right), $$
such that 
$$ \sum_{l=0}^2\|u_2\|_{H^{l/2}_{p,\gamma}(\mathbb{R};W^{2-l,q}(\mathbb{R}^N_+))} + \sum_{l=0}^3 \|Q_2\|_{H^{l/2}_{p,\gamma}(\mathbb{R};W^{3-l,q}(\mathbb{R}^N_+))}+ $$
$$ +\sum_{l=0}^2\|\gamma^{l/2}u_2\|_{L^p_\gamma(\mathbb{R};W^{2-l,q}(\mathbb{R}^N_+))} + \sum_{l=0}^3 \|\gamma^{l/2}Q_2\|_{L^p_{\gamma}(\mathbb{R};W^{3-l,q}(\mathbb{R}^N_+))}+\|\gamma^\frac{1}{2}\partial_tQ_2\|_{L^p_\gamma(\mathbb{R};L^q(\mathbb{R}^N_+))}+ $$
$$ + \|\nabla_x p_2\|_{L^p_\gamma(\mathbb{R};L^q(\mathbb{R}^N_+))}\le C\sum_{l=0}^2\|(h_2,H_2)\|_{H^{l/2}_{p,\gamma}(\mathbb{R};W^{2-l,q}(\mathbb{R}^N_+))}, $$
for some $C>0$.
\end{prop}
\begin{proof}\hfill\\
If we apply the Laplace transformation to the system (\ref{evo.sys.BCs}), thanks to Remark \ref{r.Lapl.transf.}, we get that 
$$ \left\{\begin{aligned}
    & (\lambda-\Delta)\mathcal{L}[u_2](\lambda)+\nabla \mathcal{L}[p_2](\lambda)+\beta {\rm Div}(\Delta-a)\mathcal{L}[Q_2](\lambda)=0 & \text{in}\:\:\mathbb{R}^N_+ \\
    & (\lambda+a-\Delta)\mathcal{L}[Q_2](\lambda)-\beta D(\mathcal{L}[u_2](\lambda))=0 & \text{in}\:\:\mathbb{R}^N_+ \\
    & {\rm div} \mathcal{L}[u_2](\lambda)=0 & \text{in}\:\:\mathbb{R}^N_+ \\
    & \mathcal{L}[u_2](\lambda)=\mathcal{L}[h_2](\lambda),\quad   D_N\mathcal{L}[Q_2](\lambda)=\mathcal{L}[H_2](\lambda) & \text{on}\:\:\mathbb{R}^N_0.
\end{aligned}\right. $$
If we call $G\coloneqq (D^2(h_2,H_2),\Lambda_{\gamma,1}\nabla (h_2,H_2),\Lambda_{\gamma,2}(h_2,H_2))$, then we know from Theorem \ref{t.R-bound.res.} that for any $\lambda\in\mathbb{C}$ with $Re\lambda>0$, 
$$ (\mathcal{L}[u_2](\lambda),\mathcal{L}[Q_2](\lambda))=(\mathcal{A}(\lambda),\mathcal{B}(\lambda))\mathcal{L}[G](\lambda) \quad \forall \lambda\in\Sigma_{\theta,r}. $$
We define then 
$$ (u_2,Q_2)\coloneqq \mathcal{L}^{-1}[(\mathcal{A}(\lambda),\mathcal{B}(\lambda))\mathcal{L}[G](\lambda)] $$
for $Re\lambda>0$. Thanks to Remark \ref{r.Lapl.transf.}
$$ (u_2,Q_2)=e^{\gamma t}\mathcal{F}^{-1}_\tau\left[(\mathcal{A}(\lambda),\mathcal{B}(\lambda))\mathcal{F}[e^{-\gamma t}(D^2(h_2,H_2),\Lambda_{\gamma,1}\nabla (h_2,H_2),\Lambda_{\gamma,2}(h_2,H_2))]\right]. $$
Moreover, since $h_2=H_2=0$ for $t<0$, then $\mathcal{L}[G]$ is holomorphic for $Re\lambda>0$ and, by Cauchy's Integral Theorem, we get that for any fixed $\gamma_0>0$
$$ (u_2,Q_2)(\gamma_1+i\tau)=(u_2,Q_2)(\gamma_2+i\tau) \quad \forall \gamma_1,\gamma_2>\gamma_0. $$
So, using Theorem \ref{t.R-bound.res.}, we can apply Theorem \ref{t.Weis}:
$$ \|(D^2_xu_2,\Lambda_{\gamma,1}\nabla_xu_2,\Lambda_{\gamma,2} u_2)\|_{L^p_\gamma(\mathbb{R};L^q(\mathbb{R}^N_+))}+ $$
$$ + \|(D^3_xQ_2,\Lambda_{\gamma,1}D_x^2Q_2,\Lambda_{\gamma,2} \nabla Q_2,\Lambda_{\gamma,3}Q_2)\|_{L^p_\gamma(\mathbb{R};L^q(\mathbb{R}^N_+))} \le C\|G\|_{L^p_\gamma(\mathbb{R};L^q(\mathbb{R}^N_+)}, $$
with $C>0$ which depends on $\gamma_0$. Analogously, thanks to Remark \ref{r.res.est.}, we have that
    $$ \|(D^2_xu_2,\gamma^{1/2}\nabla_xu_2,\gamma u_2)\|_{L^p_\gamma(\mathbb{R};L^q(\mathbb{R}^N_+))}+ $$
$$ + \|(D^3_xQ_2,\gamma^{1/2}D_x^2Q_2,\gamma \nabla Q_2,\gamma^\frac{3}{2}Q_2,\gamma^\frac{1}{2}\partial_tQ_2)\|_{L^p_\gamma(\mathbb{R};L^q(\mathbb{R}^N_+))} \le C\|G\|_{L^p_\gamma(\mathbb{R};L^q(\mathbb{R}^N_+))}. $$
Finally, we define $p_2(x,t)\coloneqq K_A(u_2(t),Q_2(t))$, where for a.e. $t>0$
$$ \left<\nabla K_A(u_2(t),Q_2(t)),\nabla \varphi\right>=\left<\Delta u_2(t)-\beta {\rm Div}(\Delta-a)Q_2(t),\nabla \varphi\right>\quad \forall \varphi\in\widehat{H}^1_{q^\prime}(\mathbb{R}^N_+). $$
Then, by construction, $(u_2,p_2,Q_2)$ solves (\ref{evo.sys.BCs}) and 
$$ \|\nabla p_2\|_{L^p_\gamma(\mathbb{R};L^q(\mathbb{R}^N_+))}\le C(\beta)\|(\partial_t u_2,D^2_xu_2)\|_{L^p_\gamma(\mathbb{R};L^q(\mathbb{R}^N_+))}+ $$
$$ + \|\partial_t\nabla_x Q_2\|_{L^p_\gamma(\mathbb{R};L^q(\mathbb{R}^N_+))} \lesssim \|G\|_{L^p_\gamma(\mathbb{R};L^q(\mathbb{R}^N_+))}. $$
We conclude here the estimate, because, as we have noticed before, we can estimate the $L^p$-norm of $e^{-\gamma t}\Lambda_{\gamma,k}G$ with the $H^{k/2}_{p,\gamma}$-norm of $G$ and the same we can do for the norms of $u$ and $Q$.
\end{proof}
Let us pass to the first system  (\ref{evo.sys.ext.force}). Here the result is already known:
\begin{thm}\label{t.evol.sys.ext.force}(Theorem 2.1, \cite{MS22})\\
Let $a>0$,  $\beta\in\mathbb{R}$, $\gamma_0\ge 1$, $p,q\in(1,\infty)$, let 
$$ f\in L^p_\gamma\left(\mathbb{R}_+;L^q\left(\mathbb{R}^N;\mathbb{R}^N\right)\right),\quad  G\in L^p_\gamma\left(\mathbb{R}_+;W^{1,q}\left(\mathbb{R}^N;\mathbb{R}^{N^2}\right)\right), $$
$$ v_0\in B^{2(1-1/p)}_{q,p}\left(\mathbb{R}^N;\mathbb{R}^N\right)\cap J_q\left(\mathbb{R}^N\right), \quad V_0\in B^{3-2/p}_{q,p}\left(\mathbb{R}^N;\mathbb{R}^{N^2}\right), $$
then we can find a unique $(v,\rho,V)$ which solves 
\begin{equation}\label{lin.evol.sys.trace}
    \left\{\begin{array}{ll}
       (\partial_t-\Delta)v+\nabla \rho+\beta {\rm Div}\left(\Delta V-a\left(Id-\frac{Id}{N}{\rm tr}(V)\right)\right)=f  & \text{in}\:\:\mathbb{R}_+\times \mathbb{R}^N \\
       (\partial_t-\Delta)V+a\left(V-\frac{Id}{N}{\rm tr}(V)\right)-\beta D(v)=G & \text{in}\:\:\mathbb{R}_+\times\mathbb{R}^N \\
       {\rm div}v=0  & \text{in}\:\:\mathbb{R}_+\times\mathbb{R}^N \\
       v(0)=v_0,\quad V(0)=V_0 & \text{in}\:\:\mathbb{R}^N
    \end{array}\right.
\end{equation}
with $\rho(t)\in \widehat{H}^1_q(\mathbb{R}^N)$ for a.e. $t>0$ and
$$ v\in \bigcap_{l=0}^2 H^{l/2}_{p,\gamma}\left(\mathbb{R}_+;W^{2-l,q}\left(\mathbb{R}^N;\mathbb{R}^N\right)\right),\quad V\in\bigcap_{l=0}^2 H^{l/2}_{p,\gamma}\left(\mathbb{R}_+;W^{3-l,q}\left(\mathbb{R}^N;\mathbb{R}^{N^2}\right)\right), $$
such that
$$ \sum_{l=0}^2\|v\|_{H^{l/2}_{p,\gamma}(\mathbb{R}_+;W^{2-l,q}(\mathbb{R}^N))} + \sum_{l=0}^2\|V\|_{H^{l/2}_{p,\gamma}(\mathbb{R}_+;W^{3-l,q}(\mathbb{R}^N))}+\|\nabla \rho\|_{L^p_\gamma(\mathbb{R}_+;L^q(\mathbb{R}^N))}\lesssim $$
$$ \lesssim \|f\|_{L^p_\gamma(\mathbb{R}_+;L^q(\mathbb{R}^N))}+\|G\|_{L^p_\gamma(\mathbb{R}_+;W^{1,q}(\mathbb{R}^N))}+\|v_0\|_{B^{2(1-1/p)}_{q,p}(\mathbb{R}^N)}+\|V_0\|_{B^{3-2/p}_{q,p}(\mathbb{R}^N)}. $$
\end{thm}
\begin{rem}\label{rem.trace}
We notice that, if ${\rm tr}(G)=0$ and ${\rm tr}(V_0)=0$, then ${\rm tr}(V)$ solves
$$ \left\{\begin{array}{ll}
    (\partial_t-\Delta){\rm tr}(V)=0 & \text{in}\:\:\mathbb{R}_+\times\mathbb{R}^N \\
    {\rm tr}(V)(0)=0 & \text{in}\:\:\mathbb{R}^N.
\end{array}\right. $$
Therefore, ${\rm tr}(V)\equiv0$. 
\end{rem}
Finally, we pass to the third system (\ref{evo.sys.ICs}). We already know that there is a unique solution $(u_3,Q_3)=T(t)(u_0^3,Q_0^3)$, so we just need to prove that the solution satisfies the estimate wanted. Firstly we need an interpolation lemma:
\begin{lem}\label{l.est.interp.2}
Let $\Omega\subseteq\mathbb{R}^N$ a uniform $C^2$ open set, then 
$$ \|f\|_{H^{1/2}_p(\mathbb{R};W^{1,q}(\Omega))}\lesssim \|f\|_{L^p(\mathbb{R};W^{2,q}(\Omega))}+\|\partial_t f\|_{L^p(\mathbb{R};L^q(\Omega))}, $$
$$ \|g\|_{H^{1/2}_p(\mathbb{R};W^{2,q}(\Omega))}\lesssim \|g\|_{L^p(\mathbb{R};W^{3,q}(\Omega))}+\|\partial_t g\|_{L^p(\mathbb{R};W^{1,q}(\Omega))}. $$
\end{lem}
The proof of the first estimate comes from Proposition 1 of \cite{S18}. The other one can be proven similarly. We are now ready to prove the estimate for $(u_3,Q_3)$:
\begin{prop}\label{p.evol.est.semigroup}
Let $a>0$ and $\beta\in\mathbb{R}$ and $p,q\in(1,\infty)$, let $(u_3,Q_3)=T(t)(u_0^3,Q_0^3)$ be the solution of (\ref{evo.sys.ICs}), let us call
$$ \mathcal{D}_{q,p}\coloneqq (X_q,D(\mathcal{A}))_{1-\frac{1}{p},p}, $$
let $(u_0^3,Q_0^3)\in\mathcal{D}_{q,p}$, then for any $\gamma\ge 2\gamma_0$ we have that
$$ \sum_{l=0}^2 \|u_3\|_{H^{l/2}_{p,\gamma}(\mathbb{R}_+;L^q(\mathbb{R}^N_+))}+ \sum_{l=0}^2 \|Q_3\|_{H^{l/2}_{p,\gamma}(\mathbb{R}_+;W^{3-l,q}(\mathbb{R}^N_+))}+\|\nabla_xp_3\|_{L^p_\gamma(\mathbb{R};L^q(\mathbb{R}^N_+))}\lesssim \|(u_0^3,Q_0^3)\|_{\mathcal{D}_{p,q}}, $$
where $X_q$ and $\mathcal{D}(\mathcal{A})$ are defined in Propostion \ref{p.est.semigroup} and $\gamma_0$ comes from Proposition \ref{p.est.semigroup}.
\end{prop}
The proof is the same of Theorem 3.9 of \cite{SS08}. Now we state this lemma which follows from the theory of \cite{T97}:
\begin{lem}\label{l.est.interp.1}
Let $X_1,X_2$ be two Banach spaces such that $X_2$ is dense in $X_1$, then 
$$ L^p(\mathbb{R}_+;X_2)\cap W^{1,p}\left(\mathbb{R}_+;X_1\right)\subseteq C\left([0,\infty);\left(X_1,X_2\right)_{1-1/p,p}\right) $$
with 
$$ \sup_{t\in\mathbb{R}_+}\|u(t)\|_{(X_1,X_2)_{1-1/p,p}}\lesssim \|u\|_{L^p(\mathbb{R}_+;X_2)}+\|u\|_{W^{1,p}(\mathbb{R}_+;X_1)}. $$
\end{lem}
We are now ready to prove the existence and the estimate for (\ref{evo.sys.full}), i.e. Theorem \ref{t.evol.est.}:
\begin{proof}[Proof of Theorem \ref{t.evol.est.}]\hfill\\
We write the solution as 
$$ (u,p,Q)=(u_1,p_1,Q_1)+(u_2,p_2,Q_2)+(u_3,p_3,Q_3), $$
where $(u_1,p_1,Q_1)$ is the solution of (\ref{evo.sys.ext.force}) with $(f_1,G_1)=(E_{v}[f],E_{M}[G])$, where $E_v$ and $E_M$ are the extension operators defined as in the proof of Corollary \ref{c.ex.}, $(u_2,p_2,Q_2)$ is the solution of (\ref{evo.sys.BCs}) with $(h_2,H_2)=(E_0[h-u_1],E_0[H-D_NQ_1])$, where $E_0$ is the 0-extension on $t\le0$ and where $(u_3,p_3,Q_3)$ is the solution of (\ref{evo.sys.ICs}) with $(u_0^3,Q_0^3)=(u_0-u_2(0),Q_0-Q_2(0))$. We notice that:
\begin{itemize}
    \item The existence of $(u_1,p_1,Q_1)$ follows from Theorem \ref{t.evol.sys.ext.force} and Remark \ref{rem.trace};
    \item Thanks to the choice of $f_1$ and $G_1$, we have that ${\rm div}u_1=0$ for a.e. $(t,x)\in\mathbb{R}_+\times\mathbb{R}^N$.Therefore, the existence of $(u_2,p_2,Q_2)$ comes from the fact that the $N$-component of $u_1$ is equal to 0 on $\mathbb{R}^N_0$;
    \item The existence of $(u_3,p_3,Q_3)$ comes from Proposition \ref{p.evol.est.semigroup}: from Lemma \ref{l.est.interp.1} we know that $(u_2e^{-\gamma t})(0)=u_2(0)\in B_{q,p}^{2(1-1/p)}(\mathbb{R}^N_+)$ and $(Q_2e^{-\gamma t})(0)=Q_2(0)\in B_{q,p}^{3-2/p}(\mathbb{R}^N_+)$; moreover, if we call
    $$ B_{q,p,0}^{2(1-1/p)}(\mathbb{R}^N_+)\coloneqq \{v\in B_{q,p}^{2(1-1/p)}(\mathbb{R}^N_+)\mid v=0\:\:\mathbb{R}_0^N\}, $$
    $$ B_{q,p,N}^{3-2/p}(\mathbb{R}^N_+)\coloneqq \{V\in B_{q,p}^{3-2/p}(\mathbb{R}^N_+)\mid D_NV=0\:\:\mathbb{R}_0^N\}, $$
    then it is well-known (see Theorem 2.7 of \cite{G91}) that
    $$ (u_0-u_2(0),Q_0-Q_2(0))\in\left(J_q(\mathbb{R}^N_+)\cap B_{q,p,0}^{2(1-1/p)}(\mathbb{R}^N_+)\right)\times B_{q,p,N}^{3-2/p}(\mathbb{R}^N_+)\subseteq \mathcal{D}_{q,p}. $$
    
\end{itemize}
Let us call now 
$$ I_k\coloneqq \sum_{l=0}^2\|u_k\|_{H^{l/2}_{p,\gamma}(\mathbb{R}_+;W^{2-l,q}(\mathbb{R}^N_+))} + \sum_{l=0}^2 \|Q_k\|_{H^{l/2}_{p,\gamma}(\mathbb{R}_+;W^{3-l,q}(\mathbb{R}^N_+))}+ \|\nabla_x p_k\|_{L^p_\gamma(\mathbb{R}_+;L^q(\mathbb{R}^N_+))} $$
for $k=1,2,3$. From Theorem \ref{t.evol.sys.ext.force} we know that 
$$ I_1\lesssim \|f\|_{L^p_\gamma(\mathbb{R}_+;L^q(\mathbb{R}^N_+))}+\|G\|_{L^p_\gamma(\mathbb{R}_+;W^{1,q}(\mathbb{R}^N_+))}. $$
From Proposition \ref{p.est.BCs} we get that 
$$ I_2 \lesssim \sum_{l=0}^2\|(h-u_1,H-D_NQ_1)\|_{H^{l/2}_{p,\gamma}(\mathbb{R}_+;W^{2-l,q}(\mathbb{R}^N_+))}\le I_1+ \sum_{l=0}^2\|(h,H)\|_{H^{l/2}_{p,\gamma}(\mathbb{R}_+;W^{2-l,q}(\mathbb{R}^N_+))}. $$
For what concerns the third term, thanks to Proposition \ref{p.evol.est.semigroup} we have that
$$ I_3\lesssim \|(u_0-u_2(0),Q_0-Q_2(0))\|_{\mathcal{D}_{q,p}}. $$
From Lemma \ref{l.est.interp.1} we also have that
$$ \|u_2(0)\|_{B_{q,p}^{2(1-1/p)}(\mathbb{R}^N_+)} + \|Q_2(0)\|_{B^{3-2/p}_{q,p}(\mathbb{R}^N_+)}\lesssim $$
$$ \lesssim \|e^{-\gamma t}u_2\|_{L^p(\mathbb{R}_+;W^{2,q}(\mathbb{R}^N_+;\mathbb{R}^N))}+\|e^{-\gamma t}(\gamma+\partial_t) u_2\|_{L^p(\mathbb{R}_+;L^q(\mathbb{R}^N_+;\mathbb{R}^N))}+ $$
$$ +\|e^{-\gamma t}Q_2\|_{L^p(\mathbb{R}^+;W^{3,q}(\mathbb{R}^N_+;\mathbb{R}^{N^2}))}+\|e^{-\gamma t}(\gamma+\partial_t)Q_2\|_{L^p(\mathbb{R}^+;W^{1,q}(\mathbb{R}^N_+;\mathbb{R}^{N^2}))}. $$
Now we notice that
$$ \|e^{-\gamma t}u_2\|_{L^p(\mathbb{R}_+;L^q(\mathbb{R}^N_+))}\le \frac{1}{\gamma_0}\|e^{-\gamma t}\gamma u_2\|_{L^p(\mathbb{R}_+;L^q(\mathbb{R}^N_+))}\lesssim $$
$$ \lesssim \sum_{l=0}^2\|(h,H)\|_{H^{l/2}_{p,\gamma}(\mathbb{R}_+;W^{2-l,q}(\mathbb{R}^N_+))} + \|f\|_{L^p_\gamma(\mathbb{R}_+;L^q(\mathbb{R}^N_+))}+\|G\|_{L^p_\gamma(\mathbb{R}_+;W^{1,q}(\mathbb{R}^N_+))}, $$
where we have used Proposition \ref{p.est.BCs}. In the same way we can get the other estimates. Finally, we conclude by the uniqueness: since (\ref{evo.sys.full}) is a linear system, it is sufficient to see that $(u,p,Q)=0$ when $f=G=h=H=u_0=Q_0=0$. This follows by the semigroup theory:
$$ (u,Q)=T(t)(u_0,Q_0)=0\:\Rightarrow\: \nabla p=K_A(u,Q)=0. $$
\end{proof}

\section{Proof of Theorem \ref{t.loc.ex.}}

Finally, we conclude with an application of Theorem \ref{t.evol.est.}: let us consider the general Beris-Edward model:
\begin{equation}\label{BE.sys.}
    \left\{\begin{array}{ll}
       (\partial_t-\Delta)u+\nabla p+\beta {\rm Div}(\Delta-a)Q= f(u,Q)  & \text{in}\:\:(0,T)\times \mathbb{R}^N_+ \\
       (\partial_t-\Delta+a)Q-\beta D(u)=G(u,Q)  & \text{in}\:\:(0,T)\times \mathbb{R}^N_+ \\
       {\rm div} u=0 & \text{in}\:\:(0,T)\times \mathbb{R}^N_+ \\
       u=h,\quad D_NQ=H & \text{on}\:\:(0,T)\times \mathbb{R}^N_0 \\
       u(0)=u_0,\quad Q(0)=Q_0 & \text{in}\:\:\mathbb{R}^N_+,
    \end{array}\right.
\end{equation}
where
$$ f(u,Q)=-(u\cdot \nabla) u + {\rm Div}\left[2\xi \mathbb{H}\colon Q\left(Q+\frac{Id}{N}\right)-(\xi+1)\mathbb{H}Q+(1-\xi)Q\mathbb{H}-\nabla Q\odot\nabla Q\right]-\beta {\rm Div}\mathcal{L}[\mathcal{F}(Q)];  $$
$$ G(u,Q)=-(u\cdot\nabla)Q+\xi(D(u)Q+QD(u))+W(u)Q-QW(u)-2\xi\left(Q+\frac{Id}{N}\right)Q\colon \nabla u+\mathcal{L}[\mathcal{F}(Q)],  $$
$$ W(u)=\frac{1}{2}\left(\nabla u-\nabla^Tu\right), \quad [\nabla Q\odot\nabla Q]_{jk}=\sum_{\alpha,\beta=1}^N\partial_j Q_{\alpha\beta}\partial_k Q_{\alpha\beta}\quad j,k=1,\ldots, N, $$
$$ \mathbb{H}=\Delta Q-aQ+b\mathcal{L}[Q^2]-c|Q|^2Q, \quad \mathcal{F}=bQ^2-c|Q|^2Q,  $$
where $\xi,a,b,c\in\mathbb{R}$, $\beta=\frac{2\xi}{N}$ and 
$$ \mathcal{L}[A]=A-{\rm tr}(A)\frac{Id}{N}\quad \forall A\in \mathbb{R}^{N^2}. $$
We recall the following two results from Theorem 2.1 and 2.2 of \cite{BTW75}:
\begin{thm}\label{t.Besov.emb.}\hfill\\
Let $N\in\mathbb N$, $p\in[1,\infty]$, $q,q_1\in[1,\infty]$ and $s>s_1$, then 
$$ B^s_{p,q}\left(\mathbb R^N\right)\hookrightarrow B^{s_1}_{p,q_1}\left(\mathbb R^N\right). $$
Moreover, if $m\in\mathbb N$, then
$$ B^m_{p,1}\left(\mathbb R^N\right)\hookrightarrow W^{m,p}\left(\mathbb R^N\right)\hookrightarrow B^m_{p,\infty}\left(\mathbb R^N\right). $$
\end{thm}
We are now ready to prove Theorem \ref{t.loc.ex.}:
\begin{proof}[Proof of Theorem \ref{t.loc.ex.}]\hfill\\
The strategy is to use the Contraction theorem: let us define
$$ \|(u,Q)\|_T\coloneqq \|u\|_{H^1_p((0,T);L^q(\mathbb{R}^N_+))}+\|u\|_{L^p((0,T);W^{2,q}(\mathbb{R}^N_+))} + $$
$$ + \|Q\|_{H^{1}_p((0,T);W^{1,q}(\mathbb{R}^N_+))}+\|Q\|_{L^p((0,T);W^{3,q}(\mathbb{R}^N_+))}.  $$
Let $T\in(0,1]$, then we can define also
$$ Y_1\coloneqq H^1_p((0,T);L^q(\mathbb{R}^N_+;\mathbb{R}^N))\cap L^p((0,T);W^{2,q}(\mathbb{R}^N_+;\mathbb{R}^N)); $$
$$ Y_2\coloneqq H^{1}_p((0,T);W^{1,q}(\mathbb{R}^N_+;S_0(N,\mathbb{R})))\cap L^p((0,T);W^{3,q}(\mathbb{R}^N_+;S_0(N,\mathbb{R}))). $$
Let $\omega>0$. We will apply the theorem on the space
$$ Y_\omega\coloneqq \{(v,W)\in Y_1\times Y_2\mid \|(v,W)\|_T\le \omega; \:\:v_{|t=0}=u_0,\:\:W_{|t=0}=Q_0,\}.  $$
We consider the function $\phi\colon (v,W)\in Y_\omega\mapsto (u,Q)\in Y_1\times Y_2$ which solves 
$$ \left\{\begin{array}{ll}
       (\partial_t-\Delta)u+\nabla p+\beta {\rm Div}(\Delta-a)Q= f(\mathcal{E}_1(v),\mathcal{E}_2(W))  & \text{in}\:\:(0,T)\times \mathbb{R}^N_+ \\
       (\partial_t-\Delta+a)Q-\beta D(u)=G(\mathcal{E}_1(v),\mathcal{E}_2(W))  & \text{in}\:\:(0,T)\times \mathbb{R}^N_+ \\
       {\rm div} u=0 & \text{in}\:\:(0,T)\times \mathbb{R}^N_+ \\
       u=\mathcal{E}_3(h),\quad D_NQ=\mathcal{E}_4(H) & \text{on}\:\:(0,T)\times \mathbb{R}^N_0 \\
       u(0)=u_0,\quad Q(0)=Q_0 & \text{in}\:\:\mathbb{R}^N_+,
    \end{array}\right. $$
with $p(t)=K_A(u(t),Q(t))$ and 
$$ \mathcal{E}_1(v)=E_T[v-T_1(t)(u_0,Q_0)]+\psi(t)T_1(|t|)(u_0,Q_0), $$
$$ \mathcal{E}_2(W)=E_T[W-T_2(t)(u_0,Q_0)]+\psi(t)T_2(|t|)(u_0,Q_0), $$
$$ \mathcal{E}_3(h)=E_T[h-T_1(t)(u_0,Q_0)]+\psi(t)T_1(|t|)(u_0,Q_0), $$
$$ \mathcal{E}_4(H)=E_T[H-D_NT_2(t)(u_0,Q_0)]+\psi(t)D_NT_2(|t|)(u_0,Q_0), $$
where $\psi(t)$ is a $C^\infty(\mathbb{R})$ function equal to 0 for $t<-2$ and equal to 1 for $t>-1$,   $(T_1(t),T_2(t))\coloneqq T(t)$ is the semigroup associated with the linearized problem (\ref{evo.sys.ICs}) and
$$ E_T(f)=\left\{\begin{array}{ll}
0 & t<0 \\
f(t) & t\in(0,T) \\
f(2T-t) & t\in(T,2T) \\
0 & t>2T.
\end{array}\right. $$
We have taken the extension $E_T$ of $(v,W)-T(t)(u_0,Q_0)$ and of $(h,H)-(T(t),D_NT(t))(u_0,Q_0)$ because when the function $f$ we are extending satisfies $f(0)=0$, then $E_T[f]\in H^1_p(\mathbb{R})$ with
$$ \partial_tE_T(f)=\left\{\begin{array}{ll}
0 & t<0 \\
\partial_tf(t) & t\in(0,T) \\
-\partial_tf(2T-t) & t\in(T,2T) \\
0 & t>2T.
\end{array}\right. $$
In order to find such $(u,p,Q)$ we have to check that $f(\mathcal{E}_1(v),\mathcal{E}_2(W))$, $G(\mathcal{E}_1(v),\mathcal{E}_2(W))$, $\mathcal{E}_3(h)$ and $\mathcal{E}_4(H)$ satisfies the conditions of Theorem \ref{t.evol.est.}: let us start seeing that
$$ \|e^{-\gamma t}\mathcal{E}_3(h)\|_{L^p(\mathbb{R};W^{2,q}(\mathbb{R}^N_+))} + \|e^{-\gamma t}\mathcal{E}_3(h)\|_{H^1_p(\mathbb{R};L^q(\mathbb{R}^N_+))} \lesssim $$
$$ \lesssim \|h\|_{L^p((0,T);W^{2,q}(\mathbb{R}^N_+))}+\|h\|_{H^1_p((0,T);L^q(\mathbb{R}^N_+))} + $$
$$ + \|e^{-\gamma t}T_1(t)(u_0,Q_0)\|_{L^p(\mathbb{R}_+;W^{2,q}(\mathbb{R}^N_+))}+\|e^{-\gamma t}T_1(t)(u_0,Q_0)\|_{H^1_p(\mathbb{R}_+;L^q(\mathbb{R}^N_+))}. $$
Thanks to Proposition \ref{p.evol.est.semigroup} we get that 
$$ \|e^{-\gamma t}T_1(t)(u_0,Q_0)\|_{L^p(\mathbb{R}_+;W^{2,q}(\mathbb{R}^N_+))}+\|e^{-\gamma t}T_1(t)(u_0,Q_0)\|_{H^1_p(\mathbb{R}_+;L^q(\mathbb{R}^N_+))}\lesssim $$
$$ \lesssim \|u_0\|_{B^{2(1-1/p)}_{q,p}(\mathbb{R}^N_+)}+\|Q_0\|_{B^{3-2/p}_{q,p}(\mathbb{R}^N_+)}.  $$
Finally, thanks to Lemma \ref{l.est.interp.2}, for $\ell=0,1,2$, we have that 
\begin{equation}\label{proof.est.ext.h}
    \|e^{-\gamma t}\mathcal{E}_3(h)\|_{H^{\ell/2}_{p}(\mathbb{R};W^{2-\ell,q}(\mathbb{R}^N_+))}\lesssim \|h\|_{H^{\ell/2}_p((0,T);W^{2-\ell,q}(\mathbb{R}^N_+))}+\varepsilon. 
\end{equation} 
On the other hand 
$$ \|e^{-\gamma t}\mathcal{E}_4(H)\|_{L^p(\mathbb{R};W^{2,q}(\mathbb{R}^N_+))} + \|e^{-\gamma t}\mathcal{E}_4(H)\|_{H^1_p(\mathbb{R};L^q(\mathbb{R}^N_+))} \lesssim $$
$$ \lesssim \|H\|_{L^p((0,T);W^{2,q}(\mathbb{R}^N_+))}+\|H\|_{H^1_p((0,T);L^q(\mathbb{R}^N_+))} + $$
$$ + \|e^{-\gamma t}D_NT_2(t)(u_0,Q_0)\|_{L^p(\mathbb{R}_+;W^{2,q}(\mathbb{R}^N_+))}+\|e^{-\gamma t}D_NT_2(t)(u_0,Q_0)\|_{H^1_p(\mathbb{R}_+;L^q(\mathbb{R}^N_+))}. $$
Again, thanks to Proposition \ref{p.evol.est.semigroup}, we get that 
$$ \|e^{-\gamma t}T_2(t)(u_0,Q_0)\|_{L^p(\mathbb{R}_+;W^{3,q}(\mathbb{R}^N_+))}+\|e^{-\gamma t}T_2(t)(u_0,Q_0)\|_{H^1_p(\mathbb{R}_+;W^{1,q}(\mathbb{R}^N_+))}\lesssim $$
$$ \lesssim \|u_0\|_{B^{2(1-1/p)}_{q,p}(\mathbb{R}^N_+)}+\|Q_0\|_{B^{3-2/p}_{q,p}(\mathbb{R}^N_+)}.  $$
So
\begin{equation}\label{proof.est.ext.H}
    \|e^{-\gamma t}\mathcal{E}_4(H)\|_{H^{\ell/2}_p(\mathbb{R};W^{2-\ell,q}(\mathbb{R}^N_+))}\lesssim \|H\|_{H^{\ell/2}_p((0,T);W^{2-\ell,q}(\mathbb{R}^N_+))}+\varepsilon. 
\end{equation} 
Now, we list the main non-linearities that arise of $f(u,Q)$: let $j,k,\ell=1,\ldots, N$
$$ I_1(u)= (u\cdot \nabla)u,\:\:I_2(Q)=\partial_{jk}Q\partial_\ell Q,\:\:I_3(Q)=\partial_jQQ\partial_{k\ell}Q $$
$$ I_4(Q)=\partial_jQQ\:\:I_5(Q)=\partial_jQQ^2, I_6(Q)=\partial_jQQ^3, \:\:I_7(Q)=\partial_j QQ^4 $$
$$ I_8(Q)=Q\partial_{jk\ell}Q, \:\: I_9(Q)=Q^2\partial_{jk\ell}Q; $$
Next, we list the main non-linearities of $G(u,Q)$: let $j,k,\ell =1,\ldots, N$
$$ II_1(u,Q)=\partial_jQu,\:\:II_2(u,Q)=Q\partial_ju,\:\:II_3(u,Q)=Q^2\partial_ju,\:\:II_4(Q)=Q^2,\:\:II_5(Q)=Q^3. $$
Therefore, we need to estimate
$$ \|e^{-\gamma t}I_k(\mathcal{E}_1(v),\mathcal{E}_2(W))\|_{L^p(\mathbb{R}_+;L^q(\mathbb{R}^N_+))} \quad k=1,\ldots, 8; $$
$$ \|e^{-\gamma t}II_\ell(\mathcal{E}_1(v),\mathcal{E}_2(W))\|_{L^p(\mathbb{R}_+;W^{1,q}(\mathbb{R}^N_+))}\quad \ell=1,\ldots, 5. $$
We will write in details just some of the estimates: let us start with the estimates of $I_1$:
$$ \|\mathcal{E}_1(v)\nabla \mathcal{E}_1(v)\|_{L^q(\mathbb{R}^N_+)}\le \|\mathcal{E}_1(v)\|_{L^\infty(\mathbb{R}^N_+)}\|\mathcal{E}_1(v)\|_{W^{1,q}(\mathbb{R}^N_+)}\lesssim \|\mathcal{E}_1(v)\|_{W^{1,q}(\mathbb{R}^N_+)}^2, $$
where we have used Sobolev embedding with $q>N$. Therefore
$$ \|e^ {-\gamma t}\mathcal{E}_1(v)\nabla \mathcal{E}_1(v)\|_{L^p(\mathbb{R}_+;L^q(\mathbb{R}^N_+))}\lesssim \|e^{-\gamma t}\|\mathcal{E}_1(v)\|_{W^{1,q}(\mathbb{R}^N_+)}^2\|_{L^p(\mathbb{R}_+)}=\|e^{-\frac{\gamma}{2}t}\mathcal{E}_1(v)\|_{L^{2p}(\mathbb{R}_+;W^{1,q}(\mathbb{R}^N_+))}^2. $$
By definition of $\mathcal{E}_1$ we have that 
$$ \|e^{-\frac{\gamma}{2}t}\mathcal{E}_1(v)\|_{L^{2p}(\mathbb{R}_+;W^{1,q}(\mathbb{R}^N_+))}^2\lesssim \|v\|_{L^{2p}((0,T);W^{1,q}(\mathbb{R}^N_+))}^2+\|e^{-\frac{\gamma}{2}t}T_1(t)(u_0,Q_0)\|_{L^{2p}(\mathbb{R}_+;W^{1,q}(\mathbb{R}^N_+))}^2.  $$
For what concerns the first term:
$$ \|v\|_{L^{2p}((0,T);W^{1,q}(\mathbb{R}^N_+))}^2 \le T^\frac{1}{p}\|v\|_{L^\infty((0,T);W^{1,q}(\mathbb{R}^N_+))}^2\lesssim  T^\frac{1}{p}\|e^{-\frac{\gamma}{2}t}\mathcal{E}_1(v)\|_{L^\infty(\mathbb{R}_+;W^{1,q}(\mathbb{R}^N_+))}^2\lesssim $$
$$ \lesssim T^\frac{1}{p}\left(\|e^{-\frac{\gamma}{2}t}\mathcal{E}_1(v)\|_{L^p(\mathbb{R};W^{2,q}(\mathbb{R}^N_+))}^2+\|e^{-\frac{\gamma}{2}t}\mathcal{E}_1(v)\|_{H^1_p(\mathbb{R};L^q(\mathbb{R}^N_+))}^2\right), $$
where we have used Lemma \ref{l.est.interp.1} and the fact that, since $p>2$, by Theorem \ref{t.Besov.emb.} then 
$$ \left(L^q(\mathbb{R}^N_+),W^{2,q}(\mathbb{R}^N_+)\right)_{1-1/p,p}=B^{2(1-1/p)}_{q,p}(\mathbb{R}^N_+)\hookrightarrow B^1_{q,1}(\mathbb R^N_+)\hookrightarrow W^{1,q}(\mathbb R^N). $$
On the other hand 
$$ \|e^{-\frac{\gamma}{2}t}T_1(t)(u_0,Q_0)\|_{L^{2p}(\mathbb{R}_+;W^{1,q}(\mathbb{R}^N_+))}^2 \lesssim \|e^{-\frac{\gamma}{4}t}T_1(t)(u_0,Q_0)\|_{L^\infty(\mathbb{R}_+;W^{1,q}(\mathbb{R}^N_+))}^2\lesssim \varepsilon^2,  $$
where we have used Lemma \ref{l.est.interp.1} and Proposition \ref{p.evol.est.semigroup}. Finally, we get
\begin{equation}\label{proof.I1.bound}
    \|e^{-\gamma t}I_1(\mathcal{E}_1(v))\|_{L^p(\mathbb{R};L^q(\mathbb{R}^N_+))}\lesssim  \varepsilon^2+T^\frac{1}{p}\left(\omega^2+\varepsilon^2\right).
\end{equation}
Let us pass to $I_2$: as before 
$$ \|e^{-\gamma t}\partial_{jk}\mathcal{E}_2(W)\partial_\ell\mathcal{E}_2(W)\|_{L^p(\mathbb{R}_+;L^q(\mathbb{R}^N_+))}\lesssim \|e^{-\frac{\gamma}{2}t}\mathcal{E}_2(W)\|_{L^{2p}(\mathbb{R}_+;W^{2,q}(\mathbb{R}^N_+))}^2\lesssim $$
$$ \lesssim T^{1/p}\|e^{-\frac{\gamma}{2}t}\mathcal{E}_2(W)\|_{L^\infty(\mathbb{R}_+;W^{2,q}(\mathbb{R}^N_+))}^2+\|e^{-\frac{\gamma}{4}t}T_2(t)(u_0,Q_0)\|_{L^\infty(\mathbb{R}_+;W^{2,q}(\mathbb{R}^N_+))}^2. $$
We use Lemma \ref{l.est.interp.2} noticing as before that, since $p>2$, by Theorem \ref{t.Besov.emb.} we have the embedding
$$ (W^{1,q}(\mathbb{R}^N_+),W^{3,q}(\mathbb{R}^N_+))_{1-1/p,p}=B^{3-2/p}_{q,p}(\mathbb{R}^N_+) \hookrightarrow W^{2,q}(\mathbb{R}^N_+). $$
So we get
\begin{equation}\label{proof.I2.bound}
    \|e^{-\gamma t}I_2(\mathcal{E}_2(W))\|_{L^p(\mathbb{R}_+;L^q(\mathbb{R}^N_+))}\lesssim \varepsilon^2+T^\frac{1}{p}\left(\omega^2+\varepsilon^2\right).
\end{equation}
The estimate for $I_3$ is similar: with the same strategy we get that 
$$ \|e^{-\gamma t}\partial_{j}\mathcal{E}_2(W)\mathcal{E}_2(W)\partial_{k\ell}\mathcal{E}_2(W)\|_{L^p(\mathbb{R}_+;L^q(\mathbb{R}^N_+))}\lesssim \|e^{-\frac{\gamma}{3}t}\mathcal{E}_2(W)\|_{L^{3p}(\mathbb{R}_+;W^{2,q}(\mathbb{R}^N_+))}^3\lesssim $$
$$ \lesssim T^{1/p}\|e^{-\frac{\gamma}{3}t}\mathcal{E}_2(W)\|_{L^\infty(\mathbb{R}_+;W^{2,q}(\mathbb{R}^N_+))}^3+\|e^{-\frac{\gamma}{6}t}T_2(t)(u_0,Q_0)\|_{L^\infty(\mathbb{R}_+;W^{2,q}(\mathbb{R}^N_+))}^3. $$
Then we repeat the argument for $I_2$ and $I_3$. The estimate for $I_\ell$ with $\ell=4,5,6,7$ is the same as before. Therefore, let us pass to $I_8$:
$$ \|e^{-\gamma t}\mathcal{E}_2(W)\partial_{jk\ell}\mathcal{E}_2(W)\|_{L^p(\mathbb{R}_+;L^q(\mathbb{R}^N_+))}^p\le\int_0^\infty e^{-\gamma pt}\|\mathcal{E}_2(W)\|_{W^{1,q}(\mathbb{R}^N_+)}^p\|\mathcal{E}_2(W)\|_{W^{3,q}(\mathbb{R}^N_+)}^pdt\le $$
$$ \le \|e^{-\frac{\gamma}{2}t}\mathcal{E}_2(W)\|_{L^\infty(\mathbb{R}_+;W^{1,q}(\mathbb{R}^N_+))}^p\|e^{-\frac{\gamma}{2}t}\mathcal{E}_2(W)\|_{L^p(\mathbb{R}_+;W^{3,q}(\mathbb{R}^N_+))}^p. $$
Since $p>2$, we can use Sobolev embeddings:
$$ \|e^{-\frac{\gamma}{2}t}\mathcal{E}_2(W)\|_{L^\infty(\mathbb{R}_+;W^{1,q}(\mathbb{R}^N_+))}\lesssim \|e^{-\frac{\gamma}{2}t}\mathcal{E}_2(W)\|_{W^{1/2,p}(\mathbb{R}_+;W^{1,q}(\mathbb{R}^N_+))}\lesssim 
$$
$$ \lesssim \|e^{-\frac{\gamma}{2}t}E_T[W-T_2(t)(u_0,Q_0)]\|_{W^{1/2,p}(\mathbb{R}_+;W^{1,q}(\mathbb{R}^N_+))}+\|e^{-\frac{\gamma}{2}t}T_2(t)(u_0,Q_0)\|_{W^{1/2,p}(\mathbb{R}_+;W^{1,q}(\mathbb{R}^N_+))}. $$
For what concerns the second part, we can just use the inequality
$$ \|e^{-\frac{\gamma}{2}t}T_2(t)(u_0,Q_0)\|_{W^{1/2,p}(\mathbb{R}_+;W^{1,q}(\mathbb{R}^N_+))}\le \|e^{-\frac{\gamma}{2}t}T_2(t)(u_0,Q_0)\|_{H^1_p(\mathbb{R}_+;W^{1,q}(\mathbb{R}^N_+))}\lesssim \varepsilon. $$
From the definition of $E_T$ we have that
$$ \|E_T[f]\|_{L^p(\mathbb{R}_+;W^{1,q}(\mathbb{R}^N_+))}\lesssim \|f\|_{L^p((0,T);W^{1,q}(\mathbb{R}^N_+))},\quad \|E_T[f]\|_{H^1_p(\mathbb{R}_+;W^{1,q}(\mathbb{R}^N_+))}\lesssim \|f\|_{H^1_p((0,T);W^{1,q}(\mathbb{R}^N_+))}. $$
So, since
$$ W^{1/2,p}(\mathbb{R}_+)=(L^p(\mathbb{R}_+),H^1_p(\mathbb{R}_+))_{1/2,p}, $$
By interpolation we have that
$$ \|e^{-\frac{\gamma}{2}t}E_T[W-T_2(t)(u_0,Q_0)]\|_{W^{1/2,p}(\mathbb{R}_+;W^{1,q}(\mathbb{R}^N_+))} \lesssim \|W-T_2(t)(u_0,Q_0)\|_{W^{1/2,p}((0,T);W^{1,q}(\mathbb{R}^N_+))}\le $$
$$ \le  \|W\|_{W^{1/2,p}((0,T);W^{1,q}(\mathbb{R}^N_+))}+\|T_2(t)(u_0,Q_0)\|_{W^{1/2,p}((0,T);W^{1,q}(\mathbb{R}^N_+))}. $$
In particular
$$ \|W\|_{W^{1/2,p}((0,T);W^{1,q}(\mathbb{R}^N_+))}^p= $$
$$ = \int_0^T\int_0^T\frac{\|W(t)-W(s)\|_{W^{1,q}(\mathbb{R}^N_+)}^p}{|t-s|^{p/2+1}}dtds=\int_0^T\int_0^T\frac{\|\mathcal{E}_2(W)(t)-\mathcal{E}_2(W)(s)\|_{W^{1,q}(\mathbb{R}^N_+)}^p}{|t-s|^{p/2+1}}dtds. $$
Since $p>1$, we have that $\mathcal{E}_2(W)\in C^{0,1-1/p}(\mathbb{R}_+;W^{1,q}(\mathbb{R}^N_+;\mathbb{R}^{N^2}))$, so
$$ \|W\|_{W^{1/2,p}((0,T);W^{1,q}(\mathbb{R}^N_+))}^p\lesssim \|\mathcal{E}_2(W)\|_{H^1_p(\mathbb{R}_+;W^{1,q}(\mathbb{R}^N_+))}^p\int_0^T\int_0^T|t-s|^{p/2-2}dtds.  $$
We also know that $p>2$, so $p/2-2>-1$ and 
$$ \int_0^T\int_0^T|t-s|^{p/2-2}dtds= \int_0^T\int_{-s}^{T-s}|t|^{p/2-2}dtds\simeq\int_0^T[(T-s)^{p/2-1}-(-s)^{p/2-1}]ds\simeq T^{p/2}. $$
so
$$ \|W\|_{W^{1/2,p}((0,T);W^{1,q}(\mathbb{R}^N_+))}^p\lesssim T^{p/2}\|\mathcal{E}_2(W)\|_{H^1_p(\mathbb{R}_+;W^{1,q}(\mathbb{R}^N_+))}^p. $$
Similarly, 
$$ \|T_2(t)(u_0,Q_0)\|_{W^{1/2,p}((0,T);W^{1,q}(\mathbb{R}^N_+))}\lesssim T^{1/2}\|T_2(t)(u_0,Q_0)\|_{H^1_p(\mathbb{R}_+;W^{1,q}(\mathbb{R}^N_+))}. $$
Therefore,
\begin{equation}\label{proof.I5.bound}
    \|e^{-\gamma t}I_8(\mathcal{E}_2(W))\|_{L^p(\mathbb{R}_+;L^q(\mathbb{R}^N_+))}\lesssim \varepsilon^2+T^\frac{1}{2}\left(\omega^2+\varepsilon^2\right).
\end{equation}
The estimate for $I_9$ is the same, so we pass to the $II_1$: 
$$ \|e^{-\gamma t}\partial_j\mathcal{E}_2(W)\mathcal{E}_1(v)\|_{L^p(\mathbb{R}_+;W^{1,q}(\mathbb{R}^N_+))}^p\lesssim $$
$$ \lesssim \sum_{k=1}^N\|e^{-\gamma t}\partial_{jk}\mathcal{E}_2(W)\mathcal{E}_1(v)\|_{L^p(\mathbb{R}_+;L^q(\mathbb{R}^N_+))}^p + \|e^{-\gamma t}\partial_j\mathcal{E}_2(W)\partial_k\mathcal{E}_1(v)\|_{L^p(\mathbb{R}_+;L^q(\mathbb{R}^N_+))}^p \lesssim $$
$$ \lesssim \int_0^\infty e^{-\gamma 
pt}\|\mathcal{E}_2(W)\|_{W^{2,q}(\mathbb{R}^N_+)}^p\|\mathcal{E}_1(v)\|_{W^{1,q}(\mathbb{R}^N_+)}^pdt. $$
So we can repeat the argument for $I_7$ and we get
\begin{equation}\label{proof.II1.bound}
    \|e^{-\gamma t}II_1(\mathcal{E}_1(v),\mathcal{E}_2(W))\|_{L^p(\mathbb{R}_+;W^{1,q}(\mathbb{R}^N_+))}\lesssim \varepsilon^2+T^\frac{1}{p}\left(\omega^2+\varepsilon^2\right).
\end{equation}
Let us pass to $II_2$:
$$ \|e^{-\gamma t}\mathcal{E}_2(W)\partial_j\mathcal{E}_1(v)\|_{L^p(\mathbb{R}_+;W^{1,q}(\mathbb{R}^N_+))}\le \|e^{-\gamma t}\mathcal{E}_2(W)\partial_j\mathcal{E}_1(v)\|_{L^p(\mathbb{R}_+;L^q(\mathbb{R}^N_+))} +  
 $$
 $$ + \sum_{k=1}^N\|e^{-\gamma t}\mathcal{E}_2(W)\partial_{jk}\mathcal{E}_1(v)\|_{L^p(\mathbb{R}_+;L^q(\mathbb{R}^N_+))}+\|e^{-\gamma t}\partial_k\mathcal{E}_2(W)\partial_j\mathcal{E}_1(v)\|_{L^p(\mathbb{R}_+;L^q(\mathbb{R}^N_+))}. $$
For what concerns the second term, we can repeat the argument of $II_1$. For the other one
$$ \|e^{-\gamma t}\mathcal{E}_2(W)\partial_{jk}\mathcal{E}_1(v)\|_{L^p(\mathbb{R}_+;L^q(\mathbb{R}^N_+))}\lesssim \|e^{-\frac{\gamma}{2}t}\mathcal{E}_2(W)\|_{L^\infty(\mathbb{R}_+;W^{1,q}(\mathbb{R}^N_+))}\|e^{-\frac{\gamma}{2}t}\mathcal{E}_1(v)\|_{L^p(\mathbb{R}_+;W^{2,q}(\mathbb{R}^N_+))}. $$
Then we can repeat the argument of $I_5$, so
\begin{equation}\label{proof.II2.bound}
    \|e^{-\gamma t}II_2(\mathcal{E}_1(v),\mathcal{E}_2(W))\|_{L^p(\mathbb{R}_+;W^{1,q}(\mathbb{R}^N_+))}\lesssim \varepsilon^2+T^\frac{1}{p}\left(\omega^2+ \varepsilon^2\right).
\end{equation}
Repeating this kind of arguments also with $II_3$, $II_4$ and $II_5$, if  $\varepsilon\le 1$, using (\ref{proof.est.ext.h}), (\ref{proof.est.ext.H}) and (\ref{proof.I1.bound}) to (\ref{proof.II2.bound}), we get
$$ \|\phi(v,W)\|_T\le C\left(\sum_{\ell=0}^2\|(h,H)\|_{H^{\ell/2}_p((0,T);W^{2-\ell}(\mathbb{R}^N_+))}+\varepsilon+ T^\frac{1}{p}\left(\omega^2
+ \omega^3+\omega^4+\omega^5+\varepsilon\right)\right),  $$
with $C>0$ which doesn't depend on $\omega,\varepsilon$ and $T$. So, if we take $\omega\in\mathbb{R}$ such that
$$ C\left(\sum_{\ell=0}^2\|(h,H)\|_{H^{\ell/2}_p((0,T);W^{2-\ell}(\mathbb{R}^N_+))}+\varepsilon\right)\le \frac{\omega}{2},  $$
and we choose $T>0$ sufficiently small such that it holds
$$ CT^\frac{1}{p}\left(\omega^2
+ \omega^3+\omega^4+\omega^5+\varepsilon\right)\le \frac{\omega}{2}, $$
then the function $\phi$ is well-defined and $\phi\colon Y_\omega\to Y_\omega$. Now, if we take $(v_1,W_1),(v_2,W_2)\in Y_\omega$, it can be seen as before that exists $M>0$ independent from $\omega,\varepsilon,T$ such that
$$ \|\phi(v_1,W_1)-\phi(v_2,W_2)\|_T\le M\left(\varepsilon+ T^\frac{1}{p}\left(\omega
+ \omega^2+\omega^3+\omega^4+\varepsilon\right)\right)\|(v_1,W_1)-(v_2,W_2)\|_T. $$
So, choosing $\varepsilon,T\ll1$ sufficiently small, $\phi$ is a contraction on $Y_\omega$ and therefore we have a solution. The uniqueness follows from the contraction theorem.
\end{proof}

\thanks{
\textbf{Acknowledgements.} M.M. is partially supported by JSPS Grants-in-Aid for Early-Career Scientists 21K13819 and Grants-in-Aid for Scientific Research (B) 22H01134. D.B. was partially supported by INDAM, GNAMPA group. Moreover, the authors would like to thank Dr.Yoshihiro Shibata, professor emeritus of Waseda University,  and Dr.Vladimir Georgiev, professor of Pisa University, for wonderful discussions and valuable comments. }

\clearpage
\bibliographystyle{plain}
 \bibliography{BE_ref}

\end{document}